\documentclass{article}
\usepackage{amsmath}
\usepackage{amsfonts}
\usepackage{amsthm}
\usepackage{amssymb}
\usepackage{enumerate}
\usepackage{ytableau}

\newtheorem{thm}{Theorem}[section]
\newtheorem{prop}[thm]{Proposition}
\newtheorem{lem}[thm]{Lemma}
\newtheorem{cor}[thm]{Corollary}

\theoremstyle{definition}
\newtheorem{defn}[thm]{Definition}
\newtheorem{rmk}[thm]{Remark}
\newtheorem{ex}[thm]{Example}

\begin{document}

\title{Compatibility of Higher Specht Polynomials and Decompositions of Representations}

\author{Shaul Zemel}

\maketitle


\section*{Introduction}

The quotient $R_{n}$ of the ring of polynomials in $n$ variables modulo non-constant symmetric functions is of great importance in many subjects. Indeed, it is shown in the seminal work \cite{[B]} to be the cohomology ring of the flag variety, which is the starting point of the theory of Schubert polynomials, in Algebraic Geometry and later Algebraic Combinatorics as well.

The action of the symmetric group $S_{n}$ on the polynomial ring is inherited by $R_{n}$, yielding a representation that is known to be isomorphic to the regular one of $S_{n}$. Inside it, the minimal polynomials satisfying the anti-symmetry relations with respect to column groups inside $S_{n}$ were shown (see, e.g, \cite{[Pe]}) to span an irreducible representation that is isomorphic to the classical Specht module, and are thus known as \emph{Specht polynomials}.

As the regular representation carries each Specht module with multiplicity that equals its dimension, it became natural to ask how to construct other copies of the Specht module by polynomials that project non-trivially onto $R_{n}$. This was established in \cite{[ATY]}, with their construction of the \emph{higher Specht polynomials}.

More recently, the quotients $R_{n}$ were generalized in \cite{[HRS]} to quotients $R_{n,k}$, which replace the (regular) action of $S_{n}$ on ordered $n$-tuples of different elements from the set $\mathbb{N}_{n}$ of integers between 1 and $n$ by the action on the coarser set, of ordered partitions of $\mathbb{N}_{n}$ into $k$ sets. The paper \cite{[PR]} constructs a space whose cohomology ring is $R_{n,k}$, and there are also relations to the Delta conjecture from \cite{[HRW]}. These quotients are the grossest in an even more general sequence of quotients $R_{n,k,s}$, also considered in these references. The construction of a basis of higher Specht polynomials for $R_{n,k}$ (as well as some other rings) is dealt with in \cite{[GR]}.

\medskip

This work is the first in a series of three papers, which explore these notions further, construct some new and interesting objects, and establish relations between them. We now describe these notions, indicating our goals in this paper, and hint at those of the sequels \cite{[Z1]} and \cite{[Z2]}.

Unlike the regular action, the action of $S_{n}$ on ordered partitions into $k$ sets is not transitive, as the sizes of the ordered sets is an invariant of this action. The first goal of this paper is to decompose the higher Specht basis from \cite{[GR]} into parts, each of which corresponding to the representation of a single orbit inside that action. This is achieved in Theorem \ref{main} below, with a parallel construction of homogeneous representations appearing in Proposition \ref{homrep}.

Next, recall from \cite{[HRS]} that there are relations between partitions of $n$ and partitions of $n+1$, by the operations known as star and bar insertions. In the spirit of these operations, the next goal was to do the same for the corresponding representations. For this we renormalized the higher Specht polynomials in a way that respects the embedding of $S_{n}$ into $S_{n+1}$ in an optimal way.

Also, when we consider the star and bar insertions, we focus only at doing them at the last location. Then the bar insertion lifts the induction showing up in the Branching Rule, also in the homogeneous setting (see Proposition \ref{embInd} below), and the star insertion becomes, for large enough $n$, a bijection between some representations of $S_{n}$ and some of $S_{n+1}$ (see Proposition \ref{embreps} and Corollary \ref{RnIstab} below).

To make this more explicit, recall that higher Specht polynomials carry two indices, that are standard Young tableaux of the same shape. Such pairs are known to be in bijection with the elements of $S_{n}$ via the RSK algorithm. In our normalization we do not work with RSK as it is, but rather apply the Sch\"{u}tzenberger evacuation process from, e.g., Section 3.9 of \cite{[Sa]}, to the $Q$-tableau (which is equivalent to conjugating the element of $S_{n}$ by the longest element before taking the $Q$-tableau, but using the $P$-tableau of the original element). By dividing by the scalar multiplier showing up in the original definition, we get the desired compatibility, in the sense that if $w \in S_{n}$, $F$ is the higher Specht polynomial in $n$ variables that is associated with $w$, and $\tilde{F}$ is the one in $n+1$ variables that corresponds to the image of $w$ under the embedding into $S_{n+1}$, then substituting $x_{n+1}=0$ in $\tilde{F}$ produces $F$ back again (see Proposition \ref{incnSpecht} below).

This opens the door for constructing stable higher Specht polynomials, in infinitely many variables, and their representations. They are presented in Theorem \ref{stabSpecht} below, and will be further studied in \cite{[Z2]}.

\medskip

Back to the orbits in the action of $S_{n}$ on the ordered partitions of $\mathbb{N}_{n}$ into $k$ sets, the set of sizes form a composition of $n$, which is associated, by a classical construction, with a subset of $\mathbb{N}_{n-1}$, of size $k-1$ (see Lemmas \ref{setscomp} and \ref{Snparts} below). This subset is then related to the parameters of the permutations, or the Young tableaux, participating in the corresponding sub-representation of $R_{n,k}$). The decomposition of $R_{n,k}$ in Theorem \ref{main} is into representations $R_{n,I}$ for such subsets $I$.

Some relations become more natural to state and prove for the parts $R_{n,I}$ than for $R_{n,k}$ itself (though we obtain a consequence for these representations as well, in Corollary \ref{Rnkind} below). Indeed, when moving from $n$ to $n+1$, the bar insertion adds $n$ to $I$ and $R_{n+1,I\cup\{n\}}$ is an appropriate induction of $R_{n,I}$ from $S_{n}$ to $S_{n+1}$, lifting the branching rule (see Proposition \ref{embInd} below). The star insertion keeps the set $I$, and the construction of $R_{n+1,I}$ from $R_{n,I}$, detailed in Proposition \ref{embreps} below, stabilizes, after applying it enough times, in a sense that will be made clear in Lemma \ref{Extstab} and Corollary \ref{RnIstab} below. This is the basis for defining, in \cite{[Z2]}, representations of infinite symmetric groups that involve the stable higher Specht polynomials from Theorem \ref{stabSpecht}.

In fact, considering the more general quotients $R_{n,k,s}$, we can obtain a similar construction but using multi-sets, which will be detailed, with some applications, in \cite{[Z1]}. This sequel to the current work will show how to use these notions, and generalize them, in order to decompose the representation of $S_{n}$ on homogeneous polynomials of degree $d$, for any $n$ and $d$, in two ways, each lifts a different formula for the character of this representation.

\medskip

As another point of view, we recall that \cite{[CF]} initiated the theory of representation stability, which was soon after related to the categorical construction of FI-modules in \cite{[CEF]} (see also \cite{[Fa]} for a broader survey of the subject, as well as \cite{[SS]} for a deeper investigation of the relations with category theory, and \cite{[Pu]} for the closely related notion of central stability). This means, in our context, a collection $\{U_{n}\}_{n}$, where $U_{n}$ is a representation of $S_{n}$ for each $n$ (say large enough), with a natural embedding of $U_{n}$ into $U_{n+1}$ such that the image generates $U_{n+1}$ as a representation of $S_{n+1}$ for each large enough $n$. As Remark \ref{repstab} below indicates, the representations $\{R_{n,I}\}_{n}$, for a fixed set $I$ of integers, have this property in an explicit manner (as do their homogeneous counterparts), and this explicit construction is what allows taking the limit, as will be done in \cite{[Z2]}. In \cite{[Z1]} we will prove the same property for multi-sets $I$, where the limit from \cite{[Z2]} only exists in the homogeneous setting.

\medskip

This paper is divided into 4 sections. Section \ref{ASCRSK} sets the notations and basic results for tableaux and permutations. Section \ref{HighSpecht} provides the definition of the normalized version of the higher Specht polynomials and constructs the stable ones. In Section \ref{RepDecom} we consider the representations of $S_{n}$ on higher Specht polynomials, and Section \ref{MapsRep} constructs various maps between these representations.

\medskip

I am grateful to B. Sagan, B. Rhoades, M. Gillespie, and S. van Willigenburg, for their interest, for several encouraging conversations on this subject, and for helpful comments. I am indebted, in particular, to D. Grinberg for a detailed reading of previous versions, for introducing me to several references, and for numerous suggestions which drastically improved the presentation of this paper.

\section{Ascents, Descents, Cocharges, and RSK \label{ASCRSK}}

We will use $\mathbb{N}_{n}$ for the set of integers between 1 and $n$. As usual, $S_{n}$ stands for the symmetric group on $\mathbb{N}_{n}$. We make the following definition, as in, e.g., Subsection 2.1 of \cite{[HRS]}.
\begin{defn}
Let $w$ be an element of $S_{n}$, written in one-line notation with entries $w_{i}:=w(i)$ for $1 \leq i \leq n$. Then we say that $1 \leq i \leq n-1$ is an \emph{ascending location} in $w$ if $w_{i+1}>w_{i}$, and a \emph{descending location} in $w$ when $w_{i+1}<w_{i}$. We denote by $\operatorname{Asl}(w)$ and $\operatorname{Dsl}(w)$ the sets of ascending and descending locations of $w$ respectively. \label{asperm}
\end{defn}
These sets are denoted by $\operatorname{Asc}(w)$ and $\operatorname{Des}(w)$ in \cite{[RW]}, but some references (like \cite{[GR]}) use these notations for slightly different sets, so we keep our own notation. It is clear that $\mathbb{N}_{n-1}$ is the disjoint union of the sets $\operatorname{Asl}(w)$ and $\operatorname{Dsl}(w)$ from Definition \ref{asperm} for every $w \in S_{n}$.

\medskip

We write $\lambda \vdash n$ to state that $\lambda$ is a partition of $n$, and denote the length of such $\lambda$ by $\ell(\lambda)$. We will write Ferrers diagrams (which we will, as is often done, identify with the corresponding partitions) and Young tableaux using the English notation. As usual, $\operatorname{sh}(T)$ is the shape of the standard Young tableau $T$ (as a partition), and for any $\lambda \vdash n$ we will write $\operatorname{SYT}(\lambda)$ for the set of standard Young tableaux of shape $\lambda$. Given $n$, $\lambda \vdash n$, $T\in\operatorname{SYT}(\lambda)$, and an index $i\in\mathbb{N}_{n}$, we will write $R_{T}(i)$ and $C_{T}(i)$ for the row and column respectively in which $i$ shows up (as in \cite{[CZ]}), and $v_{T}(i)$ is the box $\big(R_{T}(i),C_{T}(i)\big)$ containing $i$ in $T$ (no confusion should arise between $R_{T}(i)$ and the representations $R_{n}$ or $R_{n,k}$).
\begin{defn}
Take $\lambda \vdash n$, $T\in\operatorname{SYT}(\lambda)$, and $1 \leq i \leq n-1$. We say that $i$ is an \emph{ascending index} of $T$ in case $R_{T}(i+1) \leq R_{T}(i)$ and $C_{T}(i+1)>C_{T}(i)$, and a \emph{descending index} of $T$ when $R_{T}(i+1)>R_{T}(i)$ and $C_{T}(i+1) \leq C_{T}(i)$. We set $\operatorname{Asi}(T)$ and $\operatorname{Dsi}(T)$ to be the sets of ascending and descending indices of $T$ respectively. \label{asSYT}
\end{defn}
Note that $\operatorname{Asi}(T)$ is the set of descents of $T$ from Definition 2.1 of \cite{[GR]}, with the inversion of terminology arising from that reference using the French notation for standard Young tableaux.
\begin{rmk}
It is clear that the sets $\operatorname{Asi}(T)$ and $\operatorname{Dsi}(T)$ from Definition \ref{asSYT} are disjoint for any such $T$, and we claim that their union is $\mathbb{N}_{n-1}$. Indeed, the fact that $T$ is standard implies that one cannot have $R_{T}(i+1) \leq R_{T}(i)$ and $C_{T}(i+1) \leq C_{T}(i)$ simultaneously. Moreover, in case both $R_{T}(i+1)>R_{T}(i)$ and $C_{T}(i+1)>C_{T}(i)$ hold, the boxes $\big(R_{T}(i),C_{T}(i+1)\big)$ and $\big(R_{T}(i+1),C_{T}(i)\big)$ must contain, by the standard condition, numbers that lies strictly between $i$ and $i+1$, which is impossible. Thus for every $i\in\mathbb{N}_{n-1}$ precisely one of the conditions from Definition \ref{asSYT} has to be satisfied, as desired. \label{boxnear}
\end{rmk}

\medskip

We will denote, for $w \in S_{n}$, the two standard Young tableaux arising from the Robinson--Schensted--Knuth (RSK) algorithm as $P(w)$ and $Q(w)$, as in \cite{[Sa]}. Definitions \ref{asperm} and \ref{asSYT} are related via the RSK algorithm in the following close analogue of the Row Bumping Lemma from Section 1.1 of \cite{[Fu]} (see also Lemma 7.23.1 of \cite{[St2]}).
\begin{lem}
For any $w \in S_{n}$ we have the equalities $\operatorname{Asl}(w)=\operatorname{Asi}\big(Q(w)\big)$ and $\operatorname{Dsl}(w)=\operatorname{Dsi}\big(Q(w)\big)$. \label{RSKcor}
\end{lem}

\begin{proof}
The fact that $i$ shows up in the row $r=R_{Q(w)}(i)$ means that at the $i$th step of the RSK algorithm, the construction of the $P$-tableau had a sequence of $r-1$ boxes, in the first consecutive rows, that contain numbers, say $k_{j}$, $1 \leq j \leq r-1$, that increase with the row, and that $w_{i}$ was smaller than all of them and shoved each one to the next row until the last one remained as the rightmost box of the $r$th row. Thus, after the $i$th step $w_{i}$ shows up in the first row, and $k_{j}$ appears in row number $j+1$.

If $i\in\operatorname{Dsl}(w)$ then $w_{i+1}<w_{i}$, meaning that in the first row there is a number $\ell_{1} \leq w_{i}<k_{1}$ such that $w_{i+1}$ pushes $\ell_{1}$, to the second row. This is the first step, and we work by induction on $j$, where in the $j$th step we begin with a number $\ell_{j}<k_{j}$ in the $j$th row, and deduce that there is a number $\ell_{j+1} \leq k_{j}$ in row $j+1$ such that $\ell_{j}$ pushes $\ell_{j+1}$ to the row number $j+2$. Doing so inductively until $j=r-1$, we deduce that we reached row number $r+1$ (and perhaps even more). Hence $R_{Q(w)}(i+1)>r=R_{Q(w)}(i)$ and we get $i\in\operatorname{Dsi}\big(Q(w)\big)$.

Assuming now that $i\in\operatorname{Asl}(w)$, we find that either $w_{i+1}$ is put at the end of the first row, or it pushes some number $\ell_{1}>w_{i+1}$ to the next row. The fact that $w_{i}$ did not push $\ell_{1}$ at the $i$th step implies that $\ell_{1}>k_{1}$, and hence, by induction on $j$, either $\ell_{j}$ sits at the end of the $j$th row, or it pushes $\ell_{j+1}$, which is larger that $k_{j+1}$ it was not pushed by $k_{j}$. If we do not stop until $j=r-1$, then $\ell_{r-1}>k_{r-1}$, and then it lies at the end of the $r$th row. In any case we get $R_{Q(w)}(i+1) \leq r=R_{Q(w)}(i)$, so that $i\in\operatorname{Asi}\big(Q(w)\big)$.

Since both $\operatorname{Asl}(w)\cup\operatorname{Dsl}(w)$ and $\operatorname{Asi}\big(Q(w)\big)\cup\operatorname{Dsi}\big(Q(w)\big)$ are disjoint unions that produce $\mathbb{N}_{n-1}$, this proves the lemma.
\end{proof}
Recall that the reading word of a standard Young tableau $T$ produces an element $w \in S_{n}$ with $P(w)=T$, and Lemma \ref{RSKcor} works with $Q$. This is related to the fact that the construction of the cocharge is based on increasing the cocharge of an entry $i+1$ with respect to that of $i$ if and only if $i+1$ shows up before $i$ in the inverse permutation (see Definition 1.3 of \cite{[GR]}), so that the set $\operatorname{Dsl}(w)$ from Definition \ref{asperm} is the set of descents of $w^{-1}$ in the terminology implied from \cite{[GR]}. The fact that inverting $w$ amounts to interchanging $P$ and $Q$ by Theorem 3.6.6 of \cite{[Sa]} explains how the two conventions are related.

In addition, as it is standard in RSK to write $T=P(w)$ and $S=Q(w)$, and most of our operations will be based on $Q(w)$, we change the notation of a general standard Young tableau from $T$ to $S$ for the rest of this section.

\medskip

Recall that a \emph{composition} of $n$ is any sequence $\alpha$ of positive integers (not necessarily non-increasing) whose sum is $n$, a situation that we denote by $\alpha \vDash n$, and the \emph{length} $\ell(\alpha)$ is the number of elements in $\alpha$. It will turn out convenient to index the elements of such a composition as $\{\alpha_{h}\}_{h=0}^{k-1}$. We recall the following relations.
\begin{lem}
For an integer $n$ and a subset $J\subseteq\mathbb{N}_{n-1}$, of size $k-1$, write the elements of $J$ as $\{j_{h}\}_{h=1}^{k-1}$ in increasing order, and set $j_{0}=0$ and $j_{k}=n$. Then there exists a unique composition $\operatorname{comp}_{n}J \vDash n$, of length $k$, whose $h$th element is $j_{h+1}-j_{h}$ for any $0 \leq h<k$. Conversely, if $\alpha \vDash n$ and $\ell(\alpha)=k$ then by setting $j_{0}=0$ and inductively $j_{h+1}=j_{h}+\alpha_{h}$ and intersecting the resulting set with $\mathbb{N}_{n-1}$ yields the unique subset $J\subseteq\mathbb{N}_{n-1}$ such that $\alpha=\operatorname{comp}_{n}J$. \label{setscomp}
\end{lem}
Lemma \ref{setscomp} is very straightforward to verify (see, e.g., pages 17--18 of \cite{[St1]}), and is sometimes known as the stars-and-bars argument. We thus denote the subset $J\subseteq\mathbb{N}_{n-1}$ for which $\alpha=\operatorname{comp}_{n}J$ via this lemma as $\operatorname{comp}_{n}^{-1}\alpha$.

Recall that a \emph{multi-set} is a set where each element shows up with a finite, integral multiplicity, and the size of the multi-set is the sum of these multiplicities. If $\alpha \vDash n$ and $\ell(\alpha)=n$, then we define the \emph{content represented by $\alpha$} to be the multi-set $\mu$, whose elements are the integers $0 \leq h<k$, and where the multiplicity of $h$ in $\mu$ is $\alpha_{h}$. The size of $\mu$ is therefore $n$, and it is clear that $\mu$ determines $\alpha$ as the sequence of its multiplicities.

\begin{ex}
For $n=8$ and $J:=\{1,2,4,5,7\}$ we have $\operatorname{comp}_{8}J=112121\vDash8$, of length 6, respective cumulative sums as in $J\cup\{0,8\}$, and the represented content is 01223445 as a non-decreasing sequence (where the maximal entry is $6-1=5$). The composition associated with the subset $\{2,4,7\}$ is $2231\vDash8$, with the appropriate cumulative sums, the length is 4, and the content is 00112223, ending in $4-1=3$. We can write these contents, using multiplicities, as $012^{2}34^{2}5$ and $0^{2}1^{2}2^{3}3$ respectively. \label{excomp}
\end{ex}

We will allow semi-standard Young tableaux to have non-positive integers, including 0, as their entries (but for standard Young tableaux of shape $\lambda \vdash n$ we require that the entries are the numbers from $\mathbb{N}_{n}$, without 0), and the \emph{content} of a Young tableau is the multi-set of its entries. The explanation for our terminology is given in the following generalization of the familiar cocharge tableaux from Definition \ref{cocharge}.
\begin{defn}
We define a \emph{generalized cocharge tableau} to be a semi-standard Young tableau whose content $\mu$ involves all the integers $0 \leq h<k$ for some number $k$, all with positive multiplicities. If $\alpha \vDash n$ is the composition for which $\mu$ is the content represented by $\alpha$, then the \emph{type} of the generalized cocharge tableau is defined to be $J:=\operatorname{comp}_{n}^{-1}\alpha\subseteq\mathbb{N}_{n-1}$. \label{gencct}
\end{defn}
It is clear from Lemma \ref{setscomp} and Definition \ref{gencct} that the size of the type $J$ of a generalized cocharge tableau is the maximal element $k-1$ showing up in it, or equivalently in its content $\mu$.

\medskip

Recall the standardization process from Definition A1.5.5 of \cite{[St2]} or the end of Section 2.1 of \cite{[vL]}, which takes a semi-standard Young tableau to a standard one of the same shape. The process producing cocharge tableaux is some kind of inverse to it, namely a destandardization, similar to the one showing up in the proof of Lemma 2.2 of \cite{[GR]}. We shall need the following properties of these constructions.
\begin{lem}
An element $S\in\operatorname{SYT}(\lambda)$ is the standardization of a generalized cocharge tableau of shape $\lambda$ and type $J\subseteq\mathbb{N}_{n-1}$, with $|J|=k-1$, if and only if $\operatorname{Dsi}(S) \subseteq J$. The standardization is a bijection between the set of these generalized cocharge tableaux and the set $\{S\in\operatorname{SYT}(\lambda)\;|\;\operatorname{Dsi}(S) \subseteq J\}$, so that the number of elements in the latter set is the Kostka number $K_{\lambda,\alpha}$, where $\alpha$ is the composition $\operatorname{comp}_{n}J$. \label{ctJinv}
\end{lem}

\begin{proof}
Assume that $S$ is the standardization of the generalized cocharge tableau $C$. As the standardization replaces the boxes containing $h$ in $C$ by the numbers $j_{h}<p \leq j_{h+1}$, we deduce that if $v_{S}(p)$ has the entry $h$ in $C$, then $h$ is precisely the number of elements $j \in J$ that satisfy $j<p$. Moreover, for every $0 \leq h<k$ the number of such values of $p$ is $j_{h+1}-j_{h}=\alpha_{h}$, where $\alpha=\operatorname{comp}_{n}J$ via Lemma \ref{setscomp}. Therefore the destandardization, reconstructing $C$ from $S$, amounts to replacing every $p$ by this number, which will produce the only tableau $C$ of the content represented by $\operatorname{comp}_{n}J$ whose standardization equals $S$.

It therefore suffices to prove that if $C$ and $S$ are thus related, then $C$ is semi-standard if and only if $\operatorname{Dsi}(S) \subseteq J$. Indeed, this will prove the first assertion, the fact that the standardization and destandardization are inverse maps will thus produce the second one (because the domains and ranges of both will become the asserted sets, and the latter one will directly follow from the definition of the Kostka number $K_{\lambda,\alpha}$ as counting semi-standard Young tableaux of shape $\lambda$ and content represented by $\alpha$.

So assume that $\operatorname{Dsi}(S) \subseteq J$, and note that the box to the right of $v_{S}(p)$ and the one below $v_{S}(p)$ contain in $S$ numbers that are larger than $p$, so that the values they produce by destandardization equal at least the value obtained by $p$. Moreover, for the box below $v_{S}(p)$, say $v_{S}(q)$ for some $q>p$, there must be, by Definition \ref{asSYT}, at least one index $p \leq i<q$ such that $i\in\operatorname{Dsi}(S)$. Our assumption on $J$ yields $i \in J$, so that the value associated to $q$ by our destandardization is strictly larger than the one for $q$ (since it counts the number of elements of a set with at least one element more). We have thus established that if $\operatorname{Dsi}(S) \subseteq J$ then $C$ is semi-standard.

Conversely, the construction of the standardization uses the fact that in a semi-standard tableau, every column contains at most one box with a given entry, and if two boxes contain the same value then they are either in the same row, or one is above and strictly to the right of the other (this follows from an argument as in Remark \ref{boxnear}). So fix some $0 \leq h<k$, and recall that among the boxes containing $h$ in $C$, the values of the corresponding entries in the standardization $S$ of $C$ (which were seen to be the numbers $j_{h}<p \leq j_{h+1}$) increase from the leftmost box among them to the rightmost one.

But when we take $j_{h}<i<j_{h+1}$ and compare the locations of $v_{S}(i)$ and $v_{S}(i+1)$, this implies, via Definition \ref{asSYT}, that none of these numbers can be in $\operatorname{Dsi}(S)$, hence they are all in $\operatorname{Asi}(S)$. Since every element $i$ of the complement of $J$ in $\mathbb{N}_{n-1}$ satisfies this inequality for some $0 \leq h<k$, we deduce that this complement is contained in $\operatorname{Asi}(S)$. Therefore the complement $\operatorname{Dsi}(S)$ of $\operatorname{Asi}(S)$ in $\mathbb{N}_{n-1}$ must be contained in the set of remaining elements, which is $J$, namely we get $\operatorname{Dsi}(S) \subseteq J$ as desired. This completes the proof the lemma.
\end{proof}

We thus make the following definition.
\begin{defn}
Take $\lambda \vdash n$, $J\subseteq\mathbb{N}_{n-1}$, and $S\in\operatorname{SYT}(\lambda)$ for which $\operatorname{Dsi}(S) \subseteq J$. Then we denote the destandardization of $S$, as in Lemma \ref{ctJinv}, by $\operatorname{ct}_{J}(S)$. The standardization of a generalized cocharge tableau $C$ of type $J$ will be denoted by $\operatorname{ct}_{J}^{-1}(C)$, being the inverse of $\operatorname{ct}_{J}$. The \emph{cocharge tableau} $\operatorname{ct}(S)$ of $S$ is defined to be $\operatorname{ct}_{J}(S)$ when $J=\operatorname{Dsi}(S)$, and the \emph{total cocharge}, or just \emph{cocharge}, of $S$, denoted by $\operatorname{cc}(S)$, is the sum of the entries in $\operatorname{ct}(S)$. \label{cocharge}
\end{defn}
Up to the difference between starting from 0 vs. starting from 1, the cocharge tableaux from Definition \ref{cocharge} are the same as the quasi-Yamanouchi tableaux from Definition 2.4 of \cite{[AS]}, Definition 2.15 in the recent pre-print \cite{[BCDS]}, and other references. The map $\operatorname{ct}$ (with no index) is related to Definition 2.5 and Lemma 2.6 of \cite{[AS]}, Lemma 2.16 of \cite{[BCDS]}, and similar results.

In fact, $\operatorname{ct}_{J}(S)$ from Definition \ref{cocharge} can be constructed as follows. Begin with putting, in $\operatorname{ct}_{J}(S)$, the value 0 in the box $v_{S}(1)$ (which is always the upper left box). Assume that we filled, in $\operatorname{ct}_{J}(S)$, the box $v_{S}(i)$ by the number $j$. Then $v_{S}(i+1)$ will contain $j+1$ in $\operatorname{ct}_{J}(S)$ in case $i \in J$ (this will always be the case when $\operatorname{Dsi}(S)$), and we put $j$ in that box for $\operatorname{ct}_{J}(S)$ when $i \not\in J$ (and then in particular $i\in\operatorname{Asi}(S)$). The fact that $\operatorname{ct}_{J}$ generalizes the usual construction of the cocharge tableau in this manner, and produces a generalized cocharge tableau via Definition \ref{gencct} and Lemma \ref{ctJinv}, is the reason for our notation and terminology.

The cocharge parameter from Definition \ref{cocharge}, which will be useful below as well, is given by the following formula.
\begin{cor}
For $S\in\operatorname{SYT}(\lambda)$ and $\operatorname{Dsi}(S) \subseteq J\subseteq\mathbb{N}_{n-1}$, the sum of the entries in $\operatorname{ct}_{J}(S)$ is $\sum_{i \in J}(n-i)$, and in particular $\operatorname{cc}(S)=\sum_{i\in\operatorname{Dsi}(S)}(n-i)$. \label{sumcc}
\end{cor}

\begin{proof}
We recall from the proof of Lemma \ref{ctJinv} that if $p\in\mathbb{N}_{n}$ then the value showing up in $\operatorname{ct}_{J}(S)$ is the number of elements of $J$ that are smaller than $p$. It follows that the sum in question is \[\sum_{p=1}^{n}\big|\{i \in J\;|\;i<p\}\big|=\sum_{p=1}^{n}\sum_{p>i \in J}1=\sum_{i \in J}\sum_{p=i+1}^{n}1=\sum_{i \in J}(n-i),\] as desired. The result for $\operatorname{cc}(S)$ is just the case $J=\operatorname{Dsi}(S)$ in the latter formula. This proves the corollary.
\end{proof}

\begin{ex}
Take $n=8$ and $J:=\{1,2,4,5,7\}$ as in Example \ref{excomp}, and with $\lambda=431\vdash8$ we get the tableaux \[S=\begin{ytableau} 1 & 2 & 4 & 7 \\ 3 & 6 & 8 \\ 5 \end{ytableau},\ \operatorname{ct}_{J}(S)=\begin{ytableau} 0 & 1 & 2 & 4 \\ 2 & 4 & 5 \\ 3 \end{ytableau},\mathrm{\ and\ }\operatorname{ct}(S)=\begin{ytableau} 0 & 0 & 1 & 2 \\ 1 & 2 & 3 \\ 2 \end{ytableau},\] where the subset in the index of the latter is $\operatorname{Dsi}(S)=\{2,4,7\}$. The contents are indeed the ones showing up in Example \ref{excomp} for these sets, and their respective sums are $0+1+2+2+3+4+4+5=21$ and $0+0+1+1+2+2+2+3=11$ respectively. Replacing the elements of $J$ and of $\operatorname{Dsi}(S)$ by their complements to $n=8$ yields sets whose sums of elements are $1+3+4+6+7=21$ and $1+4+6=11$ respectively, yielding the same two values as from the tableaux, exemplifying Corollary \ref{sumcc}. \label{exct}
\end{ex}

\medskip

To every $w \in S_{n}$ we would attach, using \cite{[ATY]}, a higher Specht polynomial, by applying a rescaled Young idempotent to a monomial, whose degree is a certain cocharge parameter. The idempotent will be based on the tableau $P(w)$, but we would like the monomial to arise from $Q(w)$ but in a manner that leaves the degree unchanged when we view $w$ as an element of $S_{n+1}$. We will thus set our conventions as follows.
\begin{defn}
Given a permutation $w \in S_{n}$, we denote by $\tilde{Q}(w)$ the standard Young tableau obtained from $Q(w)$ by the Sch\"{u}tzenberger evacuation process, as defined in, e.g., Section 3.9 of \cite{[Sa]}, namely we set $\tilde{Q}(w):=\operatorname{ev}Q(w)$ in the notation from \cite{[Sa]}. Thus to $w \in S_{n}$ the modified RSK algorithm attaches the pair $\big(P(w),\tilde{Q}(w)\big)$ of standard Young tableaux of the same shape. \label{modRSK}
\end{defn}

\begin{prop}
For $w \in S_{n}$ we have $\operatorname{cc}\big(\tilde{Q}(w)\big)=\sum_{i\in\operatorname{Dsl}(w)}i$. This parameter remains the same when we consider $w$ as an element of $S_{n+1}$. \label{evkeepsdeg}
\end{prop}
The statistic from Proposition \ref{evkeepsdeg} is the \emph{major index} of $w$, in the terminology from \cite{[RW]}, \cite{[HRS]}, \cite{[GR]}, and others.

\begin{proof}
Let $w_{0}$ be the longest element of $S_{n}$, interchanging 1 with $n$, 2 with $n-1$ etc., and then the reversion of $w$ as a word (denoted by $w^{r}$ in \cite{[Sa]}) is the product $ww_{0}$. Theorems 3.2.3 and 3.9.4 of \cite{[Sa]} thus mean that $P(ww_{0})=P(w)^{t}$ (namely the transpose of $P(w)$) and $Q(ww_{0})=\operatorname{ev}Q(w)^{t}$ (note that transposition of tableaux clearly commutes with evacuation, so that this notation is not ambiguous). As Theorem 3.6.6 of \cite{[Sa]} implies that inversion of $w$ interchanges the $P$ and $Q$ tableaux, we deduce that $P(w_{0}w)=\operatorname{ev}P(w)^{t}$ and $Q(w_{0}w)=Q(w)^{t}$. In total, the conjugate $w_{0}ww_{0}$ of $w$ satisfies $Q(w_{0}ww_{0})=\operatorname{ev}Q(w)=\tilde{Q}(w)$ via Definition \ref{modRSK} (as well as $P(w_{0}ww_{0})=\operatorname{ev}P(w)$, which we shall not require). This is also visible in Theorem A1.4.3 of \cite{[St2]}.

Now, in $w_{0}w$ the $i$th entry is $n+1-w_{i}$, and the negative sign easily implies that $\operatorname{Asl}(w_{0}w)=\operatorname{Dsl}(w)$ and $\operatorname{Dsl}(w_{0}w)=\operatorname{Asl}(w)$. Similarly, in $ww_{0}$ the $i$th and $(i+1)$-st entries are $w_{n+1-i}$ and $w_{n-i}$ respectively (ordered in the opposite manner), so by considering the directions of the inequalities in Definition \ref{asperm} we get $\operatorname{Asl}(ww_{0})=\{n-i\;|\;i\in\operatorname{Dsl}(w)\}$ and $\operatorname{Dsl}(ww_{0})=\{n-i\;|\;i\in\operatorname{Asl}(w)\}$. We deduce that $\operatorname{Dsl}(w_{0}ww_{0})=\{n-i\;|\;i\in\operatorname{Dsl}(w)\}$ by composing both constructions (and similarly for $\operatorname{Asl}$, which we shall not need). Applying this together with Corollary \ref{sumcc} and Lemma \ref{RSKcor} yields the equality \[\operatorname{cc}\big(\tilde{Q}(w)\big)=\operatorname{cc}\big(Q(w_{0}ww_{0})\big)=\sum_{i\in\operatorname{Dsi}(Q(w_{0}ww_{0}))}(n-i)= \sum_{i\in\operatorname{Dsi}(Q(w))}i=\sum_{i\in\operatorname{Dsl}(w)}i,\] as desired. The last assertion follows immediately from the fact that viewing as an element of $S_{n+1}$ adds $n$ to the set $\operatorname{Asl}(w)$, and leaves $\operatorname{Dsl}(w)$ as the same set. This proves the proposition.
\end{proof}
The theorems from \cite{[Sa]} that are used in the proof of Proposition \ref{evkeepsdeg}, or Theorem A1.4.3 of \cite{[St2]}, show how inversion and left and right multiplication by $w_{0}$ act on the tableaux obtained by the RSK algorithm. The only one for which the $Q$-matrix is generically $\operatorname{ev}Q(w)$ is $w_{0}ww_{0}$, where the $P$-tableau is also replaced by its evacuation tableau. It would be interesting to see whether there is a nice description, in terms of $w$, of the permutation producing $\big(P(w),\tilde{Q}(w)\big)$ directly via RSK in general.

\begin{ex}
For the permutation $w=35271486 \in S_{8}$, the tableau $Q(w)$ is $S$ from Example \ref{exct}, and we have \[P(w)=\begin{ytableau} 1 & 4 & 6 & 8 \\ 2 & 5 & 7 \\ 3 \end{ytableau},\mathrm{\ \ and\ }\tilde{Q}(w)=\begin{ytableau} 1 & 3 & 4 & 8 \\ 2 & 5 & 6 \\ 7 \end{ytableau}\mathrm{\ with\ }\operatorname{ct}=\begin{ytableau} 0 & 1 & 1 & 3 \\ 1 & 2 & 2 \\ 3 \end{ytableau}.\] Note that $\tilde{Q}(w)$ is indeed $Q(w_{0}ww_{0})$ for $w_{0}ww_{0}=31582746 \in S_{8}$, and also equals $\operatorname{ev}S$ by the evacuation process, and the corresponding cocharge tableau has entry sum $0+1+1+1+2+2+3+3=13$, which is the same as the sum $2+4+7=13$ of the entries of $\operatorname{Dsi}(S)=\{2,4,7\}$, as Proposition \ref{evkeepsdeg} predicts. \label{exev}
\end{ex}

\medskip

We now show how to detect, among the generalized cocharge tableaux from Definition \ref{gencct}, which one are obtained using $\operatorname{ct}$ from Defintion \ref{cocharge}, with no index. Recall that in a semi-standard tableau $C$, no column contains more than one box with a given entry $h$, so that if $h$ shows up in $C$ then the leftmost and rightmost boxes containing $h$ in $C$ are well-defined (and they lie in the maximal and minimal rows among the boxes containing $h$ respectively).
\begin{lem}
The generalized cocharge tableau $C$ of shape $\lambda \vdash n$ is $\operatorname{ct}(S)$ for some $S\in\operatorname{SYT}(\lambda)$ if and only if the following condition holds: For every $h>0$, if $v$ is the leftmost box containing $h$ in $C$, then there is some box in $C$ lying in a row above that of $v$, that contains $h-1$. In particular any non-zero value showing up in $C$ appears also away from the first row of $C$. \label{cctconst}
\end{lem}
By the semi-standard condition, the box in Lemma \ref{cctconst} must lie either in the same row as $v$, or to the right of it.

\begin{proof}
Let $J$ be the type of $C$ as in Definition \ref{gencct}, and set $S:=\operatorname{ct}_{J}^{-1}(C)$ via Definition \ref{cocharge}, so that $\operatorname{Dsi}(S) \subseteq J$ by Lemma \ref{ctJinv}. Then the only possible construction of $C$ as a generalized cocharge tableau is as $\operatorname{ct}_{J}(S)$, and we recall that the equality $C=\operatorname{ct}(S)$ amounts to the latter inclusion being an equality.

Now, the construction of the standardization $\operatorname{ct}_{J}^{-1}$ implies the leftmost box $v$ containing $h$ in $C$ is $v_{S}(j_{h}+1)$, while $v_{S}(j_{h})$ is the rightmost box containing $h-1$ in $C$. Definition \ref{asSYT} now shows that if $j_{h}\in\operatorname{Asi}(S)$ then $R_{S}(j_{h}+1) \leq R_{S}(j_{h})$ and $C_{S}(j_{h}+1)>C_{S}(j_{h})$, and thus all of the boxes containing $h-1$ lie strictly to the left of $v$ and in its row or below it, while in case $j_{h}\in\operatorname{Dsi}(S)$ at least the box $v_{S}(j_{h})$ satisfies $R_{S}(j_{h}+1)>R_{S}(j_{h})$ and $C_{S}(j_{h}+1) \leq C_{S}(j_{h})$, namely it lies above $v$ and at the same column or to the right of it.

Altogether, the equality $\operatorname{Dsi}(S)=J$, namely $j_{h}\in\operatorname{Dsi}(S)$ for every $h>0$, which was seen to be equivalent to $C=\operatorname{ct}(S)$, holds if and only if the asserted condition is satisfied. As it implies that the leftmost box containing a non-zero value must lie below another box, such value must appear also away from the first row. This proves the lemma.
\end{proof}
One easily verifies that the cocharge tableaux from Examples \ref{exct} and \ref{exev} satisfy the condition from Lemma \ref{cctconst}, while the generalized cocharge tableau $\operatorname{ct}_{J}(S)$ in the former example does not (indeed, the values $h=1$ and $h=4$ for which this condition is not satisfied are precisely those for which the element $j_{h}$ of $J$, namely $j_{1}=1$ and $j_{4}=5$, does not lie in $\operatorname{Dsi}(S)$).

The proof of Proposition \ref{evkeepsdeg} shows that if $S=\tilde{Q}(w)=Q(w_{0}ww_{0})$ as in Definition \ref{modRSK}, then $\operatorname{Dsl}(w)$, which equals $\operatorname{Dsi}(\operatorname{ev}S)$ via Lemma \ref{RSKcor}, is given by $\{n-i\;|\;i\in\operatorname{Dsi}(S)\}$ and the same with $\operatorname{Dsl}$ and $\operatorname{Asi}$. Based on this, and on Lemma \ref{cctconst}, we make the following definition.
\begin{defn}
Fix a partition $\lambda \vdash n$.
\begin{enumerate}[$(i)$]
\item For an element $S\in\operatorname{SYT}(\lambda)$ we set $\operatorname{Dsi}^{c}(S):=\{n-i\;|\;i\in\operatorname{Dsi}(S)\}$ and $\operatorname{Asi}^{c}(S)=\{n-i\;|\;i\in\operatorname{Asi}(S)\}$.
\item A generalized cocharge tableau satisfying the condition from Lemma \ref{cctconst} is called simply a \emph{cocharge tableau}. The set of cocharge tableaux of shape $\lambda$ will be denoted by $\operatorname{CCT}(\lambda)$.
\item If $C\in\operatorname{CCT}(\lambda)$ and $J$ is the type of $C$ then we write $\operatorname{Dsp}^{c}(C)$ for the set $\{n-i\;|\;i \in J\}$, and $\operatorname{Asp}^{c}(C)$ for the complement of $\operatorname{Dsp}^{c}(C)$ in $\mathbb{N}_{n-1}$.
\item For $C$ and $J$ as in part $(iii)$, we denote the element $S:=\operatorname{ct}_{J}^{-1}(C)\in\operatorname{SYT}(\lambda)$ simply by $\operatorname{ct}^{-1}(C)$.
\item For $C\in\operatorname{CCT}(\lambda)$, write $\Sigma(C)$ for the entry sum of $C$.
\end{enumerate} \label{Dsic}
\end{defn}
It is clear that the sets from part $(i)$ of Definition \ref{Dsic} are complements inside $\mathbb{N}_{n-1}$, and that $\operatorname{ct}$ is a bijection between $\operatorname{SYT}(\lambda)$ and $\operatorname{CCT}(\lambda)$ (whence the notation from part $(iv)$ for its inverse). Since if $C=\operatorname{ct}(S)$ then the type $J$ of $C$ equals $\operatorname{Dsi}(S)$, we immediately deduce that $\operatorname{Dsp}^{c}(C)=\operatorname{Dsi}^{c}(S)$ and $\operatorname{Asp}^{c}(C)=\operatorname{Asi}^{c}(S)$.

We will use the following properties of the sets from Definition \ref{Dsic}.
\begin{lem}
The following assertions hold:
\begin{enumerate}[$(i)$]
\item If the content of a tableau $C\in\operatorname{CCT}(\lambda)$ involves the numbers between 0 and $k-1$ then $|\operatorname{Dsp}^{c}(C)|=k-1$.
\item If we write the elements of that set as $\{i_{g}\;|\;1 \leq g<k\}$ in increasing order, then $C$ contains precisely $i_{g}$ entries that equal at least $k-g$.
\item We have the equality $\Sigma(C)=\sum_{i\in\operatorname{Dsp}^{c}(C)}i$.
\item Determining $C$ is equivalent to filling all the rows except the first one by positive integers in a semi-standard manner such that the box with coordinates $(2,1)$ contains 1, plus indicating how many times each number between 1 and $k-1$ shows up in the first row, plus the value of $n$.
\end{enumerate} \label{Dspc}
\end{lem}
Note that by setting $i_{0}=0$ and $i_{k}=n$, part $(ii)$ of Lemma \ref{Dspc} holds also for $g=0$ and for $g=k$.

\begin{proof}
Part $(i)$ follows immediately from the fact that the type $J$ of $C$ in Definition \ref{gencct} has $k-1$ elements, since $\operatorname{Dsp}^{c}(C)$ from Definition \ref{Dsic} clearly has the same size as $J$.

The proof of Lemma \ref{ctJinv} now shows that if $J=\{j_{h}\;|\;1 \leq h<k\}$ in increasing order, then $C$ contains the entry $h$ precisely $j_{h+1}-j_{h}$ times (with $j_{0}=0$ and $j_{k}=n$ as always). Summing over $h$ between $k-g$ and $k-1$, we deduce that the number of entries of $C$ that equal at least $k-g$ is $n-j_{k-g}$. As this is just $i_{g}$ by our numbering of the set from Definition \ref{Dsic} as desired, part $(ii)$ is also established.

Next, by writing $C=\operatorname{ct}(S)$ for some $S\in{SYT}(\lambda)$, we get $\Sigma(C)=\operatorname{cc}(S)$ by Definition \ref{cocharge}, which then equals $\sum_{i\in\operatorname{Dsi}(S)}(n-i)$ via Corollary \ref{sumcc}, which is the same as $\sum_{i\in\operatorname{Dsi}^{c}(S)}i$ by Definition \ref{Dsic}. Recalling that $\operatorname{Dsp}^{c}(C)=\operatorname{Dsi}^{c}(S)$ in this case, part $(iii)$ follows.

Finally, for part $(iv)$ we note that the first row is determined by $n$ and the number of appearances of all the non-zero entries in it, via the semi-standard condition. This proves the lemma.
\end{proof}
In fact, part $(iv)$ of Lemma \ref{Dspc} is valid for every generalized cocharge tableau, once the type $J$ and the size $n$ are known, by the same argument. We will be interested essentially only in the cocharge tableaux from Definition \ref{Dsic} in this paper (and in \cite{[Z1]} we will use either them or general semi-standard Young tableaux), but the generalized ones still help in some technical considerations.

\section{Higher Specht Polynomials \label{HighSpecht}}

Let $T$ be any Young tableau of some shape $\lambda \vdash n$, and content $\mathbb{N}_{n}$. Then $R(T)$ (resp. $C(T)$) is the group of permutations in $S_{n}$ that preserve the rows (resp. columns) of $T$, namely if $\lambda$ is the partition of $n$ into $\lambda_{i}$, $1 \leq i\leq\ell(\lambda)$ (with $\ell(\lambda)$ being the length of $\lambda$ as usual) then $R(T)$ is the copy of $\prod_{i=1}^{\ell(\lambda)}S_{\lambda_{i}}$ inside $S_{n}$ in which the $i$th group $S_{\lambda_{i}}$ acts on the entries of the $i$th row in $T$ (and similarly for $C(T)$, with the transpose partition $\lambda^{t}$). Explicitly, if for every $a$ there are $t_{a}$ columns of length $a$ (so that $t_{a}$ is the multiplicity of $a$ in $\lambda^{t}$), then $C(T)$ is isomorphic to $\prod_{a}S_{a}^{t_{a}}$,

We will, however, need the following larger subgroup of $S_{n}$.
\begin{lem}
Let $q_{a}\geq0$ be such that the columns of length $a$ in $\lambda^{t}$ are those between $q_{a}+1$ and $q_{a}+t_{a}$.
\begin{enumerate}[$(i)$]
\item Given $\eta \in S_{t_{a}}$ and such a tableau $T$, the element $\eta_{T} \in S_{n}$ taking an element showing up in the box $(q_{a}+i,j)$ of $T$ for $1 \leq i \leq t_{a}$ and $1 \leq j \leq a$ to the one appearing in $\big(q_{a}+\eta(i),j\big)$ and leaves the other number invariant lies in $R(T)$. Gathering all the values of $a$, we obtain an embedding of $\prod_{a}S_{t_{a}}$ into $R(T)$.
\item The subgroup of $R(T)$ from part $(i)$ normalizes $C(T)$ and intersects it trivially. Hence together with $C(T)$ it produces a semi-direct product $\tilde{C}(T)$ inside $S_{n}$.
\item The map taking an element of the group from part $(i)$ to 1 and an element of $C(T)$ to its sign produces a character $\widetilde{\operatorname{sgn}}:\tilde{C}(T)\to\{\pm1\}$.
\end{enumerate} \label{tildeCT}
\end{lem}
Alternatively, if $C(T)$ is the subgroup preserving the partition of $\mathbb{N}_{n}$ into the ordered subsets which are the columns of $T$, then $\tilde{C}(T)$ from LEmma \ref{tildeCT} is the subgroup preserving this partition as a non-ordered partition (namely sets of the same size can be interchanged).

\begin{proof}
The definition of $\eta\mapsto\eta_{T}$ makes it obvious that the map from $S_{t_{a}}$ into $S_{n}$ is an injective group homomorphism, and since the images of different $S_{t_{a}}$'s act non-trivially on disjoint sets of numbers, they commute and the image is the group from part $(i)$.

For part $(ii)$, take $\eta_{T}$ in the image of $S_{t_{a}}$, $\sigma \in C(T)$, and $p\in\mathbb{N}_{n}$. If $C_{T}(p)$ has length different from $a$, then $\eta_{T}p=p$ and $\eta_{T}^{-1}\sigma p=\sigma p$ (since $\sigma p$ is in the same column), and when $v_{T}(p)=(q_{a}+i,j)$ for some $1 \leq i \leq t_{a}$ and $1 \leq j \leq a$, the box $v_{T}(\eta_{T}p)$ is $\big(q_{a}+\eta(i),j\big)$, applying $\sigma$ sends it to the element in the box $\big(q_{a}+\eta(i),\tilde{j}\big)$ for (possibly different) $1\leq\tilde{j} \leq a$, and after $\eta_{T}^{-1}$ we get to $(q_{a}+i,\tilde{j})$. Hence the action of $\eta_{T}^{-1}\sigma\eta_{T}$ preserves the columns of $T$, and thus lies in $C(T)$. Since $C(T)$ and $R(T)$ intersect trivially, so do $C(T)$ and the (smaller) subgroup from part $(i)$, which is thus sufficient for establishing part $(ii)$.

For part $(iii)$, any element of $\tilde{C}(T)$ can be written uniquely as $\eta_{T}\sigma$ with $\eta_{T}$ in the image of $\prod_{a}S_{a}^{t_{a}}$ and $\sigma \in C(T)$, and we get $\widetilde{\operatorname{sgn}}(\eta_{T}\sigma)=\operatorname{sgn}(\sigma)$ by definition. The product of this element with another one, say $\omega_{T}\rho$ with $\widetilde{\operatorname{sgn}}$-value of $\operatorname{sgn}(\rho)$, equals $\eta_{T}\omega_{T}\cdot(\omega_{T}^{-1}\sigma\omega_{T})\rho$, and since the sign is invariant under conjugation in $S_{n}$, the $\widetilde{\operatorname{sgn}}$-value of that element is the product of $\operatorname{sgn}(\sigma)$ and $\operatorname{sgn}(\rho)$. Thus $\widetilde{\operatorname{sgn}}$ is a character, as desired. This proves the lemma.
\end{proof}
The group $\tilde{C}(T)$ from part $(ii)$ of Lemma \ref{tildeCT} is isomorphic to the product $\prod_{a}(S_{t_{a}} \rtimes S_{a}^{t_{a}})$, where each multiplier is the wreath product of $S_{t_{a}}$ and $S_{a}$. Note that the character from part $(iii)$ there does not, in general, coincide with the restriction of the sign character of $S_{n}$ to $\tilde{C}(T)$.

\medskip

This paper investigates explicit representations of $S_{n}$ on polynomials. For this we set the following notations and definitions.
\begin{defn}
Take any two integers $n\geq1$ and $d\geq0$.
\begin{enumerate}[$(i)$]
\item The shorthand $\mathbb{Q}[\mathbf{x}_{n}]$ denotes the algebra $\mathbb{Q}[x_{1},\ldots,x_{n}]$. Similarly, $\mathbb{Z}[\mathbf{x}_{n}]$ stands for $\mathbb{Z}[x_{1},\ldots,x_{n}]$.
\item The space of polynomials in $\mathbb{Q}[\mathbf{x}_{n}]$ that are homogeneous of degree $d$ will be denoted by $\mathbb{Q}[\mathbf{x}_{n}]_{d}$, and $\mathbb{Z}[\mathbf{x}_{n}]_{d}$ is its intersection with $\mathbb{Z}[\mathbf{x}_{n}]$.
\item The \emph{content} of a monomial in $\mathbb{Q}[\mathbf{x}_{n}]$ is the multi-set of exponents showing up in it. We identify two contents that are the same except for the multiplicity of 0, since they correspond to the same monomials. If we omit the multiplicity of 0 altogether, a content can be viewed, e.g., as a partition of the degree of the monomial.
\item For any $\mu \vdash d$, we denote by $\mathbb{Q}[\mathbf{x}_{n}]_{\mu}$ the space spanned by the monomials in $\mathbb{Q}[\mathbf{x}_{n}]$ whose content is $\mu$, which intersects $\mathbb{Z}[\mathbf{x}_{n}]$ in $\mathbb{Z}[\mathbf{x}_{n}]_{\mu}$.
\end{enumerate} \label{Qxnd}
\end{defn}
It is clear that $\mathbb{Q}[\mathbf{x}_{n}]_{\mu}$ from part $(iv)$ of Definition \ref{Qxnd} is a subspace of $\mathbb{Q}[\mathbf{x}_{n}]_{d}$ from part $(ii)$ when $\mu \vdash d$. Moreover, the action of $S_{n}$ preserves homogeneity degree, making each $\mathbb{Q}[\mathbf{x}_{n}]_{d}$ a representation of $S_{n}$, and since it also preserves contents of monomials, each $\mathbb{Q}[\mathbf{x}_{n}]_{\mu}$ is a sub-representation.

Recall that to a tableau $T$ of some shape $\lambda \vdash n$ and content $\mathbb{N}_{n}$ is associated the \emph{Young symmetrizer} $\varepsilon_{T}:=\sum_{\sigma \in C(T)}\sum_{\tau \in R(T)}\operatorname{sgn}(\sigma)\sigma\tau\in\mathbb{Z}[S_{n}]$ (note that some authors interchange the multipliers from $C(T)$ and from $R(T)$, and get a different element, but this convention will be much more appropriate for our purposes). It is a multiple of an idempotent, namely there is a scalar $c$, depending only on $\lambda$, such that $\varepsilon_{T}^{2}=c\varepsilon_{T}$, or equivalently $\frac{1}{c}\varepsilon_{T}$ is an idempotent in $\mathbb{Q}[S_{n}]$. In fact, Theorem 3.10 of \cite{[JK]} shows that if $\mathcal{S}^{\lambda}$ is the Specht module associated with $\lambda$ (as considered below) then this scalar is $n!/\dim\mathcal{S}^{\lambda}$ (in fact, that theorem is stated for the opposite convention of the Young symmetrizer, but it applies for our $\varepsilon_{T}$ in the same manner).

For any element of an $S_{n}$-module, $\varepsilon_{T}$ takes it to another element of that module, and the equality $\sigma\varepsilon_{T}=\operatorname{sgn}(\sigma)\varepsilon_{T}$ implies that $C(T)$ acts on the latter element via the sign representation. We now turn to the higher Specht polynomials.
\begin{defn}
Take $\lambda \vdash n$, $S\in\operatorname{SYT}(\lambda)$, $T$ of shape $\lambda$ and content $\mathbb{N}_{n}$, and $C\in\operatorname{CCT}(\lambda)$.
\begin{enumerate}[$(i)$]
\item For every $i\in\mathbb{N}$ we write $h_{i}$ for the value showing up in the box $v_{T}(i)$ of the tableau $\operatorname{ct}(S)$ from Definition \ref{cocharge} (resp. in $C$), and set $p_{T}^{S}$ (resp. $p_{C,T}$) to be $\prod_{i=1}^{n}x_{i}^{h_{i}}\in\mathbb{Q}[\mathbf{x}_{n}]$.
\item The stabilizer in $R(T)$ of $p_{T}^{S}$ (resp. $p_{C,T}$) is the same as the stabilizer of $\operatorname{ct}(S)$ (resp. $C$) there, by acting on the rows. We denote its size by $s_{T}^{S}$ (resp. $s_{C,T}$).
\item The \emph{higher Specht polynomial} $F_{T}^{S}$ (resp. $F_{C,T}$) that is associated with $T$ and $S$ (resp. $C$) is $\varepsilon_{T}p_{T}^{S}/s_{T}^{S}$ (resp. $\varepsilon_{T}p_{C,T}/s_{C,T}$).
\item For any $w \in S_{n}$ we write $F_{w}$ for the higher Specht polynomial will be $F_{P(w)}^{\tilde{Q}(w)}$, with $\tilde{Q}(w)=\operatorname{ev}Q(w)$ from Definition \ref{modRSK}.
\end{enumerate} \label{Spechtdef}
\end{defn}
See Example \ref{FTScoeff2} below for how this construction, as well as some of the initials results about it, work.

In fact, when we change $T$ to a different tableau with the same shape, the group $R(T)$ is conjugated, but the boxes are also interchanged in such a manner that the action of that group on the boxes of $\operatorname{ct}(S)$ or $C$ becomes the same. Hence the parameters $s_{T}^{S}$ and $s_{C,T}$ from Definition \ref{Spechtdef} are independent of $T$. However, we keep it in the notation in what follows. Note that they are omitted in \cite{[ATY]}, \cite{[GR]}, and other references, but dividing by them turns out to produce better properties---see Lemma \ref{operRT} and Propositions \ref{Spechtpols} and \ref{incnSpecht} below.

\begin{rmk}
It is clear from Definition \ref{Spechtdef} that if $C=\operatorname{ct}(S)$ then $p_{C,T}=p_{T}^{S}$, $s_{C,T}=s_{T}^{S}$, and $F_{C,T}=F_{T}^{S}$. The reason for this double notation (which becomes triple via part $(iv)$ there, in case $T$ is also standard), is as follows. Using $S\in\operatorname{SYT}(\lambda)$ we obtain the notation $F_{T}^{S}$ in the original terminology of \cite{[ATY]}. The construction from $C$ is more straightforward, and will later be seen to be more adapted to the stable case, in infinitely many variables. Viewed as $F_{w}$ with $w \in S_{n}$, the fact that $S_{n}$ is naturally contained in $S_{n+1}$ is the basis for relating higher Specht polynomials with different numbers of variables. \label{samepols}
\end{rmk}

\medskip

The normalization from Definition \ref{Spechtdef} yields the following result.
\begin{lem}
The expression $\sum_{\tau \in R(T)}\tau p_{T}^{S}/s_{T}^{S}$ yields a sum of monomials, all of degree $d:=\operatorname{cc}(S)$ and with the same content, each with the coefficient 1, including the original monomial $p_{T}^{S}$. None of the other monomials is related to $p_{T}^{S}$ via the action of $C(T)$. \label{operRT}
\end{lem}

\begin{proof}
The degree of the monomial $p_{T}^{S}$ from Definition \ref{Spechtdef} is $\sum_{i=1}^{n}h_{i}$, and as the $h_{i}$'s are the entries of $\operatorname{ct}(S)$, it has degree $d=\operatorname{cc}(S)$ via Definition \ref{cocharge}. The fact that all the monomials in the sum have the same content and degree now follows from the fact that the content and the degree of monomials are preserved under the action of $S_{n}$ hence of $R(T)$.

Next, the sum $\sum_{\tau \in R(T)}\tau p_{T}^{S}$ produces the sum of all the monomials in the orbit of $p_{T}^{S}$ under $R(T)$, and the number of times each one is obtained is the size of the stabilizer in $R(T)$ of any one of these monomials. As this size is $s_{T}^{S}$ for $p_{T}^{S}$ by definition, in the asserted quotient they all show up with the coefficient 1 as desired.

We now attach, to every monomial $y\in\mathbb{Z}[\mathbf{x}_{n}]$, the vector $\deg_{T}y$ of length $\lambda_{1}$, whose $j$th entry $\deg_{T}^{j}y$ is the degree of $y$ in the variables $x_{i}$ with $C_{T}(i)=j$. When $y$ and $z$ are two monomials, we shall write $y>_{T}z$ to indicate that $\deg_{T}y$ is (strictly) larger than $\deg_{T}z$ in the reverse lexicographic order (this order might be related to some of the orders showing up at the beginning of \cite{[M]}). As the action of $C(T)$ preserves the columns of $T$ by definition, the vector $\deg_{T}y$ remains invariant if we change $y$ by one of its images under this group.

It thus follows from Definition \ref{Spechtdef} that $\deg_{T}^{j}p_{T}^{S}$ is the sum of the entries in the $j$th column of $\operatorname{ct}(S)$. Similarly, if we replace this monomial by its image under $\tau \in R(T)$, then the corresponding vector contains the sums of the degrees of the columns of $\tau\operatorname{ct}(S)$. But $\operatorname{ct}(S)$ is semi-standard (by Lemma \ref{ctJinv}), so that the action of any element of $R(T)$ can only take smaller numbers, showing up in a column more to the left, and put them in a column that is more to the right of their original location.

So take any $\tau \in R(T)$, and assume that it does not stabilize $p_{T}^{S}$, or equivalently $\operatorname{ct}(S)$. Let $j$ be the maximal column of an entry of $\operatorname{ct}(S)$ that is not stabilized by $\tau$. Hence the value of some box in that column of $\tau\operatorname{ct}(S)$ is smaller than its value in $\operatorname{ct}(S)$, with the others being at most the same as those in $\operatorname{ct}(S)$, with those in all the columns to the right of $j$ being the same in $\operatorname{ct}(S)$ and in $\tau\operatorname{ct}(S)$.

We deduce that every other monomial in our sum, which is $\tau p_{T}^{S}$ for $\tau \in R(T)$ that does not stabilize $p_{T}^{S}$, satisfies
$p_{T}^{S}>_{T}\tau p_{T}^{S}$, which is a strict inequality. As we saw that the action of $C(T)$ preserves the vector on which this inequality is based, it indeed cannot relate any such summand to $p_{T}^{S}$. This completes the proof of the lemma.
\end{proof}
The proof of Lemma \ref{operRT} shows, via Remark \ref{samepols}, that $\sum_{\tau \in R(T)}\tau p_{C,T}/s_{C,T}$ can be described similarly with the degree $d:=\Sigma(C)$. The higher Specht polynomials themselves now have the following properties.
\begin{prop}
The polynomial $F_{T}^{S}$ from Definition \ref{Spechtdef} is an element of $\mathbb{Z}[\mathbf{x}_{n}]_{d}$, where $d:=\operatorname{cc}(S)$, which contains the monomial $p_{T}^{S}$ with the multiplier 1. The group $\tilde{C}(T)$ from Lemma \ref{tildeCT} acts on it via the character $\widetilde{\operatorname{sgn}}$. \label{Spechtpols}
\end{prop}

\begin{proof}
The expression from Lemma \ref{operRT} was seen there to be in $\mathbb{Z}[\mathbf{x}_{n}]_{d}$. Since $F_{T}^{S}$ is obtained, via Definition \ref{Spechtdef}, from that expression under the action of the operator $\sum_{\sigma \in C(T)}\operatorname{sgn}(\sigma)\sigma\in\mathbb{Z}[S_{n}]$, whose elements preserve degrees, we get that $F_{T}^{S}\in\mathbb{Z}[\mathbf{x}_{n}]_{d}$ as required.

Next, it is clear that multiplying the latter element of $\mathbb{Z}[S_{n}]$ from the left by an element of $C(T)$ is the same as multiplying it by a sign. Hence letting such an element act of $F_{T}^{S}$ multiplies it by its sign. For obtaining the action of the larger group $\tilde{C}(T)$, it suffices to consider an element $\rho$ in the image of $\prod_{a}S_{a}^{t_{a}}$ from part $(i)$ of Lemma \ref{tildeCT}. Since it normalizes $C(T)$ by part $(ii)$ there, it commutes with $\sum_{\sigma \in C(T)}\operatorname{sgn}(\sigma)\sigma\in\mathbb{Z}[S_{n}]$, and the fact that it lies in $R(T)$ implies that $\rho\varepsilon_{T}=\varepsilon_{T}$ and thus it preserves $F_{T}^{S}$. Hence $\tilde{C}(T)$ acts on the higher Specht polynomial $F_{T}^{S}$ as the asserted character.

Finally, Lemma \ref{operRT} implies that the only contributions to the multiplicity of $p_{T}^{S}$ in $F_{T}^{S}$ can be from $\sum_{\sigma \in C(T)}\operatorname{sgn}(\sigma)\sigma p_{T}^{S}$ (since all the other summands in the expression there are not in the $C(T)$-orbit of that monomial). But as $\operatorname{ct}(S)$ is semi-standard, each of the monomials $\sigma P_{T}^{S}$ are distinct (with the action of $\sigma \in C(T)$ on the $j$th column being easily read from the orderings of the exponents of the variables $x_{i}$ with $C_{T}(i)=j$ in $\sigma P_{T}^{S}$), and the only contribution is from the term with trivial $\sigma$, yielding the asserted coefficient 1. This completes the proof of the proposition.
\end{proof}
Remark \ref{samepols} implies the same for $F_{C,T}$, where the degree $d$ being $\Sigma(C)$. In fact, if $\mu$ is the content of $C$ (or of $\operatorname{ct}(S)$), with or without the zeros omitted, then $F_{C,T}$ (or $F_{T}^{S}$) lies, in fact, in $\mathbb{Z}[\mathbf{x}_{n}]_{\mu}$, as does the sum from Lemma \ref{operRT} (also from $p_{C,T}$).

\begin{ex}
Consider the case where $n=4$, $\lambda=31\vdash4$, $w=2134$, and we have the tableaux \[T:=\begin{ytableau} 1 & 3 & 4 \\ 2 \end{ytableau},\ \tilde{S}:=\begin{ytableau} 1 & 3 & 4 \\ 2 \end{ytableau},\ S:=\begin{ytableau} 1 & 2 & 3 \\ 4 \end{ytableau},\mathrm{\ and\ }C:=\begin{ytableau} 0 & 0 & 0 \\ 1 \end{ytableau}.\] In this case we have $T=P(w)$, $\tilde{S}=Q(w)$, $S=\operatorname{ev}\tilde{S}=\tilde{Q}(w)$, and $C=\operatorname{ct}(S)$, and we get the monomial $p_{T}^{S}=p_{C,T}=x_{2}$. It is stabilized by $R(T)$, which has size $s_{T}^{S}=s_{C,T}=6$, so that the expression from Lemma \ref{operRT} is again $x_{2}$. After the application of the expression associated with $C(T)$ we get $F_{T}^{S}=F_{C,T}=x_{2}-x_{1}$ (note that the convention from \cite{[GR]} would have given $6(x_{2}-x_{1})$, while the normalization of \cite{[ATY]} produces $\frac{3}{4}(x_{2}-x_{1})$, because $\varepsilon_{T}^{2}=8\varepsilon_{T}$). Take now $n=5$, $\lambda=32\vdash5$, $w=24153$, and set \[T:=\begin{ytableau} 1 & 3 & 5 \\ 2 & 4 \end{ytableau},\ \tilde{S}:=\begin{ytableau} 1 & 2 & 4 \\ 3 & 5 \end{ytableau},\ S:=\begin{ytableau} 1 & 3 & 5 \\ 2 & 4 \end{ytableau},\mathrm{\ and\ }C:=\begin{ytableau} 0 & 1 & 2 \\ 1 & 2 \end{ytableau}.\] Then $T=P(w)$, $\tilde{S}=Q(w)$, $S=\operatorname{ev}\tilde{S}=\tilde{Q}(w)$, and $C=\operatorname{ct}(S)$, and we get the monomial $p_{T}^{S}=p_{C,T}=x_{2}x_{3}x_{4}^{2}x_{5}^{2}$. Here $s_{T}^{S}=s_{C,T}=1$, the group $C(T)$ has order 4, and applying the summation of the images under $R(T)$ yields \[x_{2}x_{3}x_{4}^{2}x_{5}^{2}+x_{1}^{2}x_{2}x_{3}x_{4}^{2}+x_{1}^{2}x_{2}x_{4}^{2}x_{5}+x_{2}^{2}x_{3}^{2}x_{4}x_{5}+x_{1}x_{2}^{2}x_{4}x_{5}^{2}+x_{1}x_{2}^{2}x_{3}^{2}x_{4}+\] \[+x_{2}x_{3}^{2}x_{4}^{2}x_{5}+x_{1}x_{2}x_{4}^{2}x_{5}^{2}+x_{1}x_{2}x_{3}^{2}x_{4}^{2}+x_{2}^{2}x_{3}x_{4}x_{5}^{2}+x_{1}^{2}x_{2}^{2}x_{3}x_{4}+x_{1}^{2}x_{2}^{2}x_{4}x_{5},\] indeed a sum of monomials as in Lemma \ref{operRT}. The first summand has $\deg_{T}$ vector of $(1,3,2)$, while the others yield $(2,2,2)$ (twice), $(3,1,2)$, $(1,4,1)$, $(2,3,1)$, $(3,2,1)$, $(4,1,1)$, $(2,4,0)$, $(3,3,0)$ (twice), and $(4,2,0)$, exemplifying the proof of the last assertion in that lemma. When applying the operator corresponding to $C(T)$, all the summands in the second line vanish (as they have non-trivial stabilizers there), and the value of $F_{T}^{S}=F_{C,T}$ is \[x_{5}^{2}[x_{2}x_{3}x_{4}^{2}-x_{1}x_{3}x_{4}^{2}-x_{2}x_{3}^{2}x_{4}+x_{1}x_{3}^{2}x_{4}+x_{1}x_{2}^{2}x_{4}-x_{1}^{2}x_{2}x_{4}-x_{1}x_{2}^{2}x_{3}+x_{1}^{2}x_{2}x_{3}]+\] \[-x_{5}[x_{1}x_{2}^{2}x_{4}^{2}-x_{1}^{2}x_{2}x_{4}^{2}-x_{1}x_{2}^{2}x_{3}^{2}+x_{1}^{2}x_{2}x_{3}^{2}+x_{2}^{2}x_{3}x_{4}^{2}-x_{2}^{2}x_{3}^{2}x_{4}-x_{1}^{2}x_{3}x_{4}^{2}+x_{1}^{2}x_{3}^{2}x_{4}]+\]
\[-2x_{1}x_{2}^{2}x_{3}x_{4}^{2}+2x_{1}^{2}x_{2}x_{3}x_{4}^{2}+2x_{1}x_{2}^{2}x_{3}^{2}x_{4}-2x_{1}^{2}x_{2}x_{3}^{2}x_{4}.\] Note that $\tilde{C}(T)$ contains $C(T)$ with index 2, and the permutation interchanging 1 with 3 and 2 with 4 simultaneously lies in it and preserves this polynomial. \label{FTScoeff2}
\end{ex}
The second case in Example \ref{FTScoeff2} shows that the higher Specht polynomials can contain monomials with coefficients that are not $\pm1$, and is, in fact, the minimal one (in terms of the value of $n$) exemplifying this.

\smallskip

The classical Specht polynomials are defined as follows.
\begin{defn}
Assume that $\lambda \vdash n$ is given. Then $S^{0}\in\operatorname{SYT}(\lambda)$ is the element whose entries increase consecutively first along the rows and then along the columns, and in $C^{0}\in\operatorname{CCT}(\lambda)$ all the entries in the $i$th row are equal to $i-1$. Then $C^{0}=\operatorname{ct}(S^{0})$, and $F_{T}^{S^{0}}=F_{C^{0},T}$ is called simply a \emph{Specht polynomial}. \label{clasSpecht}
\end{defn}
The first polynomial in Example \ref{FTScoeff2} is a Specht polynomial as in Definition \ref{clasSpecht}. Its degree is clearly minimal among all the higher Specht polynomials arising from the shape $\lambda$. In fact, we deduce the following stronger consequence.
\begin{cor}
For $T$, $S$, and $C$ as in Definition \ref{Spechtdef}, let $S^{0}$ and $C^{0}$ be given in Definition \ref{clasSpecht} for the same shape $\lambda$. Then the quotients $Q_{T}^{S}:=F_{T}^{S}/F_{T}^{S^{0}}$ and $Q_{C,T}:=F_{C,T}/F_{C^{0},T}$ are homogeneous polynomials in $\mathbb{Z}[\mathbf{x}_{n}]$ that are invariant under the group $\tilde{C}(T)$ from Lemma \ref{tildeCT}. \label{quotSpecht}
\end{cor}
The degree of $Q_{T}^{S}$ from Corollary \ref{quotSpecht} is $\operatorname{cc}(S)-\operatorname{cc}(S^{0})$, while for $Q_{C,T}$ it is the sum of the entries of $C$ minus those of $C^{0}$. When $C=\operatorname{ct}(S)$ we get $Q_{C,T}=Q_{T}^{S}$, and if $T$ and $S$ are given as $P(w)$ and $\tilde{Q}(w)$ respectively for some $w \in S_{n}$ as in Definitions \ref{modRSK} and \ref{Spechtdef}, then we may denote the quotient $Q_{T}^{S}$ from Corollary \ref{quotSpecht} also by $Q_{w}$. This will be useful in Remark \ref{stabquot} below.

\begin{proof}
Recalling the form of $C^{0}=\operatorname{ct}(S^{0})$ from Definition \ref{clasSpecht}, it is clearly invariant under the action of $R(T)$, so that $s_{T}^{S}=|R(T)|$ and the polynomial $F_{T}^{S^{0}}=F_{C^{0},T}$ is just the image of $p_{T}^{S^{0}}=p_{C^{0},T}$ under $\sum_{\sigma \in C(T)}\operatorname{sgn}(\sigma)\sigma$. But as the variable $x_{i}$ shows up in this monomial with the exponent $R_{T}(i)-1$, we deduce that if we denote by $A_{k}$ the set of entries $i$ with $C_{T}(i)=k$ then the polynomial from that definition becomes the product $\prod_{k=1}^{\lambda_{1}}\prod_{i<j \in A_{k}}(x_{j}-x_{i})$ (this product assumes that the columns of $T$ are ordered increasingly, which is always the case when $T$ is standard, but we keep this assumption for ease of notation, and the argument works in the same manner in general).

Turning our attention to $F_{T}^{S}$ or $F_{C,T}$, any monomial $\tau p_{T}^{S}$ or $P_{C,T}$ showing up in the proof of Proposition \ref{Spechtpols} either contains two different elements in $A_{k}$ for some $k$ with the same exponent, or it does not. The first terms are annihilated by the operator $\sum_{\sigma \in C(T)}\operatorname{sgn}(\sigma)\sigma$, and this operator takes any term of the second type to the product over $k$ of a determinant involving the variables $\{x_{i}\;|\;i \in A_{k}\}$ with different $A_{k}$ exponents. But this determinant is $\prod_{i<j \in A_{k}}(x_{j}-x_{i})$ times a sign (arising from the fact that the exponents need not match the indices in increasing order) times Schur polynomial, and the latter is in $\mathbb{Z}[\mathbf{x}_{n}]$. As the first multiplier combine to $F_{T}^{S^{0}}=F_{C^{0},T}$, and the second one produces a polynomial over $\mathbb{Z}$, the second product, which becomes $Q_{T}^{S}$ or $Q_{C,T}$ by definition, is a polynomial in $\mathbb{Z}[\mathbf{x}_{n}]$ (whose homogeneity and degree follows from those of the other two multipliers).

Finally, we recall from Proposition \ref{Spechtpols} that the group $\tilde{C}(T)$ acts on both $F_{T}^{S}$ and $F_{T}^{S^{0}}$, or $F_{C,T}$ and $F_{C^{0},T}$, via the same character $\operatorname{sgn}$ from Lemma \ref{tildeCT}. Hence the quotient is invariant under all of $\tilde{C}(T)$. This proves the corollary.
\end{proof}

\begin{ex}
As the first case in Example \ref{FTScoeff2} is an ordinary Specht polynomial, the corresponding quotient from Corollary \ref{quotSpecht} is just the constant 1 there. In the second case from that example, the polynomial from Definition \ref{clasSpecht} is $(x_{2}-x_{1})(x_{4}-x_{3})$, and the quotient $Q_{T}^{S}=Q_{C,T}$ equals \[x_{5}^{2}(x_{1}x_{2}+x_{3}x_{4})-x_{5}[x_{1}x_{2}(x_{3}+x_{4})+(x_{1}+x_{2})x_{3}x_{4}]-2x_{1}x_{2}x_{3}x_{4},\] which is invariant under interchanging $x_{1}$ with $x_{2}$, under interchanging $x_{3}$ with $x_{4}$, and interchanging $x_{1}$ and $x_{2}$ with $x_{3}$ and $x_{4}$, namely under all of $\tilde{C}(T)$. \label{quotex}
\end{ex}

\medskip

It follows from Definition \ref{Spechtdef} and Proposition \ref{Spechtpols} that for $w \in S_{n}$, the higher Specht polynomial $F_{w}$ is homogeneous of degree $\operatorname{cc}\big(\tilde{Q}(w)\big)$, and this parameter equals $\sum_{i\in\operatorname{Dsl}(w)}i$ and remains the same when we view $w$ as an element of $S_{n+1}$ (see Proposition \ref{evkeepsdeg}). In order to relate $F_{w}$ for $w \in S_{n}$ and $w \in S_{n+1}$, we now introduce some notation.
\begin{defn}
Let $n\geq1$ be an integer.
\begin{enumerate}[$(i)$]
\item If $w \in S_{n}$, let $w_{+} \in S_{n+1}$ be the element acting on $\mathbb{N}_{n}$ like $w$ and fixing $n+1$.
\item For a partition $\lambda \vdash n$, let $\lambda_{+}$ be the partition of $n+1$ obtained by adding a box to the first row of $\lambda$.
\item If $S$ is any tableau of shape $\lambda$ and content $\mathbb{N}_{n}$, then we denote by $\iota S$ the tableau of shape $\lambda_{+}$ that is obtained by filling $\lambda$ according to $S$ and putting $n+1$ in the new box of $\lambda_{+}$.
\item When $S\in\operatorname{SYT}(\lambda)$, we set $\tilde{\iota}S:=\operatorname{ev}\iota(\operatorname{ev}S)$.
\item Given a generalized cocharge tableau $C$ of shape $\lambda$ and type $J\subseteq\mathbb{N}_{n-1}$, we write $J_{+}:=\{j+1\;|\;j \in J\}$ and set $\hat{\iota}C:=\operatorname{ct}_{J_{+}}\big(\tilde{\iota}\operatorname{ct}_{J}^{-1}(C)\big)$.
\end{enumerate} \label{iotadef}
\end{defn}
The permutation $w_{+}$ from Definition \ref{iotadef} is the one denoted by $w\times1$ in \cite{[PR]} and others.

\begin{rmk}
It is clear that if $S\in\operatorname{SYT}(\lambda)$ then $\iota S\in\operatorname{SYT}(\lambda_{+})$, so that we can apply $\operatorname{ev}$ in part $(iv)$ of Definition \ref{iotadef}. In this case we also have the equality $\operatorname{Dsi}(\iota S)=\operatorname{Dsi}(S)$, so that the proof of Proposition \ref{evkeepsdeg} implies that $\operatorname{Dsi}(\tilde{\iota}S)=\{i+1\;|\;i\in\operatorname{Dsi}(S)\}$. Hence in part $(v)$ of that definition, we have $\operatorname{Dsi}(S) \subseteq J$ for $S:=\operatorname{ct}_{J}^{-1}(C)$ by Lemma \ref{ctJinv}, so that $\operatorname{Dsi}(\tilde{\iota}S) \subseteq J_{+}$ by the latter equality and $\hat{\iota}C$ is also well-defined. \label{iotawelldef}
\end{rmk}

The properties of $\tilde{\iota}$ and $\hat{\iota}$ from Definition \ref{iotadef} are as follows.
\begin{lem}
Let $S\in\operatorname{SYT}(\lambda)$ and a generalized cocharge tableau $C$ be given.
\begin{enumerate}[$(i)$]
\item Set $S_{+}$ to be the tableau obtained by adding 1 to each of the entries of $S$, add the new box for $\lambda_{+}$, push all the entries of the first row of $S_{+}$ one box to the right, and put 1 in the remaining upper left box. The resulting tableau is $\tilde{\iota}S\in\operatorname{SYT}(\lambda_{+})$.
\item We have the equality $\operatorname{Dsi}^{c}(\tilde{\iota}S)=\operatorname{Dsi}^{c}(S)$ in this case.
\item Consider the tableau $C$, add the box for $\lambda_{+}$, shove all the entries of the first row to the right, and complete with 0 the upper left corner. This produces $\hat{\iota}C$.
\item If $C\in\operatorname{CCT}(\lambda)$ then $\hat{\iota}C$ is in $\operatorname{CCT}(\lambda_{+})$ and can be expressed also as $\operatorname{ct}\big(\tilde{\iota}\operatorname{ct}^{-1}(C)\big)$, and we have $\operatorname{Dsp}^{c}(\hat{\iota}C)=\operatorname{Dsp}^{c}(C)$ and $\Sigma(\hat{\iota}C)=\Sigma(C)$.
\item The data from part $(iv)$ of Lemma \ref{Dspc} is the same for $\hat{\iota}C$ from Definitions \ref{iotadef} and \ref{Dsic}, except for replacing $n$ by $n+1$.
\end{enumerate} \label{iotaexp}
\end{lem}

\begin{proof}
Let $\tilde{S}$ be the asserted formula for $\tilde{\iota}S$ in part $(i)$, and we write $S_{i,j}$ and $\tilde{S}_{i,j}$ for the value showing up in the box $(i,j)$ of $S$ or of $\tilde{S}$. We therefore get $\tilde{S}_{i,j}=S_{i,j}+1$ wherever $i>1$, and $\tilde{S}_{1,j}=S_{1,j-1}+1$ for any $j>1$. Recalling that $S$ is standard, we deduce the standard condition on the rows below the first one, that on the upper left box and its neighbors follows from the fact that $\tilde{S}_{1,1}=1$, the first row of $\tilde{S}$ is also strictly increasing, and we also have $\tilde{S}_{1,j}<\tilde{S}_{1,j+1}<\tilde{S}_{2,j}$. In particular we get that $\tilde{S}\in\operatorname{SYT}(\lambda_{+})$.

Now, Proposition A1.4.2 of \cite{[St2]} that the evacuation process is an involution on standard Young tableaux (this can also be obtained via proof of Proposition \ref{evkeepsdeg}, since $w_{0}$ has order 2 in $S_{n}$). Hence part $(i)$ is equivalent to the equality $\operatorname{ev}\tilde{S}=\iota\operatorname{ev}S$, which we now prove.

Indeed, we saw in the first paragraph that $\tilde{S}_{1,j}<\tilde{S}_{1,j+1}<\tilde{S}_{2,j}$, so that in the first step of the evacuation process we only move along the first row. This puts $n+1$ in the designated box, and that step we get the tableau $S_{+}$, on which the evacuation process behaves in the same manner as on $S$. In total $\operatorname{ev}\tilde{S}$ indeed produces $\iota\operatorname{ev}S$, and part $(i)$ is thus established.

We now recall the equality $\operatorname{Dsi}(\tilde{\iota}S)=\{i+1\;|\;i\in\operatorname{Dsi}(S)\}$ from Remark \ref{iotawelldef}. Since $\operatorname{sh}(S)=\lambda \vdash n$ and $\operatorname{sh}(\tilde{\iota}S)=\lambda_{+} \vdash n+1$, applying Definition \ref{Dsic} yields the equality from part $(ii)$.

Next, if $J$ is the type of $C$ then we write $C=\operatorname{ct}_{J}(S)$ for some $S\in\operatorname{SYT}(\lambda)$ such that $J$ contains $\operatorname{Dsi}(S)$, and we need to show that the construction of $\hat{\iota}C$ indeed produces $\operatorname{ct}_{J_{+}}(\tilde{\iota}S)$. One way to see that is apply $\operatorname{ct}_{J_{+}}^{-1}$, which is just the standardization, to the asserted formula for $\hat{\iota}C$, and observe that it indeed produces $\tilde{\iota}S$.

Equivalently, we saw in the proof of Lemma \ref{ctJinv} that in $\operatorname{ct}_{J}(S)$, every entry $p$ is replaced by the number of elements of $J$ that are smaller than $p$. But then this number is the number of elements of $J_{+}$ that are smaller than $p+1$, which in $\tilde{\iota}S$ lies in $v_{S}(p)$ in case $R_{S}(p)>1$, and when $R_{S}(p)=1$ it is in the box $\big(1,C_{S}(p)+1\big)$ in that tableau. Since the upper left entry in any generalized cocharge tableau is 0 (in correspondence with $(1,1)$ being $v_{\tilde{\iota}S}(1)$, and $J_{+}$ contains no elements that are smaller than 1), the asserted tableau for $\hat{\iota}C$ in part $(iii)$ is indeed obtained as $\operatorname{ct}_{J_{+}}(\tilde{\iota}S)$ as desired.

Finally, if $C\in\operatorname{CCT}(\lambda)$ then $S$ from the previous two paragraphs satisfies $\operatorname{Dsi}(S)=J$, so that $\operatorname{Dsi}(\tilde{\iota}S)=J_{+}$ via Remark \ref{iotawelldef} and $\hat{\iota}C$ is indeed $\operatorname{ct}(\tilde{\iota}S)$ and lies in $\operatorname{CCT}(\lambda_{+})$ (the $\operatorname{CCT}$ condition can also be observed by the fact that if $C$ satisfies the condition from Lemma \ref{cctconst}, then it is easily verified that so does $\hat{\iota}C$). Since Definition \ref{Dsic} implies that $\operatorname{Dsp}^{c}(C)=\operatorname{Dsi}^{c}(S)$ and $\operatorname{Dsp}^{c}(\hat{\iota}C)=\operatorname{Dsi}^{c}(\tilde{\iota}S)$ in this case, part $(ii)$ implies the first equality in part $(iv)$ with part $(iii)$ of Lemma \ref{Dspc} yielding the second one. Part $(v)$ is an immediate consequence of the data from part $(iv)$ of Lemma \ref{Dspc} and the expression from our part $(iii)$. This completes the proof of the lemma.
\end{proof}
Note that as in part $(iv)$ of Lemma \ref{Dspc}, part $(v)$ of Lemma \ref{iotaexp} holds for any generalized cocharge tableau, not necessarily an element of $\operatorname{CCT}(\lambda)$. By taking complements in Remark \ref{iotawelldef} and Lemma \ref{iotaexp}, we get the equalities $\operatorname{Asi}(\iota S)=\operatorname{Asi}(S)\cup\{n\}$, $\operatorname{Asi}(\tilde{\iota}S)=\operatorname{Asi}(S)\cup\{1\}$, $\operatorname{Asi}^{c}(\tilde{\iota}S)=\operatorname{Asi}^{c}(S)\cup\{n\}$, and $\operatorname{Asp}^{c}(\hat{\iota}C)=\operatorname{Asp}^{c}(C)\cup\{n\}$.

\begin{ex}
Take $w$ as in Example \ref{exev}, with the corresponding tableaux $P(w)$, $\tilde{Q}(w)$, and its $\operatorname{ct}$-image as given there. Then the tableaux $\iota P(w)$, $\tilde{\iota}\tilde{Q}(w)$, and $\hat{\iota}\operatorname{ct}(S)$ are \[\begin{ytableau} 1 & 4 & 6 & 8 & 9 \\ 2 & 5 & 7 \\ 3 \end{ytableau},\ \ \begin{ytableau} 1 & 2 & 4 & 5 & 9 \\ 3 & 6 & 7 \\ 8 \end{ytableau},\mathrm{\ \ and\ \ }\begin{ytableau} 0 & 0 & 1 & 1 & 3 \\ 1 & 2 & 2 \\ 3 \end{ytableau}\] respectively. Recalling that $\tilde{Q}(w)=\operatorname{ev}(S)$ for $S$ from Example \ref{exct}, with $\operatorname{Dsi}(S)=\{2,4,7\}$, we get that this set also equals $\operatorname{Dsi}^{c}\big(\tilde{Q}(w)\big)$ and the $\operatorname{Dsp}^{c}$-set its $\operatorname{ct}$-image. If $J:=\operatorname{Dsi}\big(\tilde{Q}(w)\big)=\{1,4,6\}$, then $\operatorname{Dsi}\big(\tilde{\iota}\tilde{Q}(w)\big)=\{2,5,7\}=J_{+}$, and $\operatorname{Dsi}^{c}\big(\tilde{\iota}\tilde{Q}(w)\big)$, as well as the $\operatorname{Dsp}^{c}$-set of the rightmost tableau, again equal $\{2,4,7\}$, as before the application of $\tilde{\iota}$ or $\hat{\iota}$. \label{iotatabs}
\end{ex}
Note that the rows of the cocharge tableau from Example \ref{exev} and of its $\hat{\iota}$-image in Example \ref{iotatabs} are the same except for the first one, and by knowing the type (or the $\operatorname{Dsi}^{c}$-set), in both of them we must have two 1's and one 3 in the first row. The increasing of $n$ from 8 to 9 and the addition of 0 exemplifies part $(v)$ of Lemma \ref{iotaexp}.

\medskip

We now obtain the effect of taking a permutation from $S_{n}$ into $S_{n+1}$ on the associated higher Specht polynomial.
\begin{prop}
For $w \in S_{n}$, let $w_{+} \in S_{n+1}$ be as in Definition \ref{iotadef}. Then the higher Specht polynomial $F_{w_{+}}$, which is a scalar multiple of $\varepsilon_{P(w_{+})}p_{P(w_{+})}^{\tilde{Q}(w_{+})}$ via Definition \ref{Spechtdef}, can also be obtained as the same multiple of $\varepsilon_{P(w_{+})}$ acting on the monomial $p_{P(w)}^{\tilde{Q}(w)}$. Moreover, substituting $x_{n+1}=0$ in $F_{w_{+}}$ yields the original higher Specht polynomial $F_{w}$. \label{incnSpecht}
\end{prop}

\begin{proof}
As the RSK algorithm is obtained by going over $w_{+}$ from the beginning to the end, we get that $P(w_{+})=\iota P(w)$ and $Q(w_{+})=\iota Q(w)$, with the notation from Definition \ref{iotadef}. It follows that $\tilde{Q}(w_{+})=\tilde{\iota}\tilde{Q}(w)$ as well, and we denote the common shape of $P(w)$, $Q(w)$, and $\tilde{Q}(w)$ by $\lambda$, so that $P(w_{+})$, $Q(w_{+})$, and $\tilde{Q}(w_{+})$ are all of shape $\lambda_{+}$.

This relation between $T:=P(w)$ and $\iota T=P(w_{+})$ shows that the group $C(\iota T)$ is the subgroup $\{u_{+}\;|\;u \in C(T)\}$ via Definition \ref{iotadef}, and the proof of Proposition \ref{Spechtpols} implies that $F_{w_{+}}$ is obtained by the element $\sum_{\sigma \in C(\iota T)}\operatorname{sgn}(\sigma)\sigma$, on the sum of all the monomials in the $R(\iota T)$-orbit of $p_{\iota T}^{\tilde{\iota}S}$, where $S:=\tilde{Q}(w)$ and thus $\tilde{Q}(w_{+})=\tilde{\iota}S$.

Consider now the element of $R(\iota T)$ that takes each element of the first row to the one preceding it (and 1 to the last one in that row) and leaves the other rows invariant. Recall from Lemma \ref{iotaexp} that $\operatorname{ct}(\tilde{\iota}S)$ can be expressed as $\hat{\iota}C$ for $C:=\operatorname{ct}(S)$, and note that the explicit formula for the latter tableau there implies that our element of $R(\iota T)$ takes $p_{\iota T}^{\tilde{\iota}S}$ to $p_{T}^{S}$. These polynomials are thus in the same orbit of $R(\iota T)$, from which the first assertion follows.

Now note that if a tableau in the $R(\iota T)$-orbit of $\hat{\iota}C$ contains 0 in the box $\lambda_{+}\setminus\lambda$ then it is obtained from a tableau in the $R(T)$-orbit of $C$ by adding the said 0, and otherwise the corresponding monomial is divisible by $x_{n+1}$. Moreover, our expression for $C(\iota T)$ as the images of $C(T)$ implies that they all fix $n+1$.

Therefore, applying $\sum_{\sigma \in C(\tilde{T})}\operatorname{sgn}(\sigma)\sigma$ to the sum of monomials not involving $x_{n+1}$ yields $F_{w}$ (since the operator is the same as $\sum_{\sigma \in C(T)}\operatorname{sgn}(\sigma)\sigma$), and when we apply it to a monomial that is divisible by $x_{n+1}$ then the result is also divisible by $x_{n+1}$. As this application produces $F_{w_{+}}$, and substituting $x_{n+1}=0$ leaves only the terms of the first type, the second assertion is established as well. This proves the proposition.
\end{proof}

Proposition \ref{incnSpecht} is another indication that our normalization and convention, with the scalar multiples, is the most natural one for higher Specht polynomials associated with permutations. Indeed, it is easy to verify that $s_{\tilde{T}}^{\tilde{S}}$ is a non-trivial multiple of $s_{T}^{S}$ (by the number of zeros at the first row of $\operatorname{ct}(\tilde{S})=\hat{\iota}\operatorname{ct}(S)$, which is at least 2 by the expressions from Lemma \ref{iotaexp}), and the division by these factors in Definition \ref{Spechtdef} is what makes Proposition \ref{incnSpecht} work as stated. Moreover, the fact that $F_{w}$ and $F_{w_{+}}$ in that proposition clearly have the same degree is related to the preservation of the parameter from Proposition \ref{evkeepsdeg}.
\begin{ex}
In the first case from Example \ref{FTScoeff2}, we have $\lambda_{+}=41\vdash5$, $w_{+}=21345 \in S_{5}$, and \[\iota T:=\begin{ytableau} 1 & 3 & 4 & 5 \\ 2 \end{ytableau},\ \tilde{\iota}S:=\begin{ytableau} 1 & 2 & 3 & 4 \\ 5 \end{ytableau},\mathrm{\ and\ }\hat{\iota}C:=\begin{ytableau} 0 & 0 & 0 & 0 \\ 1 \end{ytableau}.\] The monomial is again $x_{2}$, and the higher Specht polynomial, which we can write as $F_{\iota T}^{\tilde{\iota}S}=F_{\hat{\iota}C,\iota T}=F_{w_{+}}$ and it is a Specht polynomial, coincides with $F_{T}^{S}=F_{C,T}=F_{w}=x_{2}-x_{1}$. In fact, the original elements are also obtained via this construction from the element $213 \in S_{3}$ and tableaux of shape $21\vdash3$. In the second case, with $w_{+}=241536 \in S_{6}$ and $\lambda_{+}=42\vdash6$, we get the tableaux \[\iota T:=\begin{ytableau} 1 & 3 & 5 & 6 \\ 2 & 4 \end{ytableau},\ \ \tilde{\iota}S:=\begin{ytableau} 1 & 2 & 4 & 6 \\ 3 & 5 \end{ytableau},\mathrm{\ and\ }\hat{\iota}C:=\begin{ytableau} 0 & 0 & 1 & 2 \\ 1 & 2 \end{ytableau}.\] Now $s_{\iota T}^{\tilde{\iota}S}=s_{\hat{\iota}C,\iota T}=2$, and $F_{\iota T}^{\tilde{\iota}S}=F_{\hat{\iota}C,\iota T}=F_{w_{+}}$ is the following expression, given in terms of $F_{T}^{S}=F_{C,T}=F_{w}$: Replace the multipliers $x_{5}^{2}$ and $x_{5}$ by $x_{5}^{2}+x_{6}^{2}$ and $x_{5}+x_{6}$, leave the last term as it is, and add the polynomial \[(x_{5}^{2}x_{6}+x_{5}x_{6}^{2})[x_{2}x_{4}^{2}-x_{1}x_{4}^{2}-x_{2}x_{3}^{2}+x_{1}x_{3}^{2}+x_{2}^{2}x_{4}-x_{1}^{2}x_{4}-x_{2}^{2}x_{3}+x_{1}^{2}x_{3}].\] \label{Fwplus}
\end{ex}

\medskip

The compatibility from Proposition \ref{incnSpecht} yields the following construction. We also saw in the proof of Proposition \ref{incnSpecht} that $\operatorname{Dsl}(w)=\operatorname{Dsl}(w_{+})$ in the notation from that lemma, meaning that the set $\operatorname{Dsl}(\hat{w})$ is well-defined and finite for every element $\hat{w}$ of the direct limit $S_{\infty}:=\bigcup_{n=1}^{\infty}S_{n}$ (where each $S_{n}$ is contained in $S_{n+1}$ via $w \mapsto w_{+}$, and we consider the permutations in all of them as acting on the full set $\mathbb{N}$, fixing all but finitely many integers).
\begin{thm}
For any $\hat{w} \in S_{\infty}$ there is a power series $F_{\hat{w}}$ in infinitely many variables, with the following properties:
\begin{enumerate}[$(i)$]
\item It is homogeneous of degree $d=\sum_{i\in\operatorname{Dsl}(\hat{w})}i$.
\item If $n$ is large enough and $w_{n} \in S_{n}$ whose image in $S_{\infty}$ is $\hat{w}$, then substituting $x_{m}=0$ for all $m>n$ in $F_{\hat{w}}$ yields the higher Specht polynomial $F_{w_{n}}$ from Definition \ref{Spechtdef}.
\item If $T_{n}:=P(w_{n})$ for such $n$, then the image in $S_{\infty}$ of the extended column subgroup $\tilde{C}(T_{n})$ from Lemma \ref{tildeCT} acts on $F_{\hat{w}}$ via the character $\widetilde{\operatorname{sgn}}$ showing up in that lemma.
\item For $n$ as in part $(ii)$, any element of $S_{\infty}$, and even of $S_{\mathbb{N}}$, that fixes all the numbers from $\mathbb{N}_{n}$ leaves $F_{\hat{w}}$ invariant.
\end{enumerate}
Moreover, for large enough $n$, properties $(ii)$ and $(iv)$ determine $F_{\hat{w}}$ from $F_{w_{n}}$. \label{stabSpecht}
\end{thm}

\begin{proof}
Take some large enough $n$, and then $w_{n+1}$ equals $(w_{n})_{+}$ in the notation from Definition \ref{iotadef}. The construction from Proposition \ref{incnSpecht} then produces $F_{w_{n+1}}$ from $F_{w_{n}}$ in our notation. We saw that the monomials not containing $x_{n+1}$ there are those from $F_{w_{n}}$, and note that any monomial that is not divisible by one of the variables lying in a length 1 column of $\lambda_{+}$ is related, by the action of the subgroup $S_{t_{1}}\subseteq\tilde{C}(T_{n+1})$ from part $(i)$ of Lemma \ref{tildeCT}, for $T_{n+1}:=P(w_{n+1})$, to a monomial not divisible by $x_{n+1}$. Hence the part of $F_{w_{n+1}}$ that is supported on these monomials is determined by $F_{w_{n}}$.

Now, if we set $S^{n}:=\tilde{Q}(w_{n})$, then $S^{n+1}=\tilde{\iota}S^{n}$ and thus part $(ii)$ of Lemma \ref{iotaexp} yields $\operatorname{Dsi}^{c}(S^{n+1})=\operatorname{Dsi}^{c}(S^{n})$ and we get $\operatorname{cc}(S^{n+1})=\operatorname{cc}(S^{n})$ by the proof of part $(iii)$ of Lemma \ref{Dspc}. Equivalently, these sets equal $\operatorname{Dsi}(w_{n+1})$ and $\operatorname{Dsi}(w_{n})$, which are the same as $\operatorname{Dsl}(\hat{w})$, and therefore the joint number is the asserted value of $d$ and the common degree of $F_{w_{n}}$ and $F_{w_{n+1}}$ by Proposition \ref{evkeepsdeg}. Moreover, increasing $n$ repeatedly adds, by parts $(iii)$ and $(iv)$ of Lemma \ref{iotaexp}, more and more zeros to the first row, so that for large enough $n$ every monomial will not be divisible by at least one variable lying in a column of length 1 in $T_{n+1}$.

Next, $C(T_{n+1})$ is the image of $C(T)$ inside $S_{n+1}$ (hence they all have the same image in $S_{\infty}$), but the group arising from $T_{n+1}$ by part $(i)$ of Lemma \ref{tildeCT} contains the image of the one obtained from $T_{n}$. Hence the image inside $S_{\infty}$ of $\tilde{C}(T_{n+1})$ contains that of $\tilde{C}(T_{n})$, and the difference is just the increasing of the permutation group acting on the columns of length 1 (with the natural extension of $\widetilde{\operatorname{sgn}}$). Moreover, once $\hat{w}$ is the image of $w_{n} \in S_{n}$, and $m>n$, all the entries that are larger than $n$ appear in columns of length 1 in the first row of $T_{m}$. It follows that $\tilde{C}(T_{m})$ contains every element of $S_{m}$ that fixes all the numbers from $\mathbb{N}_{n}$, with no restrictions on how it acts on the larger indices.

So we define $F_{\hat{w}}$ to be the following power series. For every monomial, let $l$ be the maximal index of a variable showing up in it, and then we put that monomial in $F_{\hat{w}}$ with the coefficient with which it shows up in $F_{w_{n}}$ for some $n \geq l$ (this is independent of $n$, since the difference between $F_{w_{n}}$ and $F_{w_{n+1}}$ is only based on monomials that are divisible by $x_{n+1}$, which are not our monomial). This coefficient vanishes unless the monomial is of degree $d$, so that $F_{\hat{w}}$ is homogeneous of degree $d$, yielding property $(i)$.

Substituting $x_{n+1}=0$ removes all the monomials that show up in $F_{w_{n+1}}$ but not in $F_{w_{n}}$. Hence the substitution from property $(ii)$ eliminates all the monomials that are added to $F_{w_{n}}$ at a later stage, and the remaining monomials produce $F_{w_{n}}$, as desired. Next, take two monomials showing up in $F_{\hat{w}}$ that are related by the image of $\tilde{C}(T_{n})$ for some $n$. Both of them show up in $F_{w_{m}}$ for some large $m$, and then our element from $\tilde{C}(T_{n})$ also lies in the image of $\tilde{C}(T_{m})$, with the same $\widetilde{\operatorname{sgn}}$-value. The behavior of $F_{w_{m}}$ under $\tilde{C}(T_{m})$, as given in Proposition \ref{Spechtpols}, shows that the coefficients of these two monomials are related as desired in property $(iii)$.

We now take two monomials in $F_{\hat{w}}$ that are connected by the action of an element $\sigma$ as in property $(iv)$. Once again both of them show up in $F_{w_{m}}$ for large enough $m$, and note that if $n<j \leq m$ is such that $x_{j}$ shows up in the first monomial then $n<\sigma(j) \leq m$ because $x_{j}$ appears in the second monomial. Since $\sigma$ fixes the integers from $\mathbb{N}_{n}$, and its images on the remaining integers in $\mathbb{N}_{m}$ do not affect its action on the monomial, we can choose them arbitrarily (but to be different!) and construct an element of $S_{m}$ which also relates our monomial.

But our element lies in $\tilde{C}(T_{m})$ by our construction, and even in the group $S_{t_{1}}$ from part $(i)$ of Lemma \ref{tildeCT} (so that its $\widetilde{\operatorname{sgn}}$-value is $+1$). Hence the action from Proposition \ref{Spechtpols} implies that the two monomials show up in $F_{w_{m}}$, hence also in $F_{\hat{w}}$, with the same coefficient, establishing property $(iv)$. Finally, we saw in the first paragraph that $F_{w_{n}}$ determines part of $F_{w_{n+1}}$ via the action of $S_{t_{1}}$, and the second paragraph implies that when $n$ is large enough, this part of $F_{w_{n+1}}$ is the entire polynomial. Thus properties $(ii)$ and $(iv)$ show how $F_{w_{n}}$ determines $F_{w_{m}}$ for all $m>n$, and as this yields the data about all the monomials showing up in $F_{\hat{w}}$, it determines this power series completely. This proves the theorem.
\end{proof}
Note that in the proof of Theorem \ref{stabSpecht}, for small enough $n$ the value of $F_{w_{n}}$ does not yet determine that of $F_{w_{n+1}}$ in this manner, since while the monomials from that proof indeed show up in $F_{\tilde{w}}$ in general, the latter may also contain some monomials that are divisible by $x_{n+1}$ and all its $\tilde{C}(\tilde{T})$-images (we saw this in the expression that we had to add at the end of Example \ref{Fwplus}). In any case, we may call the power series from Theorem \ref{stabSpecht} the \emph{stable higher Specht polynomial} associated with $\hat{w} \in S_{\infty}$.

\begin{ex}
We saw in Examples \ref{FTScoeff2} and \ref{Fwplus} that when $w=2134 \in S_{4}$ and $w_{+}=21345 \in S_{5}$, the (classical) associated Specht polynomials coincide, namely $F_{w_{+}}=F_{w}=x_{2}-x_{1}$. Hence if $\hat{w}$ is the image of these permutations in $S_{\infty}$, then $F_{\hat{w}}$ is the same polynomial. However, the case with $w=24153 \in S_{5}$ and $w_{+}=241536 \in S_{6}$ produces as $F_{w}$ a sum of one polynomial in $\mathbb{Q}[\mathbf{x}_{4}]$, another one multiplied by $x_{5}$, and a third one multiplied by $x_{5}^{2}$, and for $F_{w_{+}}$ we replaced $x_{5}$ by $x_{5}+x_{6}$ and $x_{5}^{2}$ by $x_{5}^{2}+x_{6}^{2}$ and added yet another polynomial times $x_{5}^{2}x_{6}+x_{5}x_{6}^{2}$. By setting $\hat{w}$ to now be the common image of these elements in $S_{\infty}$, then for $F_{\hat{w}}$ we replace the (symmetric) polynomials in $x_{5}$ and $x_{6}$ by $\sum_{m=5}^{\infty}x_{m}$, $\sum_{m=5}^{\infty}x_{m}^{2}$, and $\sum_{m=5}^{\infty}\sum_{l=5,\ l \neq m}^{\infty}x_{m}^{2}x_{l}$ respectively. \label{exstab}
\end{ex}

\begin{rmk}
If the tableau $\tilde{Q}(w)$ is $S^{0}$ from Definition \ref{clasSpecht}, so that its $\operatorname{ct}$-image is $C^{0}$, then Lemma \ref{iotaexp} implies that the same happens for $w_{+}$. It follows from the proof of Corollary \ref{quotSpecht} and Proposition \ref{incnSpecht} that both $F_{w_{+}}$ and $F_{w}$ are the polynomial from that definition, so in particular it remains the same polynomial when going from $w$ to $w_{+}$. Therefore the limit power series $F_{\hat{w}}$ from Theorem \ref{stabSpecht} is again the same polynomial in this case. Considering thus a general $\hat{w}$, we may apply the process from Proposition \ref{incnSpecht} also to the quotient from Corollary \ref{quotSpecht}, with the denominator remaining the same, and get a relation between the quotients $Q_{w_{n}}$ and $Q_{w_{n+1}}$ that is similar to that from Proposition \ref{incnSpecht}. The proof of Theorem \ref{stabSpecht} thus allows us to define a \emph{stable higher Specht quotient} $Q_{\hat{w}}$ associated with $\hat{w} \in S_{\infty}$ in a similar manner, that will have the same properties except for invariance under $\tilde{C}(T)$ instead of the character $\widetilde{\operatorname{sgn}}$. \label{stabquot}
\end{rmk}

\begin{ex}
When we consider the quotient from Example \ref{quotex}, the effect of passing from $w \in S_{5}$ to $w_{+} \in S_{6}$ is via the same replacements of $x_{5}$ and $x_{5}^{2}$ as in Example \ref{Fwplus}, and adding the expression $(x_{5}^{2}x_{6}+x_{5}x_{6}^{2})(x_{1}+x_{2}+x_{3}+x_{4})$. The stable version $Q_{\hat{w}}$ from Remark \ref{stabquot} would contain the summand $-2x_{1}x_{2}x_{3}x_{4}$, and in the remaining expressions we will have the multipliers $\sum_{m=5}^{\infty}x_{m}$ instead of $x_{5}$, $\sum_{m=5}^{\infty}x_{m}^{2}$ in the place of $x_{5}^{2}$, and the sum $x_{1}+x_{2}+x_{3}+x_{4}$ will be multiplied by $\sum_{m=5}^{\infty}\sum_{l=5,\ l \neq m}^{\infty}x_{m}^{2}x_{l}$ (corresponding to the fact that $Q_{\hat{w}}$ is the quotient of $F_{\hat{w}}$ from Example \ref{exstab} by the denominator from Example \ref{quotex}). In addition to being invariant under the permutations described in Example \ref{quotex}, it is invariant by any permutation in $S_{\mathbb{N}}$ that leaves 1, 2, 3, and 4 invariant and acts only on the indices $m\geq5$, exemplifying part $(v)$ of Theorem \ref{stabSpecht} in a non-trivial case. \label{quotstab}
\end{ex}
The stable higher Specht polynomials from Theorem \ref{stabSpecht} will be studied further in the sequel \cite{[Z2]}. It may be interesting to see what other properties do the quotients from Corollary \ref{quotSpecht} have, and it is possible that the stable setting from Remark \ref{stabquot}, as showing up in Example \ref{quotstab}, may be a better framework for investigating them.

\section{Decompositions of Representations \label{RepDecom}}

For decomposing representations of $S_{n}$, we recall that the irreducible representation of $S_{n}$ are, up to isomorphism, the Specht modules associated to partitions $\lambda \vdash n$, and we denote the corresponding module by $\mathcal{S}^{\lambda}$. In order to present our construction, as well as state the known results, we introduce the following notation from Definition 3.8 of \cite{[GR]}, and expand it. We recall that $e_{r}$ stands for the $r$th elementary symmetric function in the variables $\{x_{m}\}_{m=1}^{n}$.
\begin{defn}
Fix $\lambda \vdash n$, $S\in\operatorname{SYT}(\lambda)$, and $C\in\operatorname{SYT}(\lambda)$.
\begin{enumerate}[$(i)$]
\item We write $V^{S}$ for the space spanned by the polynomials $F_{T}^{S}$, where $T$ runs over the tableaux of shape $\lambda$ and content $\mathbb{N}_{n}$.
\item We define $V_{C}$ to be the space generated by $F_{T,C}$ for those $T$.
\item Let $h_{r}$ be a non-negative integer for every $1 \leq r \leq n$, and gather them as a vector $\vec{h}:=\{h_{r}\}_{r=1}^{n}$. then we set $V^{S}_{\vec{h}}$ and $V_{C}^{\vec{h}}$ for the image of $V^{S}$ and $V_{C}$ respectively under multiplication by the symmetric polynomial $\prod_{r=1}^{n}e_{r}^{h_{r}}$ inside $\mathbb{Q}[\mathbf{x}_{n}]$.
\end{enumerate} \label{defSpecht}
\end{defn}
The index $r$ of the entries $h_{r}$ of the vector $\vec{h}$ is bounded by $n$ since there is no elementary symmetric function of larger index in $n$ variables, but we may consider each such $\vec{h}$ as an infinite vector with $h_{r}=0$ for all $r>n$ without changing the meaning of $V^{S}_{\vec{h}}$ and $V_{C}^{\vec{h}}$ (this will become more natural in, e.g., Definition \ref{IndtVSVC} below). As in Remark \ref{samepols}, when $C=\operatorname{ct}(S)$ in Definition \ref{defSpecht} we get $V_{C}=V^{S}$ as well as $V_{C}^{\vec{h}}=V^{S}_{\vec{h}}$ for every $\vec{h}$. A basic property of these spaces is the following one.
\begin{thm}
The space $V^{S}$ from Definition \ref{defSpecht} is an irreducible representation of $S_{n}$, on a subspace of $\mathbb{Q}[\mathbf{x}_{n}]$, which is isomorphic to $\mathcal{S}^{\lambda}$, homogeneous of degree $\operatorname{cc}(S)$, and admits $\{F_{T}^{S}\;|\;T\in\operatorname{SYT}(\lambda)\}$ as a basis. The same holds for $V^{S}_{\vec{h}}$, as well as for $V_{C}$ with the basis $\{F_{C,T}\;|\;T\in\operatorname{SYT}(\lambda)\}$, and for $V_{C}^{\vec{h}}$, with the homogeneity degrees $\operatorname{cc}(S)+\sum_{r=1}^{n}rh_{r}$, $\Sigma(C)$, and $\Sigma(C)+\sum_{r=1}^{n}rh_{r}$ respectively. \label{VSVCreps}
\end{thm}

The first assertion in Theorem \ref{VSVCreps} already shows up in \cite{[ATY]} (based on earlier results like \cite{[M]} and \cite{[Pe]}), and is proved explicitly as Corollary 3.17 of \cite{[GR]}, with the homogeneity degree following from Proposition \ref{Spechtpols}. The second one is an immediate consequence, and the other ones are analogous. The main theorem from that paper states that $\bigoplus_{\lambda \vdash n}\bigoplus_{S\in\operatorname{SYT}(\lambda)}V^{S}$ lifts the coinvariant ring $R_{n,n}$ back into $\mathbb{Q}[\mathbf{x}_{n}]$ in an $S_{n}$-equivariant manner (hence so does $\bigoplus_{\lambda \vdash n}\bigoplus_{C\in\operatorname{CCT}(\lambda)}V_{C}$), and thus re-establishes, in particular, the result from \cite{[B]} that this representation is isomorphic, as an ungraded representation, to the regular one of $S_{n}$.

\begin{rmk}
The fact that $V^{S}$ and $V_{C}$ are irreducible, and the only case where they are trivial is where the action is on constants, implies that when we decompose non-zero any element of $V^{S}_{\vec{h}}$ and $V_{C}^{\vec{h}}$ in $\mathbb{Q}[\mathbf{x}_{n}]$, they all have a maximal symmetric divisor, which is the external multiplier $\prod_{r=1}^{n}e_{r}^{h_{r}}$. \label{maxsymdiv}
\end{rmk}

For $1 \leq k \leq n$ the paper \cite{[HRS]} defines the quotient $R_{n,k}$ as the quotient of $\mathbb{Q}[\mathbf{x}_{n}]$ by the ideal generated by $x_{i}^{k}$, $1 \leq i \leq n$ and the symmetric functions $e_{r}$, $n-k+1 \leq r \leq n$ in these variables. Theorem 4.11 there states that it is isomorphic, as an ungraded representation, to the one arising from the action of that group on the set $\mathcal{OP}_{n,k}$ of ordered partitions of $\mathbb{N}_{n}$ into $k$ sets. We now cite Theorem 1.7 of \cite{[GR]}, for introducing the notation that we will use later.
\begin{thm}
Given $\lambda \vdash n$ and $S\in\operatorname{SYT}(\lambda)$, let $H^{S}_{k}$ denote the set of vectors $\vec{h}=\{h_{r}\}_{r=1}^{n}$, with non-negative integral entries, that satisfy $h_{r}=0$ for every $r>n-k$ and the inequality $\sum_{r=1}^{n-k}h_{r}<k-|\operatorname{Dsi}^{c}(S)|$. Then the sum $\bigoplus_{\lambda \vdash n}\bigoplus_{S\in\operatorname{SYT}(\lambda)}\bigoplus_{\vec{h} \in H^{S}_{k}}V^{S}_{\vec{h}}$ is indeed a direct sum, and it lifts the quotient $R_{n,k}$ into $\mathbb{Q}[\mathbf{x}_{n}]$, in an $S_{n}$-equivariant manner. \label{decomRnk}
\end{thm}
The paper \cite{[GR]} uses the bound $k-|\operatorname{Dsi}(S)|$ in the definition of the set from Theorem \ref{decomRnk}, but it is clear that $\operatorname{Dsi}(S)$ and $\operatorname{Dsi}^{c}(S)$ have the same size. By setting $H_{C}^{k}$, for $C\in\operatorname{CCT}(\lambda)$, to be those vectors $\vec{h}$ with $h_{r}=0$ if $r>n-k$ for which $\sum_{r=1}^{n-k}h_{r}<k-|\operatorname{Dsp}^{c}(C)|$, we get the lifting from that theorem expressed as $\bigoplus_{\lambda \vdash n}\bigoplus_{C\in\operatorname{CCT}(\lambda)}\bigoplus_{\vec{h} \in H_{C}^{k}}V_{C}^{\vec{h}}$ (by the relations obtained when $C=\operatorname{ct}(S)$, including the equality $\operatorname{Dsp}^{c}(C)=\operatorname{Dsi}^{c}(S)$). The case $n=k$ in that theorem is easily seen to reproduce the results from \cite{[ATY]}.

\medskip

Apart from the extremal cases $k=1$ and $k=n$, the latter action is not transitive, and orbits of that action are determined by the sizes of the sets in the partition. Therefore $\mathbb{Q}[\mathcal{OP}_{n,k}]$, and with it $R_{n,k}$ and its lift from \cite{[GR]}, admits a natural decomposition into the representations arising from said orbits. We now aim to match the decomposition of $\mathbb{Q}[\mathcal{OP}_{n,k}]$ through the orbits with the expression for lifting from Theorem \ref{decomRnk}, and do it in a way that behaves well with respect to many operations on orbits of different representations.

For this we make the following observation, which essentially shows up in Subsection 4.1 of \cite{[RW]} (see also the relation with the set $\mathcal{W}_{n,k}$ from \cite{[PR]}).
\begin{lem}
Consider a subset $I\subseteq\mathbb{N}_{n-1}$, of size $k-1$, and write the composition $\operatorname{comp}_{n}I \vDash n$ from Lemma \ref{setscomp}, of length $k$, as $\{m_{h}\}_{h=0}^{k-1}$. Then there is a natural bijection between ordered partitions of $\mathbb{N}_{n}$ into $k$ sets of sizes $\{m_{h}\}_{h=0}^{k-1}$ and elements $w \in S_{n}$ for which $\operatorname{Dsl}(w) \subseteq I$. \label{Snparts}
\end{lem}
We chose the indexation in Lemma \ref{Snparts} to be with $0 \leq h \leq k-1$, as in Definition \ref{gencct}, for the normalizations of the constructions below to be more natural. The change of notation from $J$ and the multi-set $\mu$ associated with the composition $\alpha:=\operatorname{comp}_{n}J \vDash n$ to $I$, $\vec{m}$, and $m_{h}$ is due to the fact that in Lemma \ref{ctJinv} we worked with tableaux $S$ such that the set $J$ contains $\operatorname{Dsi}(S)$, while we shall soon work with sets $I$ containing $\operatorname{Dsi}^{c}(S)$ from Definition \ref{Dsic} (one might thus think of $I$ and $J$ as related by $I=\{n-i\;|\;i \in J\}$ and $J=\{n-i\;|\;i \in I\}$).

\begin{proof}
Given an ordered partition of $\mathbb{N}_{n}$ into $k$ sets, we write the elements of every set in the partition in an increasing order, and obtain a permutation $w \in S_{n}$ in one-line notation. When the sizes of these sets are the $m_{h}$'s, the fact that $\{m_{h}\}_{h=0}^{k-1}=\operatorname{comp}_{n}I$ implies that the separations between the sets show up between $i$ and $i+1$ for $i \in I$. Hence if $i$ is one of the $n-k$ locations that are not in the set $I$, then $i$ and $i+1$ represent locations that are in the same set, which we ordered increasingly. This implies, by Definition \ref{asperm}, that all these elements must be in $\operatorname{Asl}(w)$, and therefore the complement $\operatorname{Dsl}(w)$ of $\operatorname{Asl}(w)$ is contained in $I$ as desired.

Conversely, consider an element $w \in S_{n}$ for which $\operatorname{Dsl}(w) \subseteq I$, and write $I=\{i_{h}\}_{h=1}^{k-1}$ in increasing order, completed by setting $i_{0}=0$ and $i_{k}=n$ as usual. Then every subsequence $\{w_{i}\}_{i=i_{h}+1}^{i_{h+1}}$ of the one-line notation for $w$ is increasing by Definition \ref{asperm} and the assumption on $\operatorname{Dsl}(w)$. We thus gather them into a set, and create an ordered partition of $\mathbb{N}_{n}$ into $k$ sets of the prescribed sizes. It is clear that these constructions are inverse to one another, yielding the two sides of the asserted bijection. This proves the lemma.
\end{proof}

We thus make the following definition.
\begin{defn}
For $I\subseteq\mathbb{N}_{n-1}$, of size $k-1$ for $1 \leq k \leq n$, and the length $k$ composition $\vec{m}:=\{m_{h}\}_{h=0}^{k-1}:=\operatorname{comp}_{n}I \vDash n$, we denote by $\mathcal{OP}_{n,I}$ the set of ordered partitions of $\mathbb{N}_{n}$ into $k$ sets with sizes $\{m_{h}\}_{h=0}^{k-1}$ as in Lemma \ref{Snparts}, as well as $W_{\vec{m}}=W_{m_{0},\ldots,m_{k-1}}$ for the corresponding representation $\mathbb{Q}[\mathcal{OP}_{n,I}]$ of $S_{n}$. \label{OPnIdef}
\end{defn}
Lemma \ref{Snparts} and Definition \ref{OPnIdef} are closely related to the \emph{ascent starred permutations} from \cite{[RW]}, \cite{[HRS]}, and others, where the locations of the stars are according to the complement of $I$ in our notation.

It is clear that $\mathcal{OP}_{n,k}$ is the disjoint union of the sets $\mathcal{OP}_{n,I}$ from Definition \ref{OPnIdef}, with $|I|=k-1$, and that the latter are precisely the orbits in the action of $S_{n}$ on the former. Hence $\mathbb{Q}[\mathcal{OP}_{n,k}]$ equals $\bigoplus_{I\subseteq\mathbb{N}_{n-1},\ |I|=k-1}\mathbb{Q}[\mathcal{OP}_{n,I}]$, which means, via Lemma \ref{setscomp}, that this is the same as the direct sum of $W_{\vec{m}}$ over compositions $\vec{m} \vDash n$ with $\ell(\vec{m})=k$.

\begin{ex}
The set $I=\mathbb{N}_{n-1}$ yields as $\mathbb{Q}[\mathcal{OP}_{n,I}]$ (which is $W_{1,\ldots,1}$) the standard representation $\mathbb{Q}[\mathcal{OP}_{n,n}]=\mathbb{Q}[S_{n}]$ (by identifying each element of $S_{n}$ with the sequence of singletons arising from its one-line notation), and every $w \in S_{n}$ satisfies the condition from Lemma \ref{Snparts}. From the empty set $I$ we get as $\mathbb{Q}[\mathcal{OP}_{n,I}]=W_{n}$ the trivial representation, since $\mathcal{OP}_{n,I}=\mathcal{OP}_{n,1}$ consists of a single set $\mathbb{N}_{n}$, and there is only one permutation, namely $w=\operatorname{Id}_{n}$, for which $\operatorname{Dsl}(w)$ is empty.
\label{OPextriv}
\end{ex}
Example \ref{OPextriv}, containing the trivial settings, covers the cases with $n\leq2$. The first non-trivial examples are as follows.
\begin{ex}
Take $n=3$ and $k=2$, so that $I$ is either $\{1\}$ or $\{2\}$. In the former case $\mathcal{OP}_{3,I}$ consists of the ordered partitions $\big(\{1\},\{2,3\}\big)$, $\big(\{2\},\{1,3\}\big)$, and $\big(\{3\},\{1,2\}\big)$, associated with the elements 123, 213, and 312 of $S_{3}$ (the three permutations for which only 1 might be in the $\operatorname{Dsl}$-set), and they span the representation $W_{1,2}$. The elements of $\mathcal{OP}_{3,I}$ in the latter case are $\big(\{1,2\},\{3\}\big)$, $\big(\{1,3\},\{2\}\big)$, and $\big(\{2,3\},\{1\}\big)$, the corresponding permutations are 123, 132, and 231 (with only 2 possibly in the $\operatorname{Dsl}$-set), and the representation is $W_{2,1}$. \label{OPexn3}
\end{ex}

\begin{ex}
When $n=4$ and $k=2$, the singletons $\{1\}$ produce the representation $W_{1,3}$ spanned by the partitions associated with 1234, 2134, 3124, and 4123, from $\{2\}$ we get $W_{2,2}$ with the basis corresponding to 1234, 1324, 1423, 2314, 2413, and 3412, and using $\{3\}$ the associated representation $W_{3,1}$ is generated by the elements related to 1234, 1243, 1342, and 2341. For $n=4$ and $k=3$ we obtain from the sets $\{1,2\}$, $\{1,3\}$, and $\{2,3\}$ the respective representations $W_{1,1,2}$, $W_{1,2,1}$, and $W_{2,1,1}$. \label{OPexn4}
\end{ex}

\medskip

The combination of the construction from \cite{[ATY]} with the modified RSK algorithm from Definition \ref{modRSK} combines to a bijection between the elements of $S_{n}$ and the higher Specht polynomials $F_{w}=F_{T}^{S}=F_{T,C}$ from Definition \ref{Spechtdef} (indeed, Theorem \ref{VSVCreps} allows us to restrict attention to standard $T$ as well), which altogether is, by Theorem \ref{VSVCreps}, a union of bases for the irreducible components $V^{S}$ from Definition \ref{defSpecht} producing $R_{n,n}$. We recall the sets $\operatorname{Dsi}^{c}(S)$ and $\operatorname{Asi}^{c}(S)$ from Definition \ref{Dsic} for a standard Young tableau, and generalize this bijection, as well as Definition \ref{Spechtdef}, as follows.
\begin{defn}
Take $I\subseteq\mathbb{N}_{n-1}$, $w \in S_{n}$, $\lambda \vdash n$, $S\in\operatorname{SYT}(\lambda)$, $T\in\operatorname{SYT}(\lambda)$, and $C\in\operatorname{SYT}(\lambda)$, and assume that $I$ contains $\operatorname{Dsl}(w)$, $\operatorname{Dsi}^{c}(S)$, and $\operatorname{Dsp}^{c}(C)$. For every $i \in I$, let $r_{i}$ be $|\{j\in\mathbb{N}_{n-1} \setminus I\;|\;j<i\}|$.
\begin{enumerate}[$(i)$]
\item We define $\operatorname{Asl}_{I}(w):=I\setminus\operatorname{Dsl}(w)=I\cap\operatorname{Asl}(w)$, and define $F_{w,I}$ to be the product $F_{w}\cdot\prod_{i\in\operatorname{Asl}_{I}(w)}e_{r_{i}}$.
\item Denote the set $I\setminus\operatorname{Dsi}^{c}(S)=I\cap\operatorname{Asi}^{c}(S)$ by $\operatorname{Asi}^{c}_{I}(S)$, and define $F_{T,I}^{S}$ to be $F_{T}^{S}\cdot\prod_{i\in\operatorname{Asi}^{c}_{I}(S)}e_{r_{i}}$. In addition, for $r>0$ set $h_{r}:=\big|\{i\in\operatorname{Asi}^{c}_{I}(S)\;|\;r_{i}=r\}\big|$ and write $\vec{h}^{S}_{I}$ for the vector $\{h_{r}\}_{r=1}^{n}$.
\item Let $\operatorname{Asp}_{I}^{c}(C)$ denote $I\setminus\operatorname{Dsp}^{c}(C)$, and set $F_{C,T}^{I}$ to be $F_{C,T}\cdot\prod_{i\in\operatorname{Asp}^{c}_{I}(C)}e_{r_{i}}$. If $r>0$ then denote $h_{r}:=\big|\{i\in\operatorname{Asp}^{c}_{I}(C)\;|\;r_{i}=r\}\big|$, and define $\vec{h}_{C}^{I}$ to be the vector $\{h_{r}\}_{r=1}^{n}$.
\end{enumerate} \label{FwIdef}
\end{defn}
Of course, we use the convention $e_{0}=1$ in case $r_{i}=0$ for some $i$ in Definition \ref{FwIdef}, and an empty product there equals 1 as usual. The comparison of the various objects from that definition (extending, e.g., Definition \ref{Spechtdef}), and their relations with Definition \ref{defSpecht}, are as follows.
\begin{lem}
Consider a subset $I\subseteq\mathbb{N}_{n-1}$.
\begin{enumerate}[$(i)$]
\item If $T=P(w)$ and $S=\tilde{Q}(w)$ for $w \in S_{n}$, then $F_{w,I}$ and $F_{T,I}^{S}$ are defined together, and they are equal in this case.
\item When $C=\operatorname{ct}(S)$, the polynomial $F_{C,T}^{I}$ is defined if and only if $F_{T,I}^{S}$ is, and then they are equal.
\item The polynomials $\big\{F_{T,I}^{S}\;|\;T\in\operatorname{SYT}(\lambda)\big\}$, where $S\in\operatorname{SYT}(\lambda)$, form a basis for the representation $V^{S}_{\vec{h}^{S}_{I}}$, which homogeneous of degree $\operatorname{cc}(S)+\sum_{r=1}^{n}rh_{r}$.
\item For $C\in\operatorname{SYT}(\lambda)$ the polynomials $\big\{F_{C,T}^{I}\;|\;T\in\operatorname{SYT}(\lambda)\big\}$, which are homogeneous of degree $\Sigma(C)+\sum_{r=1}^{n}rh_{r}$, span the representation $V_{C}^{\vec{h}_{C}^{I}}$, and are a basis for it.
\end{enumerate} \label{spanreps}
\end{lem}

\begin{proof}
Lemma \ref{RSKcor}, the proof of Proposition \ref{evkeepsdeg}, and the definition of the sets in Definition \ref{Dsic} show that the equalities $\operatorname{Asi}^{c}(S)=\operatorname{Asl}(w)$ and $\operatorname{Dsi}^{c}(S)=\operatorname{Dsl}(w)$ hold for $S:=\tilde{Q}(w)$ from Definition \ref{modRSK}, and we recall that $\operatorname{Dsp}^{c}(C)=\operatorname{Dsi}^{c}(S)$ when $C=\operatorname{ct}(S)$. Hence the condition of being contained in $I$ is the same, and we get $\operatorname{Asi}^{c}_{I}(S)=\operatorname{Asl}_{I}(w)$ under the former condition and $\operatorname{Asp}^{c}_{I}(C)=\operatorname{Asi}^{c}_{I}(S)$ in the latter. The remaining parts of Definition \ref{FwIdef} thus imply parts $(i)$ and $(ii)$, while parts $(iii)$ and $(iv)$ are immediate consequences of that definition, together with Theorem \ref{VSVCreps} and Definition \ref{defSpecht}. This proves the lemma.
\end{proof}
In fact, in Definition \ref{FwIdef} we may take $T$ to be any tableau of shape $\lambda$ and content $\mathbb{N}_{n}$, as in Definition \ref{Spechtdef}, and the first two parts of Lemma \ref{spanreps} continue to hold (with the additional polynomials do not increase the representations from the last two parts there). The objects from Definition \ref{FwIdef} and Lemma \ref{spanreps} take the following form in the settings from Example \ref{OPextriv}.
\begin{ex}
When $I=\mathbb{N}_{n-1}$, the containment condition is satisfied for every $w$, $S$, $T$, and $C$, but as the complement of this set is empty, we get $r_{i}=0$ for every $i$ and thus Definition \ref{FwIdef} and the two latter parts of Lemma \ref{spanreps} reduce to Definition \ref{Spechtdef} and Theorem \ref{VSVCreps} in this case (with $\vec{h}^{S}_{I}$ and $\vec{h}_{C}^{I}$ being the vector of zeros). For the empty set $I$, the restriction from Definition \ref{FwIdef} only leaves $w=\operatorname{Id}_{n}$, $\lambda=n$ (the partition of length 1), and the unique elements $S\in\operatorname{SYT}(\lambda)$ and $C\in\operatorname{CCT}(\lambda)$ (the latter with all the entries 0), and as the products in that definition are empty ones, all the polynomials obtained in this case reduce to the constant 1. \label{trivsets}
\end{ex}

\begin{ex}
When $n=3$ and $k=2$ as in Example \ref{OPexn3}, the element 123 of $S_{3}$ is associated with the tableaux from the end of Example \ref{trivsets}, and the remaining elements for $I=\{1\}$ there satisfy \[\tilde{Q}(213)=\tilde{Q}(312)=\begin{ytableau} 1 & 2 \\ 3 \end{ytableau},\mathrm{\ with\ }\operatorname{ct}=\begin{ytableau} 0 & 0 \\ 1 \end{ytableau}\mathrm{\ \ and\ }\operatorname{Dsi}^{c}=\operatorname{Dsp}^{c}=\{1\},\] while when $I=\{2\}$ there are the elements \[\tilde{Q}(132)=\tilde{Q}(231)=\begin{ytableau} 1 & 3 \\ 2 \end{ytableau},\mathrm{\ where\ }\operatorname{ct}=\begin{ytableau} 0 & 1 \\ 1 \end{ytableau}\mathrm{\ \ with\ }\operatorname{Dsi}^{c}=\operatorname{Dsp}^{c}=\{2\}.\] The $P$-tableaux are these two elements of $\operatorname{SYT}(21\vdash3)$, one for each value of $I$, meaning that in the first case these two elements produce the basis $x_{2}-x_{1}$ and $x_{3}-x_{1}$ for the representation $V^{\substack{12 \\ 3\hphantom{4}}}=V_{\substack{00 \\ 1\hphantom{1}}}$, while the second one yields $(x_{2}-x_{1})x_{3}$ and $(x_{3}-x_{1})x_{2}$, which form a basis for $V^{\substack{13 \\ 2\hphantom{4}}}=V_{\substack{01 \\ 1\hphantom{2}}}$. Finally, note that for $w=123$ the value of $r_{i}$ for the unique element of $\operatorname{Asl}_{I}(w)=I$ is 0 when $I=\{1\}$ and 1 in case $I=\{2\}$. Hence $W_{1,2}$ is the direct sum of the first representation and the trivial representation $V^{123}=V_{000}$ on constants, while $W_{2,1}$ is obtained by adding to the second representation the trivial one on multiples of the symmetric function $e_{1}=x_{1}+x_{2}+x_{3}$. \label{n3kI}
\end{ex}

\begin{ex}
Take now $n=4$ and $k=2$, and consider the expressions from Example \ref{OPexn4}. The elements $w \in S_{4}$ for which $\operatorname{Dsi}(w)$ equals $\{1\}$ or $\{3\}$ are associated via $\tilde{Q}$ and then $\operatorname{ct}$ with the tableaux \[\begin{ytableau} 1 & 2 & 3 \\ 4 \end{ytableau}\mathrm{\ \ and\ \ }\begin{ytableau} 0 & 0 & 0 \\ 1 \end{ytableau}\mathrm{\ \ or\ \ }\begin{ytableau} 1 & 3 & 4 \\ 2 \end{ytableau}\mathrm{\ \ and\ \ }\begin{ytableau} 0 & 1 & 1 \\ 1 \end{ytableau},\] while for $I=\{2\}$ the three permutations 1324, 1423, and 2314 and the two elements 2413 and 3412 give \[\begin{ytableau} 1 & 2 & 4 \\ 3 \end{ytableau}\mathrm{\ \ with\ \ }\begin{ytableau} 0 & 0 & 1 \\ 1 \end{ytableau}\mathrm{\ \ and\ \ }\begin{ytableau} 1 & 2  \\ 3 & 4 \end{ytableau}\mathrm{\ \ with\ \ }\begin{ytableau} 0 & 0 \\ 1 & 1 \end{ytableau}\] by the same operations. Since for $w=1234 \in S_{4}$ the single value of $r_{i}$ is 0, 2, and 1 respectively, we deduce that $W_{1,3}$ is spanned by the basis $x_{2}-x_{1}$, $x_{3}-x_{1}$, and $x_{4}-x_{1}$ of $V^{\substack{123 \\ 4\hphantom{56}}}=V_{\substack{000 \\ 1\hphantom{11}}}$ and 1 for the trivial representation $V^{1234}=V_{0000}$, while for $W_{3,1}$ we get the basis $(x_{2}-x_{1})x_{3}x_{4}$, $(x_{3}-x_{1})x_{2}x_{4}$, and $(x_{4}-x_{1})x_{2}x_{3}$ for $V^{\substack{134 \\ 2\hphantom{56}}}=V_{\substack{011 \\ 1\hphantom{22}}}$, and the symmetric function $e_{2}$. The representation $W_{2,2}$ is spanned by the union of three bases: One is $(x_{2}-x_{1})(x_{3}+x_{4})$, $(x_{3}-x_{1})(x_{2}+x_{4})$, and $(x_{4}-x_{1})(x_{2}+x_{3})$ generating $V^{\substack{124 \\ 3\hphantom{56}}}=V_{\substack{001 \\ 1\hphantom{12}}}$, another one is $(x_{2}-x_{1})(x_{4}-x_{3})$ and $(x_{3}-x_{1})(x_{4}-x_{2})$ spanning $V^{\substack{12 \\ 34}}=V_{\substack{00 \\ 11}}$, and the last one is $e_{1}$. Moving over to $k=3$ and the sets $\{1,2\}$, $\{1,3\}$, and $\{2,3\}$, we have $r_{1}=0$ in the former two cases, $r_{3}=1$ in the latter two cases, and $r_{2}$ is 0 in the first case and 1 in the last one. These considerations show that the corresponding representations $W_{1,1,2}$, $W_{1,2,1}$, and $W_{2,1,1}$ from Example \ref{OPexn4} are described by the expressions \[V^{1234} \oplus V^{\substack{123 \\ 4\hphantom{56}}} \oplus V^{\substack{124 \\ 3\hphantom{56}}} \oplus V^{\substack{12 \\ 34}} \oplus V^{\substack{12 \\ 3\hphantom{5} \\ 4\hphantom{6}}}=V_{0000} \oplus V_{\substack{000 \\ 1\hphantom{11}}} \oplus V_{\substack{001 \\ 1\hphantom{12}}} \oplus V_{\substack{00 \\ 11}} \oplus V_{\substack{00 \\ 1\hphantom{1} \\ 2\hphantom{2}}},\] \[V^{1234}e_{1} \oplus V^{\substack{123 \\ 4\hphantom{56}}}e_{1} \oplus V^{\substack{134 \\ 2\hphantom{56}}} \oplus V^{\substack{13 \\ 24}} \oplus V^{\substack{13 \\ 2\hphantom{5} \\ 4\hphantom{6}}}=V_{0000}e_{1} \oplus V_{\substack{000 \\ 1\hphantom{11}}}e_{1} \oplus V_{\substack{011 \\ 1\hphantom{22}}} \oplus V_{\substack{01 \\ 12}} \oplus V_{\substack{01 \\ 1\hphantom{1} \\ 2\hphantom{2}}},\] and, respectively, \[V^{1234}e_{1}^{2} \oplus V^{\substack{124 \\ 3\hphantom{56}}}e_{1} \oplus V^{\substack{134 \\ 2\hphantom{56}}}e_{1} \oplus V^{\substack{12 \\ 34}}e_{1} \oplus V^{\substack{14 \\ 2\hphantom{5} \\ 3\hphantom{6}}}=V_{0000}e_{1}^{2} \oplus V_{\substack{001 \\ 1\hphantom{12}}}e_{1} \oplus V_{\substack{011 \\ 1\hphantom{22}}}e_{1} \oplus V_{\substack{00 \\ 11}}e_{1} \oplus V_{\substack{02 \\ 1\hphantom{1} \\ 2\hphantom{2}}}.\] \label{n4kI}
\end{ex}

\begin{rmk}
The case $I=\{1\}$ in Examples \ref{n3kI} and \ref{n4kI} extends to the statement that for every $n$, the corresponding representation $W_{1,n-1}$ is the direct sum of the following two representations: One arises from $\lambda=n-1,1 \vdash n$ with $S^{0}$ or $C^{0}$ from Definition \ref{clasSpecht} and consists of linear polynomials with coefficient sum 0, and the second one is the trivial representation on constants. \label{deg1}
\end{rmk}

\begin{rmk}
By comparing the degrees from part $(iii)$ of Lemma \ref{spanreps} with the sets $\mathcal{OP}_{n,I}$ from Definition \ref{OPnIdef} as well as with their union $\mathcal{OP}_{n,k}$ (or via the interpretation through Lemma \ref{Snparts} as ascent starred permutations), we obtain a statistic on the latter set, that is based on the classical major index. It differs, however, from the extension presented in Equation (2.4) of \cite{[HRS]}. It may be interesting to relate it to other existing statistics on ordered set partitions. \label{genmaj}
\end{rmk}

\medskip

The construction from Definition \ref{FwIdef} has the following properties.
\begin{lem}
Given $\lambda \vdash n$, $1 \leq k \leq n$, and $S\in\operatorname{SYT}(\lambda)$, the map taking a set $\operatorname{Dsi}^{c}(S) \subseteq I\subseteq\mathbb{N}_{n-1}$ to $\vec{h}^{S}_{I}$ is a bijection between this collection of sets and the set $H^{S}_{k}$ from Theorem \ref{decomRnk}. Similarly, for such $\lambda$ and $k$ and for $C\in\operatorname{CCT}(\lambda)$, we get a bijection between the subsets $\operatorname{Dsp}^{c}(C) \subseteq I\subseteq\mathbb{N}_{n-1}$ and the elements of $H_{C}^{k}$ by sending $I$ to $\vec{h}_{C}^{I}$. \label{vechSI}
\end{lem}

\begin{proof}
Write $D$ for $\operatorname{Dsi}^{c}(S)$ or $\operatorname{Dsp}^{c}(C)$, and $A$ for $\operatorname{Asi}^{c}_{I}(S)$ or for $\operatorname{Asp}^{c}_{I}(C)$, which is $I \setminus D$ by Definition \ref{FwIdef}. Now, for every $i \in A$ the number $r_{i}$ is the size of a subset of the complement of $I$ inside $\mathbb{N}_{n-1}$, hence it is bounded by $n-1-|I|=n-k$. It follows that $h_{r}=0$ for all $r>n-k$. Moreover, since every $i \in A$ with $r_{i}\geq1$ contributes 1 to the $\sum_{r=1}^{n-k}h_{r}$, the latter sum is bounded by $|A|$ (with this being a bound rather than an equality because $r_{i}$ might vanish for some values of $i$). But as this set is the complement of $D$ inside $I$, this complement is of size $k-1-|D|$ and hence $\sum_{r=1}^{n-k}h_{r}<k-|D|$. Hence $\vec{h}^{S}_{I} \in H^{S}_{k}$ and $\vec{h}_{C}^{I} \in H_{C}^{k}$ for every such set $I$ by definition.

To obtain the bijections, we need to show how to reconstruct $I$ from a given vector in that set. We complete the vector $\vec{h}$ by defining $h_{0}$ such that $\sum_{r=0}^{n-k}h_{r}=k-1-|D|$, which we write as $k-d$ by setting $d:=|D|+1$. Moreover, determining a set $I$ of size $k-1$ that contains $D$ of size $d-1$ is equivalent to determining the complement $A$, namely a subset of size $k-d$ inside the set $\mathbb{N}_{n-1} \setminus D$ (this set is $\operatorname{Asi}^{c}(S)$ or $\operatorname{Asp}^{c}(C)$), whose size is $n-d$. We wish to do this from the values $\{h_{r}\}_{r=0}^{n-k}$, and we observe that while $r_{i}$ is given in terms of the complement of $I$ in $\mathbb{N}_{n-1}$, this complement is the same as that of $A$ inside $\mathbb{N}_{n-1} \setminus D$ because $D \subseteq I$.

We should thus have $k-d$ values $r_{i}$, associated to elements of $A$, counting, for each $i$ there, the number of elements in that complement that are smaller than it. Since this value increases with $i$, we deduce that the smallest $h_{0}$ elements of $A$ have no element of the complement that is smaller than them, the next $h_{1}$ smallest ones have one element in the complement that is smaller than them, and so on. This forces $A$ to contain the smallest $h_{0}$ elements of $\mathbb{N}_{n-1} \setminus D$, then skip one from that set as the values increase, take the next $h_{1}$ smallest elements there, and so on. The set $A$ thus constructed, of size $k-d$, and the resulting set $I$, are then easily seen to be the unique ones for which the associated vector is our $\vec{h}$. This completes the proof of the lemma.
\end{proof}

\begin{ex}
The statement in Lemma \ref{vechSI} is based on the set $D$ from the proof, rather than the tableau $S$ or $C$ (or the partition $\lambda$). We thus take, as an example, the case $n=9$, the subset $D:=\{3,6\}\subseteq\mathbb{N}_{8}$ (hence $d=3$), and $k=7$. The subset $I:=\{1,3,4,5,6,8\}\subseteq\mathbb{N}_{8}$ is of size $k-1=6$ and contains $D$ with complement $A:=\{1,4,5,8\}$, and we have $\mathbb{N}_{8} \setminus I=\{2,7\}$. It follows from Definition \ref{FwIdef} that $r_{1}=0$, $r_{4}=r_{5}=1$, and $r_{8}=2$, and thus $\vec{h}$ is with $h_{1}=2$, $h_{2}=1$, and $h_{r}=0$ for $r>2=n-k$, with the sum $2+1=3$ being smaller than $k-|D|=5$. To reconstruct $I$ from this vector, we need to see that determine $A$ as a subset of size 4 inside $\mathbb{N}_{8} \setminus D=\{1,2,4,5,7,8\}$ using this vector. We complete the vector to have entry sum $k-d=4$ by setting $h_{0}=1$, so that we need $A$ to be with one element having $r$-value 0, another two with $r$-value 1, and finally an element with $r$-value 2. So we put the single smallest element 1 of that complement in $A$ (since $h_{0}=1$), skip the next element 2, then add the next two elements 4 and 5 to $A$ (as $h_{1}=2$), skip the next one 7, and complete with the single remaining element 8 (due to the fact that $h_{2}=1$), which indeed gives the original set $A$. \label{bijex}
\end{ex}
Here is how it works in the previous examples, with small values.
\begin{ex}
When $k=n$ we must take $I=\mathbb{N}_{n-1}$ (and $r_{i}=0$ for every $r$, since $r>0=n-k$), and for $k=1$ and $d=1$ (hence $D=\emptyset$) we must take $I$ to be empty as well, covering the cases from Example \ref{trivsets}. It is also clear that when $k=d=|D|+1$ the only choice is $I=D$ (corresponding to the trivial vector $\vec{h}$, since the sum of its entries must vanish), covering many cases from Examples \ref{n3kI} and \ref{n4kI}. If $k=2$ then the only remaining possibility is $d=1$ (namely $D$ is empty, and $S$ and $C$ associated with $w=\operatorname{Id}_{n}$), and then $I=\{i\}$ is associated with the vector arising from $r_{i}=i-1$, which is the 0 vector if and only if $i=1$, which completes Example \ref{n3kI}, a part of Example \ref{n4kI}, and Remark \ref{deg1}. In the remaining case $n=4$ and $k=3$, where only the value of $h_{1}$ can show up (since $n-k=1$) and it is bounded by $k-d$, the considerations from Lemma \ref{vechSI} and Example \ref{bijex} yield as follows. For $d=2$, the set $D=\{1\}$ yields $I=\{1,2\}$ when $h_{1}=0$ and $I=\{1,3\}$ in case $h_{1}=1$, from $D=\{2\}$ we get $I=\{1,2\}$ if $h_{1}=0$ and $I=\{2,3\}$ when $h_{1}=1$, and in case $D=\{3\}$ we have $I=\{1,3\}$ from $h_{1}=0$ and $I=\{2,3\}$ if $h_{1}=1$, and with the empty set $D$ with $d=1$, the value $h_{1}=0$ yields $I=\{1,2\}$, when $h_{1}=1$ we get $I=\{1,3\}$, and in case $h_{1}=2$ we have $I=\{2,3\}$. \label{bijsmall}
\end{ex}

\medskip

Our main comparison result is now as follows.
\begin{thm}
Let $R_{n,I}$ be the space spanned by the polynomials $F_{w,I}$ from Definition \ref{FwIdef}, for elements $w \in S_{n}$ for which $\operatorname{Dsl}(w) \subseteq I$.
\begin{enumerate}[$(i)$]
\item The space $R_{n,I}$ is a representation of $S_{n}$ that decomposes as the direct sums $\bigoplus_{\lambda \vdash n}\bigoplus_{S\in\operatorname{SYT}(\lambda),\ \operatorname{Dsi}^{c}(S) \subseteq I}V^{S}_{\vec{h}^{S}_{I}}$ and $\bigoplus_{\lambda \vdash n}\bigoplus_{C\in\operatorname{CCT}(\lambda),\ \operatorname{Dsp}^{c}(C) \subseteq I}V_{C}^{\vec{h}_{C}^{I}}$.
\item The representation $R_{n,I}$ is abstractly isomorphic to $\mathbb{Q}[\mathcal{OP}_{n,I}]$.
\item The sum of $R_{n,I}$ over subsets $I\subseteq\mathbb{N}_{n-1}$ of size $k-1$ is a direct sum, which lifts $R_{n,k}$ in $\mathbb{Q}[\mathbf{x}_{n}]$.
\end{enumerate} \label{main}
\end{thm}

\begin{proof}
The fact that $R_{n,I}$ is the sum of the representations from part $(i)$, using either $\operatorname{SYT}(\lambda)$ or $\operatorname{CCT}(\lambda)$, follows directly from Lemma \ref{spanreps}. Lemma \ref{vechSI} now shows that by summing over such $I$ we obtain precisely the sum from Theorem \ref{decomRnk} (or its analogue using cocharge tableaux). As the latter sum is direct by that theorem, we deduce part $(iii)$, and the direct sum property from part $(i)$ is an immediate consequence.

Now, Theorem \ref{VSVCreps} shows that every direct summand $V^{S}_{\vec{h}^{S}_{I}}$ from part $(i)$, where $S\in\operatorname{SYT}(\lambda)$ for some $\lambda \vdash n$, is isomorphic to the Specht module $\mathcal{S}^{\lambda}$. Thus the multiplicity of $\mathcal{S}^{\lambda}$ in $R_{n,I}$ is the number such tableaux $S$ for which $\operatorname{Dsi}^{c}(S) \subseteq I$. But by setting $J:=\{n-i\;|\;i \in I\}$ we get that this containment is equivalent to $\operatorname{Dsi}(S) \subseteq J$, via Definition \ref{Dsic}. Hence if $\alpha=\operatorname{comp}_{n}J \vDash n$ as in Lemma \ref{setscomp}, then Lemma \ref{ctJinv} shows that this multiplicity is $K_{\lambda,\alpha}$.

We now recall that $\mathcal{OP}_{n,I}$ is an orbit on the action of $S_{n}$ on $\mathcal{OP}_{n,k}$, so that $\mathbb{Q}[\mathcal{OP}_{n,I}]$ is the representation induced by the trivial one of the stabilizer of an element. If we take that element to be the one associated with $w=\operatorname{Id}_{n}$ via Lemma \ref{Snparts}, and write $\operatorname{comp}_{n}I \vDash n$ as $\vec{m}:=\{m_{h}\}_{h=0}^{k-1}$ as in Definition \ref{OPnIdef} (so that $\mathbb{Q}[\mathcal{OP}_{n,I}]=W_{\vec{m}}$), then the stabilizer in question is $\prod_{h=0}^{k-1}S_{m_{h}}$. Since $\alpha$ is obtained from $\vec{m}$ by inverting the order of the entries in the composition (this is clear from the relation between $I$ and $J$), we deduce that $W_{\vec{m}}$ is isomorphic to the representation $M^{\alpha}$ in the notation from \cite{[Sa]}.

But then the multiplicity of $\mathcal{S}^{\lambda}$ in $W_{\vec{m}} \cong M^{\alpha}$ is determined by Theorem 2.11.2 of \cite{[Sa]} to also be the Kostka number $K_{\lambda,\alpha}$. Hence $\mathbb{Q}[\mathcal{OP}_{n,I}]=W_{\vec{m}}$ and $R_{n,I}$ contain each Specht module with the same multiplicity, and they are thus isomorphic as representations of $S_{n}$, yielding part $(ii)$. This completes the proof of the theorem.
\end{proof}

\begin{ex}
The case $I=\mathbb{N}_{n-1}$ and $k=n$ in Theorem \ref{main}, for which part $(iii)$ there yields $R_{n,I}=R_{n,n}=R_{n}$, reduces to the regular representation from \cite{[B]} and the isomorphism from \cite{[ATY]}, via the considerations from Example \ref{trivsets}. Considering the homogeneity degree of each part in this case yields, as a consequence, Theorem 7.2 of \cite{[CF]}. The other case there, of empty $I$ and $k=1$, produces the trivial representation on the constants as associated with the trivial action of $S_{n}$ on a singleton. \label{trivthm}
\end{ex}
Example \ref{trivthm} concerns all the cases where the lift of $R_{n,k}$ is a single representation of the sort $R_{n,I}$ from Theorem \ref{main}. As usual, here are the smallest situations where this is no longer the case.
\begin{ex}
For $k=2$ and $I=\{1\}$, the representation $R_{n,I}$ arising from Examples \ref{n3kI} and \ref{n4kI} for $n=3$ and $n=4$ and from Remark \ref{deg1} in general is closely related to the standard, permutation representation of $S_{n}$. The other representation in Example \ref{n3kI}, as well as that associated with $n=4$ and $I=\{3\}$ in Example \ref{n4kI}, is a different copy of an isomorphic representation that shows up inside $R_{n,2}$ (in fact, Remark \ref{deg1} can be extended by the statement that for any $n$, the case $k=2$ and $I=\{n\}$ yields such a representation). The representation $R_{4,3}$ is the direct sum of the last three representations showing up at the end of Example \ref{n4kI}. \label{stdrep}
\end{ex}
Note that the two representations obtained via Example \ref{n3kI} were seen in Example \ref{stdrep} to be isomorphic, as are the three components inside $R_{4,3}$ there. However, as we saw in Example \ref{n4kI}, the remaining set $I=\{2\}$ when $n=4$ and $k=2$ produces a component that is not isomorphic to the other two (because of the additional irreducible component $V^{\substack{12 \\ 34}}=V_{\substack{00 \\ 11}}$), so that Example \ref{stdrep} also contains the first situation where the direct sum from part $(iii)$ of Theorem \ref{main} is not of isomorphic representations.

\medskip

Recall from Lemma \ref{spanreps} that the irreducible representations $V^{S}_{\vec{h}}$ and $V_{C}^{\vec{h}}$ are all homogeneous. Indeed, since the action of $S_{n}$ preserves degrees, in any sub-representation of $\mathbb{Q}[\mathbf{x}_{n}]$ that is graded, a natural decomposition would be into homogeneous components. While the representations $R_{n,I}$ are reducible, and are indeed not homogeneous in general, they are based on a transitive action, and one might ask whether they admit a natural homogeneous analogue. We now show that they do.
\begin{defn}
Given $I$, $w$, $T$, $S$, and $C$ as in Definition \ref{FwIdef}, with the sets $\operatorname{Asl}_{I}(w)$, $\operatorname{Asi}^{c}_{I}(S)$, and $\operatorname{Asp}_{I}^{c}(C)$. We then define $F_{w,I}^{\mathrm{hom}}:=F_{w}\cdot\prod_{i\in\operatorname{Asl}_{I}(w)}e_{i}$ as well as $F_{T,I}^{S,\mathrm{hom}}:=F_{T}^{S}\cdot\prod_{i\in\operatorname{Asi}^{c}_{I}(S)}e_{i}$ and $F_{C,T}^{I,\mathrm{hom}}:=F_{C,T}\cdot\prod_{i\in\operatorname{Asp}_{I}^{c}(C)}e_{i}$, and let $\vec{h}(I,S)$ and $\vec{h}(C,I)$ be the characteristic vectors of the sets $\operatorname{Asi}^{c}_{I}(S)$ and $\operatorname{Asp}_{I}^{c}(C)$ respectively, in which $h_{r}$ equals 1 in case $r$ is in that set, and is 0 otherwise. \label{hompols}
\end{defn}
Definition \ref{hompols} can also be extended by allowing $T$ to be a tableau of shape $\lambda$ and content $\mathbb{N}_{n}$ that is not necessarily standard. The notions from Definition \ref{hompols} are related by the following analogue of Lemma \ref{spanreps}.
\begin{lem}
Take again some subset $I$ of $\mathbb{N}_{n-1}$.
\begin{enumerate}[$(i)$]
\item When $w \in S_{n}$ and $T$ and $S$ are related by $T=P(w)$ and $S=\tilde{Q}(w)$ for $w \in S_{n}$, we have $F_{w,I}^{\mathrm{hom}}=F_{T,I}^{S,\mathrm{hom}}$, with both sides defined together.
\item In case $C=\operatorname{ct}(S)$, the equality $F_{C,T}^{I,\mathrm{hom}}=F_{T,I}^{S,\mathrm{hom}}$, where again both sides are defined together.
\item If $S\in\operatorname{SYT}(\lambda)$ is with $\operatorname{Dsi}^{c}(S) \subseteq I$, then the representation $V^{S}_{\vec{h}(I,S)}$ is homogeneous of degree $\sum_{i \in I}i$, and it admits the set $\big\{F_{T,I}^{S,\mathrm{hom}}\;|\;T\in\operatorname{SYT}(\lambda)\big\}$ as a basis.
\item Given $C\in\operatorname{SYT}(\lambda)$ for which $\operatorname{Dsp}^{c}(C) \subseteq I$, the representation $V_{C}^{\vec{h}(C,I)}$ is also homogeneous of degree $\sum_{i \in I}i$, and $\big\{F_{C,T}^{I,\mathrm{hom}}\;|\;T\in\operatorname{SYT}(\lambda)\big\}$ form a basis for it.
\end{enumerate} \label{repshom}
\end{lem}

\begin{proof}
All parts are established just like those of Lemma \ref{spanreps} (with Definition \ref{hompols} replacing Definition \ref{FwIdef}), and the only assertion that remains to be verified is the asserted homogeneity degrees in parts $(iii)$ and $(iv)$. Now, the homogeneity degree of $V^{S}$ is $\operatorname{cc}(S)=\sum_{i\in\operatorname{Dsi}^{c}(S)}i$ by Theorem \ref{VSVCreps}, Corollary \ref{sumcc} and Definition \ref{Dsic}, while that theorem and Lemma \ref{Dspc} imply that $V_{C}$ is homogeneous of degree $\Sigma(C)=\sum_{i\in\operatorname{Dsp}^{c}(C)}i$. But the additional degree $\sum_{r=1}^{n}rh_{r}$ becomes the entry sum of a set when $\vec{h}$ is the characteristic vector of that set, meaning that when this set is the complement $\operatorname{Asi}^{c}_{I}(S)$ or $\operatorname{Asp}_{I}^{c}(C)$ of our set in $I$, the total degree is $\sum_{i \in I}i$ as desired. This proves the lemma.
\end{proof}

This is the homogeneous version of Theorem \ref{main}.
\begin{prop}
We define $R_{n,I}^{\mathrm{hom}}$ to be the space that is generated by the polynomials $F_{w,I}^{\mathrm{hom}}$ from Definition \ref{hompols}, where $w$ goes over elements of $S_{n}$ that satisfy $\operatorname{Dsl}(w) \subseteq I$.
\begin{enumerate}[$(i)$]
\item The space $R_{n,I}^{\mathrm{hom}}$ is also a representation of $S_{n}$, whose decomposition is $\bigoplus_{\lambda \vdash n}\bigoplus_{S\in\operatorname{SYT}(\lambda),\ \operatorname{Dsi}^{c}(S) \subseteq I}\!V^{S}_{\vec{h}(I,S)}\!=\!\bigoplus_{\lambda \vdash n}\bigoplus_{C\in\operatorname{CCT}(\lambda),\ \operatorname{Dsp}^{c}(C) \subseteq I}\!V_{C}^{\vec{h}(C,I)}$.
\item The representation $R_{n,I}^{\mathrm{hom}}$ is also isomorphic to $\mathbb{Q}[\mathcal{OP}_{n,I}]$.
\item The representation $R_{n,I}^{\mathrm{hom}}$ is on homogeneous polynomials of degree $\sum_{i \in I}i$.
\end{enumerate} \label{homrep}
\end{prop}

\begin{proof}
The expression for the sum in part $(i)$ and the deduction of part $(ii)$ from this direct sum are proved like those of Theorem \ref{main}. Part $(iii)$ follows directly from part $(i)$ by the homogeneity degree from parts $(iii)$ and $(iv)$ of Lemma \ref{repshom}.
\end{proof}

\begin{rmk}
The direct sum property from part $(i)$ of Proposition \ref{repshom}, which is important for part $(ii)$ there, can be obtained directly, but will be established in a different manner in the sequel \cite{[Z1]}, which will make a more extensive use of these representations for some applications. Note that the analogue of Lemma \ref{vechSI} for the vectors from Definition \ref{hompols} is easy, so that if we set $R_{n,k}^{\mathrm{hom}}$ to be the sum of the representations $R_{n,I}^{\mathrm{hom}}$ over the subsets $I\subseteq\mathbb{N}_{n-1}$ with $|I|=k-1$, it is a direct sum as in part $(iii)$ of Theorem \ref{main} (this will also be proved and expanded in that reference). \label{Rnkhom}
\end{rmk}

\begin{ex}
When $I$ is empty, the representation $R_{n,I}^{\mathrm{hom}}$ is again the trivial one on constants, like $R_{n,I}$ from Example \ref{trivthm}. If $I=\{1\}$, the fact that the trivial representation in $R_{n,I}^{\mathrm{hom}}$ is now multiplied by $e_{1}$ (rather that remaining constant as in $R_{n,I}$ showing up in Example \ref{stdrep}) shows that this representation is precisely the standard one. For a more general set $I=\{i\}$ with $k=2$, we obtain $R_{n,I}^{\mathrm{hom}}$ from $R_{n,I}$ by replacing the multiplier of the trivial representation from $e_{i-1}$ that was obtained in Example \ref{bijsmall} to $e_{i}$. When $n=4$ and $k=3$, we recall the expressions for $W_{1,1,2}$, $W_{1,2,1}$, and $W_{2,1,1}$ from Example \ref{n4kI}, which are $R_{4,I}$ for the respective sets $\{1,2\}$, $\{1,3\}$, and $\{2,3\}$. The analogous representations $R_{4,I}^{\mathrm{hom}}$ for these sets $I$ are given by \[V^{1234}e_{1,2} \oplus V^{\substack{123 \\ 4\hphantom{56}}}e_{2} \oplus V^{\substack{124 \\ 3\hphantom{56}}}e_{1} \oplus V^{\substack{12 \\ 34}}e_{1} \oplus V^{\substack{12 \\ 3\hphantom{5} \\ 4\hphantom{6}}}\!\!\!=\!V_{0000}e_{1,2} \oplus V_{\substack{000 \\ 1\hphantom{11}}}e_{2} \oplus V_{\substack{001 \\ 1\hphantom{12}}}e_{1} \oplus V_{\substack{00 \\ 11}}e_{1} \oplus V_{\substack{00 \\ 1\hphantom{1} \\ 2\hphantom{2}}},\] \[V^{1234}e_{1,3} \oplus V^{\substack{123 \\ 4\hphantom{56}}}e_{3} \oplus V^{\substack{134 \\ 2\hphantom{56}}}e_{1} \oplus V^{\substack{13 \\ 24}} \oplus V^{\substack{13 \\ 2\hphantom{5} \\ 4\hphantom{6}}}=V_{0000}e_{1,3} \oplus V_{\substack{000 \\ 1\hphantom{11}}}e_{3} \oplus V_{\substack{011 \\ 1\hphantom{22}}}e_{1} \oplus V_{\substack{01 \\ 12}} \oplus V_{\substack{01 \\ 1\hphantom{1} \\ 2\hphantom{2}}},\] and, respectively, \[V^{1234}e_{2,3} \oplus V^{\substack{124 \\ 3\hphantom{56}}}e_{3} \oplus V^{\substack{134 \\ 2\hphantom{56}}}e_{2} \oplus V^{\substack{12 \\ 34}}e_{3} \oplus V^{\substack{14 \\ 2\hphantom{5} \\ 3\hphantom{6}}}\!\!\!=\!V_{0000}e_{2,3} \oplus V_{\substack{001 \\ 1\hphantom{12}}}e_{3} \oplus V_{\substack{011 \\ 1\hphantom{22}}}e_{2} \oplus V_{\substack{00 \\ 11}}e_{3} \oplus V_{\substack{02 \\ 1\hphantom{1} \\ 2\hphantom{2}}},\] where we shortened $e_{1}e_{2}$, $e_{1}e_{3}$, and $e_{2}e_{3}$, to $e_{1,2}$, $e_{1,3}$, and $e_{2,3}$ respectively. \label{exhom}
\end{ex}
As it is clear that for any $i \not\in I$, the number $r_{i}$ from Definition \ref{FwIdef} is strictly smaller than $i$, we deduce that the statistic from Remark \ref{genmaj} associates with each $w \in S_{n}$ and $\operatorname{Dsl}(w) \subseteq I\subseteq\mathbb{N}_{n-1}$ a value that is bounded by the sum of values from $I$ as in Proposition \ref{homrep}, with equality holding if and only if $\operatorname{Dsl}(w) \subseteq I$. Since every $I\subseteq\mathbb{N}_{n-1}$ equals $\operatorname{Dsl}(w)$ for some $w \in S_{n}$ (hence also $\operatorname{Dsi}^{c}_{I}(S)$ for some standard Young tableau $S$), one can view Proposition \ref{homrep} as a ``homogenization to the smallest possible degree'' of Theorem \ref{main}, as is visible in Example \ref{exhom}.

In fact, corollary \ref{sumcc} suggests that indexing the sub-representation of $R_{n,I}^{\mathrm{hom}}$ that is associated with $S$ via the generalized cocharge tableau $\operatorname{ct}_{J}(S)$, where $J:=\{n-i\;|\;i \in I\}$ as above, might be more natural. In \cite{[Z1]} we will modify this idea to get good decompositions of the action of $S_{n}$ on homogeneous polynomials of a fixed degree.

\section{Maps Between Representations \label{MapsRep}}

The representations $R_{n,k}$ for different values of $k$ do not admit natural maps between them. However, the components $R_{n,I}$ from Theorem \ref{main}, as well as their homogenized versions $R_{n,I}^{\mathrm{hom}}$ from Proposition \ref{homrep}, can be related to one another. Indeed, by embedding $I$ inside a larger subset $\tilde{I}$ of $\mathbb{N}_{n-1}$, the conditions from Definitions \ref{FwIdef} and \ref{hompols} are relaxed, yielding a relation between $R_{n,I}$ (resp. $R_{n,I}^{\mathrm{hom}}$) and $R_{n,\tilde{I}}$ (resp. $R_{n,\tilde{I}}^{\mathrm{hom}}$). This is related to $\mathcal{OP}_{n,I}$ being a quotient of $\mathcal{OP}_{n,\tilde{I}}$ by an $S_{n}$-equivariant equivalence relation (gluing some sets of blocks to larger, single blocks).

The latter operation is simplest when $\tilde{I}$ is the union of $I$ and a single new element $\ell$, a case from which the more general operation is successively built. The result is then as follows.
\begin{prop}
Assume that $I\subseteq\mathbb{N}_{n-1}$ and $\ell\in\mathbb{N}_{n-1} \setminus I$, and set $\tilde{I}:=I\cup\{\ell\}$.
\begin{enumerate}[$(i)$]
\item The representation $R_{n,\tilde{I}}^{\mathrm{hom}}$ is given by the direct sum of $R_{n,I}^{\mathrm{hom}}e_{\ell}$ and the sum $\bigoplus_{\lambda \vdash n}\bigoplus_{S\in\operatorname{SYT}(\lambda),\ \ell\in\operatorname{Dsi}^{c}(S) \subseteq I}V^{S}_{\vec{h}(I,S)}$, where the latter is the same as $\bigoplus_{\lambda \vdash n}\bigoplus_{C\in\operatorname{CCT}(\lambda),\ \ell\in\operatorname{Dsp}^{c}(C) \subseteq I}V_{C}^{\vec{h}(C,I)}$.
\item One part of $R_{n,\tilde{I}}$ is the sum over $\lambda \vdash n$ of $\bigoplus_{S\in\operatorname{SYT}(\lambda),\ \ell\in\operatorname{Dsi}^{c}(S) \subseteq I}V^{S}_{\vec{h}^{S}_{I}}$, or equivalently $\bigoplus_{C\in\operatorname{CCT}(\lambda),\ \ell\in\operatorname{Dsp}^{c}(C) \subseteq I}V_{C}^{\vec{h}_{C}^{I}}$. The remaining part is obtained from $R_{n,I}$ by replacing every multiplier $e_{r_{i}}$ for $\ell<i \in I$ by $e_{r_{i}-1}$, and multiplying the result by $e_{r}$, where $r:=|\{j\in\mathbb{N}_{n-1} \setminus I\;|\;j<\ell\}|$.
\item In case $\ell=n-1 \not\in I$ and $\tilde{I}:=I\cup\{n-1\}$, the second expression in part $(ii)$ is just $e_{n-k}R_{n,I}$, where $k:=|I|+1$ as usual.
\end{enumerate} \label{incI}
\end{prop}

\begin{proof}
First, it is clear that if $\operatorname{Dsi}^{c}(S)$ or $\operatorname{Dsp}^{c}(C)$ is contained in $I$ then it is contained in $\tilde{I}$, but with $\tilde{I}$ the associated representations will also involve tableaux $S$ or $C$ whose associated set contains $\ell$. The direct sum property from Proposition \ref{homrep} and Theorem \ref{main} yield, by considering the latter families of tableaux, the corresponding part of $R_{n,\tilde{I}}^{\mathrm{hom}}$ and $R_{n,\tilde{I}}$.

For the remaining tableaux $S$ or $C$, for which $\operatorname{Dsi}^{c}(S)$ or $\operatorname{Dsp}^{c}(C)$ is contained in $I$, we need to compare the vectors $\vec{h}(I,S)$ and $\vec{h}^{S}_{I}$, or $\vec{h}(C,I)$ and $\vec{h}_{C}^{I}$, with $\vec{h}(\tilde{I},S)$ and $\vec{h}^{S}_{\tilde{I}}$ or $\vec{h}(C,\tilde{I})$ and $\vec{h}_{C}^{\tilde{I}}$ respectively.

Now, $\vec{h}^{S}_{I}$ and $\vec{h}(C,I)$ are characteristic vectors of sets, and $\vec{h}(\tilde{I},S)$ and $\vec{h}(C,\tilde{I})$ are the characteristic vectors of the union of these sets with $\ell$. Hence each summand from $R_{n,I}^{\mathrm{hom}}$ is multiplied by $e_{\ell}$ in order to produce the corresponding summand in $R_{n,\tilde{I}}^{\mathrm{hom}}$, and part $(i)$ follows.

For $R_{n,I}$ and $R_{n,\tilde{I}}$, the vectors in question are based on the quantities $r_{i}$ from Definition \ref{FwIdef}, and we consider the effect of adding $\ell$ to $I$ on these expressions. As this quantity counts the number of elements in the complement of $I$ in $\mathbb{N}_{n-1}$ that are smaller than $i$, we deduce that for $i<\ell$ this number remains the same when we replace $I$ by $\tilde{I}$, while if $i>\ell$ then it decreases by 1 (since $\ell<i$ is counted in $\mathbb{N}_{n-1} \setminus I$ but not in $\mathbb{N}_{n-1}\setminus\tilde{I}$). Moreover, for $\vec{h}^{S}_{\tilde{I}}$ and $\vec{h}_{C}^{\tilde{I}}$ we have to add the index $r_{\ell}$ (by our assumption on $S$ or $C$), and part $(ii)$ is also established.

For part $(iii)$ we simply apply part $(ii)$ with $\ell=n-1$, and note that there are no elements $\ell<i \in I$. Hence no parameters are changed in that part, and since the value of $r_{\ell}$ is $n-k$, this expression reduces to the asserted one. This proves the proposition.
\end{proof}

\begin{ex}
Consider the empty set $I$, for $k=1$. Given any $\ell$, for the set $\tilde{I}=\{\ell\}$ the representations $R_{n,\tilde{I}}^{\mathrm{hom}}$ and $R_{n,\tilde{I}}$ contain the new components, with non-trivial representations of $S_{n}$, and the constant representation obtained from $R_{n,I}^{\mathrm{hom}}=R_{n,I}$ is multiplied, via Proposition \ref{incI}, by $e_{\ell}$ for the former and $e_{\ell-1}$ for the latter (compare with Examples \ref{bijsmall} and \ref{stdrep}). On the other hand, if $k=n-1$ then $R_{n,I}$ for any $I\subseteq\mathbb{N}_{n-1}$ of size $n-2$ only involves multipliers that are powers of $e_{1}$ (since $h_{r}=0$ for $r>n-k=1$), and when we go to $R_{n,\tilde{I}}=R_{n,n}$ all these multipliers become $e_{0}=1$ (by part $(ii)$ of Proposition \ref{incI}), and we add the representations associated with $S$ or $C$ whose corresponding set contains the element that is missing in $I$ as well (for $R_{n,I}^{\mathrm{hom}}$ and $R_{n,\tilde{I}}^{\mathrm{hom}}$ the description is via part $(i)$ of that proposition). \label{addItriv}
\end{ex}

\begin{ex}
If $n=4$ and $k=2$, then there are three values of $I$, and we saw, via Examples \ref{n4kI}, \ref{bijsmall}, \ref{stdrep}, \ref{exhom}, and \ref{addItriv} and Remark \ref{deg1}, that the associated representations $R_{4,I}$ are \[V^{1234} \oplus V^{\substack{123 \\ 4\hphantom{56}}}=V_{0000} \oplus V_{\substack{000 \\ 1\hphantom{11}}},\qquad V^{1234}e_{2} \oplus V^{\substack{134 \\ 2\hphantom{56}}}=V_{0000}e_{2} \oplus V_{\substack{011 \\ 1\hphantom{22}}},\] and \[V^{1234}e_{1} \oplus V^{\substack{124 \\ 3\hphantom{56}}} \oplus V^{\substack{12 \\ 34}}=V_{0000}e_{1} \oplus V_{\substack{001 \\ 1\hphantom{12}}} \oplus V_{\substack{00 \\ 11}}\] for $I$ being $\{1\}$, $\{3\}$, and $\{2\}$ respectively, with the representations $R_{4,I}^{\mathrm{hom}}$ being \[V^{1234}e_{1} \oplus V^{\substack{123 \\ 4\hphantom{56}}}=V_{0000}e_{1} \oplus V_{\substack{000 \\ 1\hphantom{11}}},\qquad V^{1234}e_{3} \oplus V^{\substack{134 \\ 2\hphantom{56}}}=V_{0000}e_{3} \oplus V_{\substack{011 \\ 1\hphantom{22}}},\] and \[V^{1234}e_{2} \oplus V^{\substack{124 \\ 3\hphantom{56}}} \oplus V^{\substack{12 \\ 34}}=V_{0000}e_{2} \oplus V_{\substack{001 \\ 1\hphantom{12}}} \oplus V_{\substack{00 \\ 11}}\] respectively. To each such set $I$ we can add one of two elements of the complement to get the set $\tilde{I}$, which is $\{1,2\}$, $\{1,3\}$, or $\{2,3\}$, and every such $\tilde{I}$ is obtained from two different sets $I$. The associated representations $R_{4,\tilde{I}}=W_{\vec{m}}$ show up at the end of Example \ref{n4kI}, with the analogues $R_{4,\tilde{I}}^{\mathrm{hom}}$ given at the end of Example \ref{exhom}. One can verify that for each of the 6 containments $I\subseteq\tilde{I}$, the expression for $R_{4,\tilde{I}}^{\mathrm{hom}}$ is described using our $R_{4,I}^{\mathrm{hom}}$ by part $(i)$ Proposition \ref{incI}, with part $(ii)$ of that proposition giving $R_{4,\tilde{I}}$ from $R_{I}$ (and the case $\ell=3$ is in correspondence with part $(iii)$ there). \label{addIn4k2}
\end{ex}

\medskip

We now recall that star and bar insertions from \cite{[HRS]} and others, which embed $\mathcal{OP}_{n,k}$ inside $\mathcal{OP}_{n+1,k}$ and $\mathcal{OP}_{n+1,k+1}$ respectively. For the former, we choose one of the $k$ blocks and add $n+1$ to that block (so there are $k$ such operations). The latter takes a space between two consecutive blocks, including the space before the first one and the one after the last one, and adds $\{n+1\}$ as a new singleton there (hence there are $k+1$ such operations). All these operations produce maps from $R_{n,k}$ into $R_{n+1,k}$ or $R_{n+1,k+1}$.

We will be interested in the operations putting $n+1$ in the rightmost position, due to the following observation.
\begin{lem}
Applying the star operation at the rightmost position takes $\mathcal{OP}_{n,I}$ into $\mathcal{OP}_{n+1,I}$. The bar operation at the rightmost position sends elements of the former set to $\mathcal{OP}_{n+1,I\cup\{n\}}$. \label{operslast}
\end{lem}

\begin{proof}
Recall from Definition \ref{OPnIdef} that the sizes of the sets in elements of $\mathcal{OP}_{n,I}$ are determined by the composition $\vec{m}:=\operatorname{comp}_{n}I \vDash n$, which equals $\{m_{h}\}_{h=0}^{k-1}$ by Lemma \ref{Snparts} in case $|I|=k-1$. Then the star operation replaces $m_{k-1}$ by $m_{k-1}+1$, which produces a composition of $n+1$, and after applying the inverse map $\operatorname{comp}_{n+1}^{-1}$ (which exists by Lemma \ref{setscomp}), we reproduce the set $I$, but now as a subset of $\mathbb{N}_{n}$. With the bar operation, we keep the entry $m_{k-1}$ as well, and add another entry $m_{k}=1$, to get a partition of $n+1$ of length $k+1$. The application of $\operatorname{comp}_{n+1}^{-1}$ to this composition is easily verified to produce $I\cup\{n\}\subseteq\mathbb{N}_{n}$. This proves the lemma.
\end{proof}

\begin{ex}
Consider $n=6$, and the ordered partition $\big(\{2,5\},\{1,3,6\},\{4\}\big)$, where $k=3$ and $I=\{2,5\}$ so that the sizes of this partition are given by $\vec{m}=231\operatorname{comp}_{6}I\vDash6$ (of length $k=3$). The star operation replaces the last set $\{4\}$ by $\{4,7\}$, of size 2, and that the resulting composition $232\vDash7$ is obtained as $\operatorname{comp}_{7}I$. Applying the bar operation adds a new singleton $\{7\}$ at the end, the associated composition is now $2311\vDash7$, and the set $\tilde{I}$ that is $\operatorname{comp}_{7}^{-1}$ of that composition is $\{2,5,6\}=I\cup\{6\}$. \label{barstarex}
\end{ex}
Remark \ref{otherlocs} below comments on these operations in the other locations.

\medskip

The operations from Lemma \ref{operslast} take sets on which $S_{n}$ acts and produces sets with an action of $S_{n+1}$. As Definition \ref{OPnIdef} relates the former to representations of $S_{n}$ and the latter to those of $S_{n+1}$, we would like to relate them. The main relation is the Branching Rule, for stating which, and for defining our lifts of it, we shall need some notation.
\begin{defn}
Consider an integer $n$, and a partition $\lambda \vdash n$.
\begin{enumerate}[$(i)$]
\item An \emph{external corner} of $\lambda$ is a box that is not in the Ferrers diagram of $\lambda$, such that adding it to that diagram produces the Ferrers diagram of a partition. The set of the external corners of $\lambda$ is denoted by $\operatorname{EC}(\lambda)$.
\item For every $v\in\operatorname{EC}(\lambda)$, we write $\lambda+v \vdash n+1$ for the partition whose Ferrers diagram is the union of that of $\lambda$ with $v$.
\item If $T\in\operatorname{SYT}(\lambda)$ then we write $T+v$ for the tableau of shape $\lambda+v$ obtained by filling $\lambda$ as in $T$ and putting $n+1$ inside $v$. It is in $\operatorname{SYT}(\lambda+v)$.
\item For $S\in\operatorname{SYT}(\lambda)$ we write $S\tilde{+}v$ for $\operatorname{ev}(\operatorname{ev}S+v)\in\operatorname{SYT}(\lambda+v)$.
\item Given $C\in\operatorname{CCT}(\lambda)$, we set $C\hat{+}v:=\operatorname{ct}\big(\operatorname{ct}^{-1}(C)\tilde{+}v\big)\in\operatorname{CCT}(\lambda+v)$.
\item we say that $v$ \emph{lies/is/sits below $n$ in $T$} if the row containing $v$ is below the row $R_{T}(n)$, namely the row index of $v$ is strictly larger than $R_{T}(n)$.
\end{enumerate} \label{ECadd}
\end{defn}

\begin{rmk}
In the notation from Definition \ref{ECadd}, the Branching Rule states that the induced representation $\operatorname{Ind}_{S_{n}}^{S_{n+1}}\mathcal{S}^{\lambda}$ equals $\bigoplus_{v\in\operatorname{EC}(\lambda)}\mathcal{S}^{\lambda+v}$. The fact that $\dim\mathcal{S}^{\lambda}=|\operatorname{SYT}(\lambda)|$ is related to the Branching Rule by the construction $T+v$, which is inverted by observing that for every $\nu \vdash n+1$, every element $\tilde{T}\in\operatorname{SYT}(\nu)$ is $T+v$ for a unique tuple of $\lambda \vdash n$, $T\in\operatorname{SYT}(\lambda)$, and $v\in\operatorname{EC}(\lambda)$ (with $\lambda+v=\nu$), where $v$ is the box $v_{\tilde{T}}(n+1)$ (which is then an \emph{internal corner} of $\nu$), and $\lambda$ and $T$ are obtained by removing it. The bijectivity of $\operatorname{ev}$ implies that every such $\tilde{T}$ can be written uniquely as $S\tilde{+}v$ for some $\lambda \vdash n$, $T\in\operatorname{SYT}(\lambda)$, and $v\in\operatorname{EC}(\lambda)$ for which $\lambda+v=\nu$. Similarly, the fact that $\operatorname{ct}$ is bijective yields that every element of $\operatorname{CCT}(\nu)$ is obtained as $C\hat{+}v$ for a unique choice of $\lambda \vdash n$, $C\in\operatorname{CCT}(\lambda)$, and $v\in\operatorname{EC}(\lambda)$ (where again $\lambda+v=\nu$). \label{branching}
\end{rmk}

The properties of the operations from Definition \ref{ECadd} in terms of the sets from Definitions \ref{asSYT} and \ref{Dsic} are given in the following extension of Lemma \ref{iotaexp}.
\begin{lem}
Fix $\lambda \vdash n$, $T$ and $S$ from $\operatorname{SYT}(\lambda)$, $C\in\operatorname{CCT}(\lambda)$, and $v\in\operatorname{EC}(\lambda)$.
\begin{enumerate}[$(i)$]
\item The set $\operatorname{Dsi}(T+v)$ equals $\operatorname{Dsi}(T)\cup\{n\}$ in case $v$ lies below $n$ in $T$, and is $\operatorname{Dsi}(T)$ otherwise.
\item Considering $\operatorname{Dsi}^{c}(S\tilde{+}v)$, it is $\operatorname{Dsi}^{c}(S)\cup\{n\}$ when $v$ is below $n$ in $\operatorname{ev}S$, and it equals $\operatorname{Dsi}^{c}(S)$ otherwise.
\item For $\operatorname{Dsp}^{c}(C\hat{+}v)$, it is given by $\operatorname{Dsp}^{c}(C)\cup\{n\}$ if $v$ sits below $n$ in the tableau $\operatorname{ev}\operatorname{ct}^{-1}(C)$, and it coincides with $\operatorname{Dsp}^{c}(C)$ otherwise.
\item In the latter case, the content of $C\hat{+}v$ is the same as that of $C$ plus one more instance of 0, so that $\Sigma(C\hat{+}v)=\Sigma(C)$, and also $\operatorname{cc}(S\tilde{+}v)=\operatorname{cc}(S)$. In the former it is obtained by adding 1 to each entry of $C$ and putting one instance of 0, and we have the equalities $\Sigma(C\hat{+}v)=\Sigma(C)+n$ as well as $\operatorname{cc}(S\tilde{+}v)=\operatorname{cc}(S)+n$.
\item We have $\operatorname{Dsi}(T)=\operatorname{Dsi}(T+v)\cap\mathbb{N}_{n-1}$, $\operatorname{Dsi}^{c}(S)=\operatorname{Dsi}^{c}(S\tilde{+}v)\cap\mathbb{N}_{n-1}$, and $\operatorname{Dsp}^{c}(C)=\operatorname{Dsp}^{c}(C\hat{+}v)\cap\mathbb{N}_{n-1}$.
\item Given $I\subseteq\mathbb{N}_{n-1}$, we have $\operatorname{Dsi}^{c}(S) \subseteq I$ if and only if $\operatorname{Dsi}^{c}(S\tilde{+}v) \subseteq I\cup\{n\}$ inside $\mathbb{N}_{n}$, and similarly $\operatorname{Dsp}^{c}(C) \subseteq I$ if and only if $\operatorname{Dsp}^{c}(C\hat{+}v) \subseteq I\cup\{n\}$.
\end{enumerate} \label{addv}
\end{lem}

\begin{proof}
The determination of whether $i<n$ lies in $\operatorname{Dsi}(T+v)$ depends, via Definition \ref{asSYT}, only in the locations of $i$ and $i+1$ in $T+v$. As these are the same as their locations in $T$ by Definition \ref{ECadd}, we deduce that $i$ lies in that set if and only if it is in $\operatorname{Dsi}(T)$. The condition for including $n$ is precisely the asserted one (by Definition \ref{asSYT}), yielding part $(i)$.

We now recall from the proof of Proposition \ref{evkeepsdeg} that $\operatorname{Dsi}(\operatorname{ev}S)$ is given by $\{n-i\;|\;i\in\operatorname{Dsi}(S)\}$, so that part $(i)$ implies that $\operatorname{Dsi}(\operatorname{ev}S+v)$ is either the same set or is obtained by adding $n$ (using our dichotomy, but now comparing $v$ with $R_{\operatorname{ev}S}(n)$ rather than $R_{S}(n)$). By applying $\operatorname{ev}$ again, now with $n+1$, we deduce from Definition \ref{ECadd} that $\operatorname{Dsi}(S\tilde{+}v)$ either equals $\operatorname{Dsi}(S)_{+}$ in the notation from Definition \ref{iotadef} (namely adding 1 to each of the elements), or its union with $\{1\}$. Considering Definition \ref{Dsic} for $\operatorname{Dsi}^{c}(S)$ (which is the same as $\operatorname{Dsi}(\operatorname{ev}S)$ as we already saw) and for $\operatorname{Dsi}^{c}(S\tilde{+}v)$ thus implies part $(ii)$.

Taking $C\in\operatorname{CCT}(\lambda)$, set $S:=\operatorname{ct}^{-1}(C)$, and then Definition \ref{ECadd} implies that $C\hat{+}v=\operatorname{ct}(S\tilde{+}v)$. Since Definition \ref{Dsic} then implies that $\operatorname{Dsp}^{c}(C)=\operatorname{Dsi}^{c}(S)$ and $\operatorname{Dsp}^{c}(C\hat{+}v)=\operatorname{Dsi}^{c}(S\tilde{+}v)$, we deduce part $(iii)$ from part $(ii)$.

Part $(iv)$ is a consequence of part $(iii)$ via parts $(ii)$ and $(iii)$ of Lemma \ref{Dspc} (using part $(ii)$ for the assertions about $S\tilde{+}v$). Part $(v)$ follows directly from the previous three parts, and parts $(ii)$ and $(iii)$ also give part $(vi)$ as an immediate consequence. This proves the lemma.
\end{proof}
Lemma \ref{addv} also holds for $\operatorname{Asi}(T+v)$, $\operatorname{Asi}^{c}(S\tilde{+}v)$, and $\operatorname{Asp}^{c}(C\hat{+}v)$ (with the sets remaining the same when $v$ is below $n$ in the corresponding tableau and we add $n$ to them otherwise).
\begin{rmk}
The partition $\lambda_{+}$ from Definition \ref{iotadef} is $\lambda+v$ from Definition \ref{ECadd} where $v=(1,\lambda_{1}+1)$ is the unique element of $\operatorname{EC}(\lambda)$ that lies in the first row, so that $\iota T$, $\tilde{\iota}S$, and $\hat{\iota}C$ are $T+v$, $S\tilde{+}v$, and $C\hat{+}v$ with this $v$ respectively. Since this $v$ does not sit below $n$ in any tableau of shape $\lambda$, parts $(ii)$ and $(iii)$ of Lemma \ref{addv} reduce, in this case, to the equalities from parts $(ii)$ and $(iv)$ in Lemma \ref{iotaexp}. On the other hand, $\operatorname{EC}(\lambda)$ always contains the element $v=\big(\ell(\lambda)+1,1\big)$, which lies below every entry of $\lambda$ and with which the sets from Lemma \ref{addv} increase. \label{iotaaddv}
\end{rmk}

\begin{ex}
Let $T$ be the tableau denoted by $S$ in Example \ref{exct}, set $S$ to be its $\operatorname{ev}$-image, and define $C$ to be the $\operatorname{ct}$-image of the latter, with the last two showing up in Example \ref{exev}, all of shape $\lambda=431 \vdash n=8$, and with $\operatorname{Dsi}(T)=\operatorname{Dsi}^{c}(S)=\operatorname{Dsp}^{c}(C)=\{2,4,7\}$. The set $\operatorname{EC}(\lambda)$ from Definition \ref{ECadd} contains four boxes, with $v=(5,1)$ yielding, via Remark \ref{iotaaddv}, the images under $\iota$, $\tilde{\iota}$, and $\hat{\iota}$, with the latter two being given in Example \ref{iotatabs}, and having shape $\lambda_{+}=531\vdash9$ and the same respective sets $\operatorname{Dsi}$, $\operatorname{Dsi}^{c}$, and $\operatorname{Dsp}^{c}$. For the external corner $v$ in the second row, with $\lambda+v=441\vdash9$, we get the tableaux \[T+v=\begin{ytableau} 1 & 2 & 4 & 7 \\ 3 & 6 & 8 & 9 \\ 5 \end{ytableau},\ S\tilde{+}v=\begin{ytableau} 1 & 2 & 4 & 5 \\ 3 & 6 & 7 & 9 \\ 8 \end{ytableau}\mathrm{\ and\ }C\hat{+}v=\begin{ytableau} 0 & 0 & 1 & 1 \\ 1 & 2 & 2 & 3 \\ 3 \end{ytableau},\] again with the same respective sets. When $v$ is the external corner in the third row, we get $\lambda+v=432\vdash9$ and the tableaux \[T+v=\begin{ytableau} 1 & 2 & 4 & 7 \\ 3 & 6 & 8 \\ 5 & 9 \end{ytableau},\ S\tilde{+}v=\begin{ytableau} 1 & 4 & 5 & 9 \\ 2 & 6 & 7 \\ 3 & 8 \end{ytableau}\mathrm{\ and\ }C\hat{+}v=\begin{ytableau} 0 & 2 & 2 & 4 \\ 1 & 3 & 3 \\ 2 & 4 \end{ytableau},\] where now $\operatorname{Dsi}(T+v)=\operatorname{Dsi}^{c}(S\tilde{+})=\operatorname{Dsp}^{c}(C\hat{+}v)=\{2,4,7,8\}$ since this is the first external corner lying below the row $R_{\operatorname{ev}\operatorname{ct}^{-1}(C)}(8)=R_{\operatorname{ev}S}(8)=R_{T}(8)=2$. Finally, with $v$ yielding a new fourth row in $\lambda+v=4311\vdash9$, the sets of the resulting tableaux are as in the previous case (as Remark \ref{iotaaddv} predicts), and the tableaux themselves are \[T+v=\begin{ytableau} 1 & 2 & 4 & 7 \\ 3 & 6 & 8 \\ 5 \\ 9 \end{ytableau},\ S\tilde{+}v=\begin{ytableau} 1 & 4 & 5 & 9 \\ 2 & 6 & 7 \\ 3 \\ 8 \end{ytableau}\mathrm{\ and\ }C\hat{+}v=\begin{ytableau} 0 & 2 & 2 & 4 \\ 1 & 3 & 3 \\ 2 \\ 4 \end{ytableau}.\] \label{exaddv}
\end{ex}
Note that $\hat{\iota}C$ from Example \ref{iotatabs} and $C\hat{+}v$ for $v\in\operatorname{EC}(\lambda)$ lying in the second row have the same content as $C$ from Example \ref{exev}, while the other two cases of $C\hat{+}v$ in Example \ref{exaddv} have a different content, which corresponds to adding $n+1=9$ to the $\operatorname{Dsp}^{c}$-set via the description from part $(ii)$ of Lemma \ref{Dspc} (and the fact that $k=4$ also grows by 1 in this operation), as part $(iv)$ of Lemma \ref{addv} predicts.

\medskip

In order to relate the Branching Rule, as given in Remark \ref{branching}, with the representations from Definition \ref{defSpecht}, we make the following definition.
\begin{defn}
Take $\lambda \vdash n$, $S\in\operatorname{SYT}(\lambda)$, $C\in\operatorname{CCT}(\lambda)$, $v\in\operatorname{EC}(\lambda)$, and an index $0 \leq t \leq n+1$.
\begin{enumerate}[$(i)$]
\item Set $\delta^{S,v}$ to be 0 in case $v$ lies below $n$ in $\operatorname{ev}S$, and 1 otherwise. Similarly, $\delta_{C,v}$ is defined to be 0 when $v$ sits below sits below $n$ in $\operatorname{ev}\operatorname{ct}(C)$, and 1 otherwise.
\item We define $\operatorname{Ind}_{t,S_{n}}^{S_{n+1}}V^{S}$ to be $\bigoplus_{v\in\operatorname{EC}(\lambda)}V^{S\tilde{+}v}e_{t}^{\delta^{S,v}}$. In a similar manner, we set $\operatorname{Ind}_{t,S_{n}}^{S_{n+1}}V_{C}:=\bigoplus_{v\in\operatorname{EC}(\lambda)}V_{C\hat{+}v}e_{t}^{\delta_{C,v}}$.
\item Given a vector $\vec{h}=\{h_{r}\}_{r=1}^{n}$ as usual, which we consider to contain also the entry $h_{n+1}=0$, and an element $\delta\in\{0,1\}$, we write $\vec{h}+\delta_{t}$ for the vector which is obtained from $\vec{h}$ by adding $\delta$ to $h_{t}$ in case $t>0$ and leaving the other $h_{r}$'s invariant. The notation $\vec{h}+\delta^{t}$ is the same as $\vec{h}+\delta_{t}$.
\item For any such vector $\vec{h}$ we set $\operatorname{Ind}_{t,S_{n}}^{S_{n+1}}V^{S}_{\vec{h}}$ to be $\bigoplus_{v\in\operatorname{EC}(\lambda)}V^{S\tilde{+}v}_{\vec{h}+\delta^{S,v}_{t}}$, as well as $\operatorname{Ind}_{t,S_{n}}^{S_{n+1}}V_{C}^{\vec{h}}:=\bigoplus_{v\in\operatorname{EC}(\lambda)}V_{C\hat{+}v}^{\vec{h}+\delta_{C,v}^{t}}e_{t}^{\delta_{C,v}}$
\item We set $\operatorname{Ext}_{S_{n}}^{S_{n+1}}V^{S}$ to be the representation $\bigoplus_{v\in\operatorname{EC}(\lambda)}\delta^{S,v}V^{S\tilde{+}v}$, and similarly $\operatorname{Ext}_{S_{n}}^{S_{n+1}}V_{C}=\bigoplus_{v\in\operatorname{EC}(\lambda)}\delta_{C,v}V_{C\hat{+}v}$ (namely we take only the summands associated with those $v\in\operatorname{EC}(\lambda)$ for which $\delta^{S,v}$ or $\delta_{C,v}$ equals 1). More generally, for every $\vec{h}$ we define $\operatorname{Ext}_{S_{n}}^{S_{n+1}}V^{S}_{\vec{h}}=\bigoplus_{v\in\operatorname{EC}(\lambda)}\delta^{S,v}V^{S\tilde{+}v}_{\vec{h}}$ and $\operatorname{Ext}_{S_{n}}^{S_{n+1}}V_{C}^{\vec{h}}=\bigoplus_{v\in\operatorname{EC}(\lambda)}\delta_{C,v}V_{C\hat{+}v}^{\vec{h}}$, where we consider $\vec{h}$ as extended by $h_{n+1}=0$.
\item If $U$ is a representation of $S_{n}$ that admits a decomposition as either $\bigoplus_{\lambda \vdash n}\bigoplus_{S\in\operatorname{SYT}(\lambda)}\bigoplus_{\vec{h} \in H^{S}}V^{S}_{\vec{h}}$ or $\bigoplus_{\lambda \vdash n}\bigoplus_{C\in\operatorname{CCT}(\lambda)}\bigoplus_{\vec{h} \in H_{C}}V_{C}^{\vec{h}}$, where $H^{S}$ and $H_{C}$ are sets of vectors of the usual form, then we write $\operatorname{Ind}_{t,S_{n}}^{S_{n+1}}U$ and $\operatorname{Ext}_{S_{n}}^{S_{n+1}}U$ for the direct sums $\bigoplus_{\lambda \vdash n}\bigoplus_{S\in\operatorname{SYT}(\lambda)}\bigoplus_{\vec{h} \in H_{S}}\operatorname{Ind}_{t,S_{n}}^{S_{n+1}}V^{S}_{\vec{h}}$ and $\bigoplus_{\lambda \vdash n}\bigoplus_{S\in\operatorname{SYT}(\lambda)}\bigoplus_{\vec{h} \in H_{S}}\operatorname{Ext}_{S_{n}}^{S_{n+1}}V^{S}_{\vec{h}}$ respectively using the first expression, and for the representations $\bigoplus_{\lambda \vdash n}\bigoplus_{C\in\operatorname{CCT}(\lambda)}\bigoplus_{\vec{h} \in H_{C}}\operatorname{Ind}_{t,S_{n}}^{S_{n+1}}V_{C}^{\vec{h}}$ and $\bigoplus_{\lambda \vdash n}\bigoplus_{C\in\operatorname{CCT}(\lambda)}\bigoplus_{\vec{h} \in H_{C}}\operatorname{Ext}_{S_{n}}^{S_{n+1}}V_{C}^{\vec{h}}$ with the second one.
\end{enumerate} \label{IndtVSVC}
\end{defn}
It is clear that if $C=\operatorname{ct}(S)$ then $\delta_{C,v}=\delta^{S,v}$ in Definition \ref{IndtVSVC}, and recalling that $V^{S}=V_{C}$ via Definition \ref{defSpecht} and Remark \ref{samepols}, hence also $V^{S}_{\vec{h}}=V_{C}^{\vec{h}}$ for every vector $\vec{h}$, the expressions for their $\operatorname{Ind}_{t,S_{n}}^{S_{n+1}}$-images are indeed the same. For $t=0$ we get $\operatorname{Ind}_{t,S_{n}}^{S_{n+1}}V^{S}=\bigoplus_{v\in\operatorname{EC}(\lambda)}V^{S\tilde{+}v}$ and $\operatorname{Ind}_{t,S_{n}}^{S_{n+1}}V_{C}=\bigoplus_{v\in\operatorname{EC}(\lambda)}V_{C\hat{+}v}$, and $\vec{h}+\delta_{0}=\vec{h}+\delta^{0}$ is the same as $\vec{h}$ (up to adding vanishing entries, as commented after Definition \ref{defSpecht}), as we considered for $\operatorname{Ext}_{S_{n}}^{S_{n+1}}V^{S}_{\vec{h}}$ and $\operatorname{Ext}_{S_{n}}^{S_{n+1}}V_{C}^{\vec{h}}$.

\begin{ex}
For $S$ and $C=\operatorname{ct}(S)$ as in Examples \ref{exev} and \ref{exaddv}, with $n=8$ and $\lambda:=431\vdash8$, we consider image of the representation $V^{S}=V_{C}$ under $\operatorname{Ind}_{t,S_{8}}^{S_{9}}$ for some $0 \leq t\leq9$. The number of summands is the size 4 of $\operatorname{EC}(\lambda)$, where the box $v$ there for which $S\tilde{+}v=\tilde{\iota}S$ and $C\hat{+}v=\hat{\iota}C$ always satisfies $\delta^{S,v}=\delta_{C,v}=1$, and produces the summand $V^{\tilde{\iota}S}e_{t}=V_{\hat{\iota}C}e_{t}$, with the indices of this representation being given in Example \ref{iotatabs}. The box $v$ from the second row contributes the representation $V^{S\tilde{+}v}e_{t}=V_{C\hat{+}v}e_{t}$ (as here we still have $\delta^{S,v}=\delta_{C,v}=1$) and the other two choices of $v$, the tableaux for all of which are given explicitly in Example \ref{exaddv}, give two summands $V^{S\tilde{+}v}=V_{C\hat{+}v}$ without the multiplier $e_{t}$, since $\delta^{S,v}=\delta_{C,v}=0$ in these cases. The representation $\operatorname{Ext}_{S_{n}}^{S_{n+1}}V^{S}=\operatorname{Ext}_{S_{n}}^{S_{n+1}}V_{C}$ is just the direct sum of the first two components, without the multiplier $e_{t}$. \label{Index}
\end{ex}

The meaning of Definition \ref{IndtVSVC} is as follows.
\begin{lem}
The operation $\operatorname{Ind}_{t,S_{n}}^{S_{n+1}}$ lifts the Branching Rule into an operation on representations inside $\mathbb{Q}[\mathbf{x}_{n}]$ that commutes with multiplication by symmetric functions and with direct sums. When $t=n$ it is a homogeneous operation of degree $n$, and otherwise it is not homogeneous at all. \label{propInd}
\end{lem}
The symmetric functions in Lemma \ref{propInd} are only polynomials in the $e_{r}$'s, but this is all that will be used later.

\begin{proof}
We will work with $S\in\operatorname{SYT}(\lambda)$ and $C\in\operatorname{CCT}(\lambda)$ for some $\lambda \vdash n$, under the assumption that $C=\operatorname{ct}(S)$ throughout, so that $V^{S}=V_{C}$, to state the argument in the two notations simultaneously.

Definition \ref{IndtVSVC} means that in $\operatorname{Ind}_{t,S_{n}}^{S_{n+1}}V^{S}=\operatorname{Ind}_{t,S_{n}}^{S_{n+1}}V_{C}$, the component $V^{S\tilde{+}v}=V^{C\hat{+}v}$ shows up as it is in case $v$ lies below $n$ in $\operatorname{ev}S=\operatorname{ev}\operatorname{ct}^{-1}(C)$, and multiplied by $e_{t}$ otherwise. Recall that the index $\vec{h}$ indicates that $V^{S}=V_{C}$ is multiplied by $\prod_{r=1}^{n}e_{r}^{h_{r}}$, and note that in the product arising from $\vec{h}+\delta^{S,v}_{t}$, or equivalently $\vec{h}+\delta_{C,v}^{t}$, the exponent of $e_{t}$ is $h_{t}+\delta^{S,v}=h_{t}+\delta_{S,v}$ and the exponent of $e_{r}$ for any other $r$ is $h_{r}$. This establishes the commutation between $\operatorname{Ind}_{t,S_{n}}^{S_{n+1}}$ and the multiplication by $\prod_{r=1}^{n}e_{r}^{h_{r}}$, and the assertion about direct sums follows immediately from the fact that $\operatorname{Ind}_{t,S_{n}}^{S_{n+1}}U$ is the direct sum of the $\operatorname{Ind}_{t,S_{n}}^{S_{n+1}}$-images of its irreducible components.

The lifting of the Branching Rule, as given in Remark \ref{branching}, is an immediate consequence of the fact that $\operatorname{sh}(S\tilde{+}v)=\operatorname{sh}(C\hat{+}v)=\lambda+v$ for every $v\in\operatorname{EC}(\lambda)$ via Definition \ref{ECadd}, since the each representation in Definition \ref{defSpecht} is isomorphic, by Theorem \ref{VSVCreps}, to the Specht module of the corresponding shape (or partition).

Finally, the two commutation properties implies that for proving homogeneity, it suffices to consider a single representation $V^{S}=V_{C}$, with no index. This representation is homogeneous of degree $\operatorname{cc}(S)=\Sigma(C)$ by Theorem \ref{VSVCreps}, and we saw in the proof of Lemma \ref{repshom} that this value is the same as $\sum_{i\in\operatorname{Dsi}^{c}(S)}i=\sum_{i\in\operatorname{Dsp}^{c}(C)}i$.

Lemma \ref{addv} therefore shows that $V^{S\tilde{+}v}=V_{C\hat{+}v}$ is homogeneous of degree $\operatorname{cc}(S)+n=\Sigma(C)+n$ in case $v$ lies below $n$ in $\operatorname{ev}S=\operatorname{ev}\operatorname{ct}^{-1}(C)$, and the same degree $\operatorname{cc}(S)=\Sigma(C)$ as $V^{S}=V_{C}$ otherwise. The former terms show up as they are inside $\operatorname{Ind}_{t,S_{n}}^{S_{n+1}}V^{S}$ by Definition \ref{IndtVSVC}, but the latter ones are multiplied by $e_{t}$, reaching homogeneity degree $\operatorname{cc}(S)+t=\Sigma(C)+t$.

This shows that if $t=n$ then $\operatorname{Ind}_{n,S_{n}}^{S_{n+1}}V^{S}=\operatorname{Ind}_{n,S_{n}}^{S_{n+1}}V_{C}$ is indeed a homogeneous representation of degree $\operatorname{cc}(S)+n=\Sigma(C)+n$. As Remark \ref{iotaaddv} implies that in general this operation contains at least one summand of the latter degree and at least one other summand whose degree involves $t$, we deduce that for other values of $t$ the resulting representation is not homogeneous. This completes the proof of the lemma.
\end{proof}

\begin{rmk}
In fact, the proof of Lemma \ref{propInd} implies that $\operatorname{Ext}_{S_{n}}^{S_{n+1}}$ is homogeneous of degree 0 and takes non-zero representations to non-zero representations (indeed, when applied to $V^{S}_{\vec{h}}$ or $V_{C}^{\vec{h}}$, for $S$ or $C$ of shape $\lambda$, it was seen, via Remark \ref{iotaaddv}, to contain $V^{\tilde{\iota}S}_{\vec{h}}$ or $V_{\hat{\iota}C}^{\vec{h}}$, and since the corresponding tableau is of shape $\lambda_{+}$, with a longer first row, we call $\operatorname{Ext}_{S_{n}}^{S_{n+1}}U$ the \emph{extension} of $U$). Moreover, it shows that $\operatorname{Ind}_{t,S_{n}}^{S_{n+1}}$ can be written as $e_{t}\operatorname{Ext}_{S_{n}}^{S_{n+1}}$ plus another operator, which is homogeneous of degree $n$ and also sends, by Remark \ref{iotaaddv}, non-zero representations to non-zero representations. This complementary operator, which is based, for each $S$ or $C$, on $v\in\operatorname{EC}(\lambda)$ with $\delta^{S,v}$ or $\delta_{C,v}$ vanishing, will be useful in \cite{[Z1]}. \label{IndExt}
\end{rmk}

\begin{ex}
The representation $V^{S}=V_{C}$ from Example \ref{Index} is homogeneous of degree 13, which is the value of $\operatorname{cc}(S)=\Sigma(C)$ from Example \ref{exev}. As we saw in Examples \ref{iotatabs} and \ref{exaddv}, two elements $v$ of $\operatorname{EC}(\lambda=431\vdash8)$ give $S\tilde{+}v$ and $C\hat{+}v$ with the same value of $\operatorname{cc}(S)=\operatorname{ent}(C)$, and two others have a value of 21. Thus the image of the initial representation under $\operatorname{Ind}_{t,S_{8}}^{S_{9}}$ for some $0 \leq t\leq9$ is the direct sum of four representations, two of degree $13+t$ and two of degree 21, which are the same for $t=n=8$ (thus yielding a homogeneous representation) but not otherwise. The first part is $e_{t}$ times the $\operatorname{Ext}_{S_{n}}^{S_{n+1}}$-image of our initial representation, and the latter is also of degree 13. \label{homInd}
\end{ex}

\medskip

The effect of the bar insertion at the last location can now be described using Definition \ref{IndtVSVC} in the following result, which may be viewed as a higher Specht polynomial analogue of Theorem 7.4 of \cite{[PR]}.
\begin{prop}
For any $I\subseteq\mathbb{N}_{n-1}$ we have $R_{n+1,I\cup\{n\}}=\operatorname{Ind}_{n-k,S_{n}}^{S_{n+1}}R_{n,I}$, where $k:=|I|+1$ as always, and $R_{n+1,I\cup\{n\}}^{\mathrm{hom}}=\operatorname{Ind}_{n,S_{n}}^{S_{n+1}}R_{n,I}^{\mathrm{hom}}$. \label{embInd}
\end{prop}

\begin{proof}
As we saw in the proofs of Theorem \ref{main} and Proposition \ref{homrep}, the representations $R_{n,I}$ and $R_{n,I}^{\mathrm{hom}}$ have decompositions as in Definition \ref{IndtVSVC}, with the set $H^{S}$ and $H_{C}$ being a singleton (containing the vector $\vec{h}^{S}_{I}$, $\vec{h}_{C}^{I}$, $\vec{h}(I,S)$, or $\vec{h}(C,I)$ respectively) when $\operatorname{Dsi}^{c}(S)$ or $\operatorname{Dsp}^{c}(C)$ is contained in $I$, and being empty otherwise. Part $(vi)$ of Lemma \ref{addv} implies that for each irreducible component of $R_{n,I}$ or $R_{n,I}^{\mathrm{hom}}$, its induction $\operatorname{Ind}_{t,S_{n}}^{S_{n+1}}$ for any $t$ is supported on representations showing up in $R_{n+1,I\cup\{n\}}$ and in $R_{n+1,I\cup\{n\}}^{\mathrm{hom}}$.

Moreover, the complement $\mathbb{N}_{n}\setminus(I\cup\{n\})$ is the same as $\mathbb{N} \setminus I$, and it follows from Lemma \ref{addv} that the set $\operatorname{Asi}^{c}_{I\cup\{n\}}(S\tilde{+}v)=\operatorname{Asi}^{c}(S\tilde{+}v)\cap(I\cup\{n\})$ is $\operatorname{Asi}^{c}_{I}(S)$ when $v$ is below $n$ in $\operatorname{ev}S$ and equals $\operatorname{Asi}^{c}_{I}(S)\cup\{n\}$ otherwise, and similarly $\operatorname{Asp}^{c}_{I\cup\{n\}}(C\hat{+}v)=\operatorname{Asp}^{c}(C\hat{+}v)\cap(I\cup\{n\})$ equals $\operatorname{Asp}^{c}_{I}(C)$ in case $v$ is below $n$ in $\operatorname{ev}\operatorname{ct}^{-1}(C)$, and is the same as $\operatorname{Asp}^{c}_{I}(C)\cup\{n\}$ otherwise. Hence if $i$ is in $\operatorname{Asi}^{c}_{I}(S)$ or $\operatorname{Asp}^{c}_{I}(C)$ then the parameter $r_{i}$ from Definition \ref{FwIdef} is the same for $S$, $C$, and $I$ and for $S\tilde{+}v$, $C\hat{+}v$ and $I\cup\{n\}$, and if $v$ is not below $n$ in $\operatorname{ev}S$ or in $\operatorname{ev}\operatorname{ct}^{-1}(C)$ then for the element $n$ of $\operatorname{Asi}^{c}_{I\cup\{n\}}(S\tilde{+}v)$ or $\operatorname{Asp}^{c}_{I\cup\{n\}}(C\hat{+}v)$ we get $r_{n}=n-k$.

Gathering the definitions, we obtain that $\vec{h}^{S\tilde{+}v}_{I\cup\{n\}}$, $\vec{h}_{C\hat{+}v}^{I\cup\{n\}}$, $\vec{h}(I\cup\{n\},S\tilde{+}v)$, and $\vec{h}(C\hat{+}v,I\cup\{n\})$ are equal to $\vec{h}^{S}_{I}$, $\vec{h}_{C}^{I}$, $\vec{h}(I,S)$, and $\vec{h}(C,I)$ respectively (implicitly assumed to have the additional entry $h_{n+1}=0$ as above) in case $v$ lies below $n$ in $\operatorname{ev}S$ or in $\operatorname{ev}\operatorname{ct}^{-1}(C)$, and otherwise the existence of $r_{n}=n-k$ implies that they are obtained from these extensions by adding 1 to the entry at index $n-k$, $n-k$, $n$ and $n$ respectively. Hence we can write these vectors in the notation from Definition \ref{IndtVSVC} as $\vec{h}^{S}_{I}+\delta^{S,v}_{n-k}$, $\vec{h}_{C}^{I}+\delta_{S,v}^{n-k}$, $\vec{h}(I,S)+\delta^{S,v}_{n}$, and $\vec{h}(C,I)+\delta_{S,v}^{n}$ respectively, meaning that for every component of $R_{n,I}$ (resp. $R_{n,I}^{\mathrm{hom}}$), its image under $\operatorname{Ind}_{n-k,S_{n}}^{S_{n+1}}$ (resp. $\operatorname{Ind}_{n,S_{n}}^{S_{n+1}}$) is contained in $R_{n+1,I\cup\{n\}}$ (resp. $R_{n+1,I\cup\{n\}}^{\mathrm{hom}}$).

The direct sum property from Definition \ref{IndtVSVC} and Lemma \ref{propInd} therefore implies that the representations $\operatorname{Ind}_{n-k,S_{n}}^{S_{n+1}}R_{n,I}$ and $\operatorname{Ind}_{n,S_{n}}^{S_{n+1}}R_{n,I}^{\mathrm{hom}}$ are direct sums of sub-representations of $R_{n+1,I\cup\{n\}}$ and $R_{n+1,I\cup\{n\}}^{\mathrm{hom}}$ respectively, and we saw from Theorem \ref{main} and Proposition \ref{homrep} that the latter representations are supported on partitions $\nu \vdash n+1$ and on tableaux (standard or cocharge) of shape $\nu$ whose corresponding set, $\operatorname{Dsi}^{c}$ or $\operatorname{Dsp}^{c}$, is contained in $I\cup\{n\}$.

But Remark \ref{branching} shows that every such tableau can be expressed uniquely as $S\tilde{+}v$ or $C\hat{+}v$ with $\lambda \vdash n$, $S\in\operatorname{SYT}(\mu)$ or $C\in\operatorname{CCT}(\mu)$, and $v\in\operatorname{EC}(\lambda)$. Hence every summand in these representations is obtained, via the corresponding induction operator, from a unique summand in $R_{n,I}$ and $R_{n,I}^{\mathrm{hom}}$ respectively, which produces the desired equality. This completes the proof of the proposition.
\end{proof}
Note that for $k=n$ and $I=\mathbb{N}_{n-1}$ the value $t=n-k$ vanishes, which is the reason for including it in Definition \ref{IndtVSVC} (this amounts, via Example \ref{trivthm}, to the regular representation $R_{n+1,n+1}$ of $S_{n+1}$ being induced from that of $S_{n}$, namely $R_{n,n}$, via the decomposition from \cite{[ATY]}). The homogeneity from Lemma \ref{propInd} is related, by Proposition \ref{embInd}, to the fact that $R_{n,I}^{\mathrm{hom}}$ and $R_{n+1,I\cup\{n\}}^{\mathrm{hom}}$ are homogeneous (by Proposition \ref{homrep}).

In the case from Example \ref{homInd}, the representation $V^{S}=V_{C}$ participates, with no vector $\vec{h}$, in $R_{8,I}$ and $R_{8,I}^{\mathrm{hom}}$ for $I:=\{2,4,7\}$, so that $k:=|I|+1=4$. Thus  by Proposition \ref{embInd} implies that $R_{9,I\cup\{8\}}$ and $R_{9,I\cup\{8\}}^{\mathrm{hom}}$ contain $\operatorname{Ind}_{4,S_{8}}^{S_{9}}$ and the homogeneous representation $\operatorname{Ind}_{8,S_{8}}^{S_{9}}$ respectively, with the components that are described in that example.

\medskip

We now turn to the star insertion at the last location, analogously to Theorem 7.1 of \cite{[PR]}, which uses the other construction from Definition \ref{IndtVSVC}.
\begin{prop}
Given any subset $I\subseteq\mathbb{N}_{n-1}$, of some size $k-1$, we have the equalities $R_{n+1,I}=\operatorname{Ext}_{S_{n}}^{S_{n+1}}R_{n,I}$ and $R_{n+1,I}^{\mathrm{hom}}=\operatorname{Ext}_{S_{n}}^{S_{n+1}}R_{n,I}^{\mathrm{hom}}$. \label{embreps}
\end{prop}

\begin{proof}
Parts $(iii)$ and $(i)$ of Proposition \ref{incI}, with $n+1$ and $\ell=n$, express $R_{n+1,I\cup\{n\}}$ and $R_{n+1,I\cup\{n\}}^{\mathrm{hom}}$ as $e_{n-k}R_{n+1,I}$ and $e_{n}R_{n+1,I}^{\mathrm{hom}}$ respectively, plus direct sums of representations that are based on tableaux of shape $\nu \vdash n+1$ such that their corresponding set $\operatorname{Dsi}^{c}$ or $\operatorname{Dsp}^{c}$ contains $n$. It follows that we can identify $R_{n+1,I}$ and $R_{n+1,I}^{\mathrm{hom}}$ by taking the parts of $R_{n+1,I\cup\{n\}}$ and $R_{n+1,I\cup\{n\}}^{\mathrm{hom}}$ whose components are indexed by tableaux whose sets do not contain $n$ (and are therefore contained in $I$, since they must be contained in $I\cup\{n\}$ by Theorem \ref{main} and Proposition \ref{homrep}), and dividing by $e_{n}$ and $e_{n-k}$ respectively.

But Proposition \ref{embInd} shows that $R_{n+1,I\cup\{n\}}$ and $R_{n+1,I\cup\{n\}}^{\mathrm{hom}}$ can be given in terms of induced representations via Definition \ref{IndtVSVC}. Moreover, Remark \ref{IndExt} expresses these induced representations as $e_{n-k}\operatorname{Ext}_{S_{n}}^{S_{n+1}}R_{n,I}$ or $e_{n}\operatorname{Ext}_{S_{n}}^{S_{n+1}}R_{n,I}^{\mathrm{hom}}$ plus the images of the initial representations under a complementary operator, and the image of this complementary operator only involves, via Definition \ref{IndtVSVC} and Lemma \ref{addv}, precisely the representations whose associated tableaux has $n$ contained in their set, $\operatorname{Dsi}^{c}$ or $\operatorname{Dsp}^{c}$.

It follows from the first paragraph that for getting $R_{n+1,I}$ or $R_{n+1,I}^{\mathrm{hom}}$, we have to take $e_{n-k}\operatorname{Ext}_{S_{n}}^{S_{n+1}}R_{n,I}$ and divide by $e_{n-k}$ or consider $e_{n}\operatorname{Ext}_{S_{n}}^{S_{n+1}}R_{n,I}^{\mathrm{hom}}$ divided by $e_{n}$, yielding the desired equalities. This proves the proposition.
\end{proof}
Proposition \ref{embreps} is in correspondence with the fact from Remark \ref{IndExt} that $\operatorname{Ext}_{S_{n}}^{S_{n+1}}$ is homogeneous of degree 0.
\begin{ex}
In the case $n=3$ and $k=2$, for $I=\{1\}$ we get the representations $R_{3,I}=V^{123} \oplus V^{\substack{12 \\ 3\hphantom{4}}}=V^{000} \oplus V^{\substack{00 \\ 1\hphantom{1}}}$ and $R_{3,I}^{\mathrm{hom}}=V^{123}e_{1} \oplus V^{\substack{12 \\ 3\hphantom{4}}}=V^{000}e_{1} \oplus V^{\substack{00 \\ 1\hphantom{1}}}$. One easily verifies, via Definition \ref{IndtVSVC}, that applying $\operatorname{Ind}_{1,S_{3}}^{S_{4}}$ to the former yields the representation $R_{4,I\cup\{3\}}=W_{1,2,1}$ from Example \ref{n4kI}, while the image of the latter under $\operatorname{Ind}_{3,S_{3}}^{S_{4}}$ is the homogeneous version from Example \ref{exhom}. Moreover, letting $\operatorname{Ext}_{S_{3}}^{S_{4}}$ act on them produces the representations $R_{4,I}$ and $R_{4,I}^{\mathrm{hom}}$, as given in these examples (or in Remark \ref{deg1} and Example \ref{bijsmall}). When $I=\{2\}$ we have $R_{3,I}=V^{123}e_{1} \oplus V^{\substack{13 \\ 2\hphantom{4}}}=V^{000}e_{1} \oplus V^{\substack{01 \\ 1\hphantom{2}}}$ and $R_{3,I}^{\mathrm{hom}}=V^{123}e_{2} \oplus V^{\substack{13 \\ 2\hphantom{4}}}=V^{000}e_{2} \oplus V^{\substack{01 \\ 1\hphantom{2}}}$, with $\operatorname{Ind}_{1,S_{3}}^{S_{4}}$ taking the former to $R_{4,I\cup\{3\}}=W_{2,1,1}$, $\operatorname{Ind}_{3,S_{3}}^{S_{4}}$ sending the latter to its homogeneous counterpart, and $\operatorname{Ext}_{S_{3}}^{S_{4}}$ yielding the asserted images (note that unlike the previous case, here $\operatorname{Ext}_{S_{3}}^{S_{4}}$ takes the non-trivial component to the sum of two representations). \label{opersex}
\end{ex}

\medskip

The representations $R_{n,k}$ can thus be obtained inductively as follows.
\begin{cor}
For any $0 \leq k \leq n$, the representation $R_{n+1,k+1}$ is the direct sum of $\operatorname{Ind}_{n-k,S_{n}}^{S_{n+1}}R_{n,k}$ (which is 0 in case $k=0$) and $\operatorname{Ext}_{S_{n}}^{S_{n+1}}R_{n,k+1}$ (which vanishes when $k=n$). \label{Rnkind}
\end{cor}
In the extremal cases $k=0$ and $k=n$, Corollary \ref{Rnkind} reduces to the assertion that applying $\operatorname{Ext}_{S_{n}}^{S_{n+1}}$ to the trivial representation $R_{n,1}$ yields the trivial representation $R_{n+1,1}$, while $\operatorname{Ind}_{0,S_{n}}^{S_{n+1}}$ takes the regular representation $R_{n,n}$ of $S_{n}$ from \cite{[ATY]} to the regular one $R_{n+1,n+1}$ of $S_{n+1}$, as Example \ref{trivthm} describes.

\begin{proof}
Theorem \ref{main} expresses $R_{n+1,k+1}$ as the direct sum of $R_{n+1,\tilde{I}}$ over subsets $\tilde{I}\subseteq\mathbb{N}_{n}$ of size $k$. We decompose these sums into the one over sets $\tilde{I}$ containing $n$, and those which do not.

Now, Proposition \ref{embInd} shows that if $\tilde{I}$ contains $n$, namely $\tilde{I}=I\cup\{n\}$ for some $I\subseteq\mathbb{N}_{n-1}$ of size $k-1$, then the corresponding summand is $\operatorname{Ind}_{n-k,S_{n}}^{S_{n+1}}R_{n,I}$. By taking the direct sum over such $I$ we get $\operatorname{Ind}_{n-k,S_{n}}^{S_{n+1}}R_{n,k}$ by Theorem \ref{main} again, so that the sum over these sets $\tilde{I}$ yields the first desired summand. Moreover, since $n\in\tilde{I}$ implies that $k=|\tilde{I}|\geq1$, this part does not exist when $k=0$.

When $n\not\in\tilde{I}$ we get that $\tilde{I}=I\subseteq\mathbb{N}_{n-1}$, now of size $k$, and the associated summand is $\operatorname{Ext}_{S_{n}}^{S_{n+1}}R_{n,I}$ by Proposition \ref{embreps}. Applying Theorem \ref{main} once more, the direct sum of these elements is the second asserted summand. As the size of $\tilde{I}=I\subseteq\mathbb{N}_{n-1}$ is at most $n-1$, these parts do not show up in case $k=n$. This completes the proof of the corollary.
\end{proof}
Recall from Remark \ref{Rnkhom} that in \cite{[Z1]} we will show that the sum $R_{n,k}^{\mathrm{hom}}$ of the $R_{n,I}^{\mathrm{hom}}$ for sets $I$ of size $k-1$ is a direct sum. Once this is established, the argument proving Corollary \ref{Rnkind} expresses $R_{n+1,k+1}^{\mathrm{hom}}$ as the direct sum of $\operatorname{Ind}_{n,S_{n}}^{S_{n+1}}R_{n,k}^{\mathrm{hom}}$ and $\operatorname{Ext}_{S_{n}}^{S_{n+1}}R_{n,k+1}^{\mathrm{hom}}$, with the first summand vanishing if $k=0$ (where we get the same trivial setting), and the second one not showing up when $k=n$.

Note that Example \ref{opersex} presents the two parts of $R_{3,2}$ and $R_{3,2}^{\mathrm{hom}}$, with the trivial $R_{3,1}=R_{3,1}^{\mathrm{hom}}$, the regular $R_{3,3}$, and its homogeneous counterpart described in Example \ref{trivthm}. Using these calculations, we can verify Corollary \ref{Rnkind} for $n=3$ and all $0 \leq k\leq3$ explicitly (also in the homogeneous setting).

\medskip

The construction from Proposition \ref{embreps} stabilizes in the following sense.
\begin{lem}
Take $\lambda \vdash n$ with $\lambda_{1}>\lambda_{2}$, $S\in\operatorname{SYT}(\lambda)$ and $C\in\operatorname{CCT}(\lambda)$, and assume that the first $\lambda_{2}+1$ numbers appear in the first row of $S$ and the first $\lambda_{2}+1$ entries of $C$ vanish.
\begin{enumerate}[$(i)$]
\item We have $\operatorname{Ext}_{S_{n}}^{S_{n+1}}V^{S}_{\vec{h}}=V^{\tilde{\iota}S}_{\vec{h}}$ and $\operatorname{Ext}_{S_{n}}^{S_{n+1}}V_{C}^{\vec{h}}=V_{\hat{\iota}C}^{\vec{h}}$ for any vector $\vec{h}$.
\item To obtain generators for these representations, let $F_{w}\prod_{r=1}^{n}e_{r}^{h_{r}}$ be a generator of the original representation, for some $w \in S_{n}$ with $\tilde{Q}(w)=S$ or $\operatorname{ct}\big(\tilde{Q}(w)\big)=C$, divide by $\prod_{r=1}^{n}e_{r}^{h_{r}}$ in $\mathbb{Q}[\mathbf{x}_{n}]$, consider the number $m$ such that $v_{T}(m)=(1,\lambda_{2}+1)$ for $T:=P(w)$, add another variable $x_{n+1}$, symmetrize $F_{w}$ with respect to $x_{m}$ and $x_{n+1}$, and multiply by $\prod_{r=1}^{n}e_{r}^{h_{r}}$, now as an element of $\mathbb{Q}[\mathbf{x}_{n+1}]$.
\end{enumerate} \label{Extstab}
\end{lem}

\begin{proof}
We work with $S$, which we assume to be $\operatorname{ct}^{-1}(C)$ when the given is $C$ as always. When we apply the evacuation process from Section 3.9 of \cite{[Sa]} to $S$, the location of $n$ is determined by the first evacuation path. Our condition on $S$ (which follows from that on $C$ when required, via, e.g., Lemma \ref{ctJinv}) implies that in the first $\lambda_{2}$ steps, we evacuate along the first row, which means that it reaches the position in that row where there is no box below it. Hence the first evacuation path goes along the full first row, meaning that in $\operatorname{ev}S$ (which is then $\operatorname{ev}\operatorname{ct}^{-1}(C)$), the number $n$ shows up at the end of the first row.

But then any element $v\in\operatorname{EC}(\lambda)$, except for the one in the first row, lies below $n$ in $\operatorname{ev}S$ (or $\operatorname{ev}\operatorname{ct}^{-1}(C)$), so that Definition \ref{IndtVSVC} yields $\delta^{S,v}=0$ (as well as $\delta_{C,v}=0$) and thus $V^{S\tilde{+}v}_{\vec{h}}$ or $V_{C\hat{+}v}^{\vec{h}}$ does not show up in $\operatorname{Ext}_{S_{n}}^{S_{n+1}}V^{S}_{\vec{h}}$ or $\operatorname{Ext}_{S_{n}}^{S_{n+1}}V_{C}^{\vec{h}}$. It follows, as in Remark \ref{iotaaddv}, that only the representation $V^{\tilde{\iota}S}_{\vec{h}}$ or $V_{\hat{\iota}C}^{\vec{h}}$ appears in $\operatorname{Ext}_{S_{n}}^{S_{n+1}}V^{S}_{\vec{h}}$ or $\operatorname{Ext}_{S_{n}}^{S_{n+1}}V_{C}^{\vec{h}}$, yielding part $(i)$.

We now recall that for $w$ with $\tilde{Q}(w)=S$ or $\operatorname{ct}\big(\tilde{Q}(w)\big)=C$, and $P(w)=T$ for some $T\in\operatorname{SYT}(\lambda)$, the higher Specht polynomial $F_{w}$, which is $F_{T}^{S}$ or $F_{C,T}$ via Remark \ref{samepols}, lies in the irreducible representation $V^{S}$ or $V_{C}$ and hence generates it. The same therefore holds for $F_{w}\prod_{r=1}^{n}e_{r}^{h_{r}}$ with respect to $V^{S}_{\vec{h}}$ or $V_{C}^{\vec{h}}$, and the multiplier is recognizable in the product via Remark \ref{maxsymdiv}. As in the proof of Theorem \ref{stabSpecht}, the asserted process is the one producing $F_{w_{+}}$ from Proposition \ref{incnSpecht}, for $w_{+} \in S_{n+1}$ from Definition \ref{iotadef}. Since $\tilde{Q}(w_{+})=\tilde{\iota}S$ or $\operatorname{ct}\big(\tilde{Q}(w_{+})\big)=\hat{\iota}C$ (and $P(\tilde{w})=\iota T$), we deduce as before that $F_{w_{+}}$ is in $V^{\tilde{\iota}S}$ or $V_{\hat{\iota}C}$ and thus generates it, and thus after multiplication by the asserted element of $\mathbb{Q}[\mathbf{x}_{n+1}]$ we get a generator for $V^{\tilde{\iota}S}_{\vec{h}}$ or $V_{\hat{\iota}C}^{\vec{h}}$, as required for part $(ii)$. This completes the proof of the lemma.
\end{proof}
The condition from Lemma \ref{Extstab} is not a necessary one. For example, take $w=251364 \in S_{6}$ representing the ordered partition from Example \ref{barstarex}, with \[P(w)=\begin{ytableau} 1 & 3 & 4 \\ 2 & 5 & 6 \end{ytableau},\ Q(w)=\begin{ytableau} 1 & 2 & 5 \\ 3 & 4 & 6 \end{ytableau},\tilde{Q}(w)=\begin{ytableau} 1 & 3 & 4 \\ 2 & 5 & 6 \end{ytableau},C=\begin{ytableau} 0 & 1 & 1 \\ 1 & 2 & 2 \end{ytableau}\] where $C=\operatorname{ct}\big(\tilde{Q}(w)\big)$, the shape is $\lambda=33\vdash6$ with only two elements in $\operatorname{EC}(\lambda)$, and $\operatorname{Ext}_{S_{n}}^{S_{n+1}}V^{S}=V^{\tilde{\iota}S}$ (for $S=\tilde{Q}(w)$) and $\operatorname{Ext}_{S_{n}}^{S_{n+1}}V_{C}=V_{\hat{\iota}C}$. However, it is clear that the condition from that lemma does not hold for our $S$ and $C$.

\begin{rmk}
We note that one can strengthen Lemma \ref{Extstab} by obtaining the same result under the condition that $\lambda_{1}=\lambda_{2}$ and the first $\lambda_{2}$ numbers are in the first row of $S$ (or the first row of $C$ contains only zeros). We will not need, however, this stronger result for our applications here or in the sequels \cite{[Z1]} and \cite{[Z2]}, but will see it in Example \ref{stabex} below. \label{strong}
\end{rmk}

\begin{cor}
Let $i_{\max}$ be the maximal element of $I$, which we take to be 0 in case $I=\emptyset$, and assume that $n>2i_{\max}$. Then $R_{n+1,I}$ and $R_{n+1,I}^{\mathrm{hom}}$ are generated, as representations of $S_{n+1}$, by the following process. Consider a collection of generators of $R_{n,I}$ and $R_{n,I}^{\mathrm{hom}}$, divide each of them by its maximal symmetric divisor from $\mathbb{Q}[\mathbf{x}_{n}]$, add the variable $x_{n+1}$, symmetrize with respect to $x_{n+1}$ and a variable whose index lies in a column of length 1 in the tableau producing that generator, and multiply back by the symmetric function, but now from $\mathbb{Q}[\mathbf{x}_{n+1}]$. This produces generators for $R_{n+1,I}$ and $R_{n+1,I}^{\mathrm{hom}}$. \label{RnIstab}
\end{cor}
The maximal symmetric divisor in Corollary \ref{RnIstab} is defined via Remark \ref{maxsymdiv}. In fact, Remark \ref{strong} allows us to relax the assumption in that corollary to $n\geq2i_{\max}$, with the same proof.

\begin{proof}
For every $\lambda \vdash n$ and every $S\in\operatorname{SYT}(\lambda)$ or $C\in\operatorname{CCT}(\lambda)$ for which $I$ contains $\operatorname{Dsi}^{c}(S)$ or $\operatorname{Dsp}^{c}(C)$, the maximal number in the latter set is at most $i_{\max}$. It follows, via Definition \ref{Dsic}, that the minimal element of $\operatorname{Dsi}(S)$ (where we set $S:=\operatorname{ct}^{-1}(C)$ when we work with $C$) is at least $n-i_{\max}$, meaning that the first $n-i_{\max}$ numbers must lie in the first row of $S$ (and thus $C$ must contain at least $n-i_{\max}$ zeros, all of which are in the first row). Therefore $\lambda_{1} \geq n-i_{\max}$, and since $\lambda \vdash n$ this implies that $\lambda_{2} \leq n-\lambda_{1} \leq i_{\max}$. But the inequality $n>2i_{\max}$ implies that $n-i_{\max}>i_{\max}\geq\lambda_{2}$, and therefore the condition from Lemma \ref{Extstab} is satisfied for every such $S$ or $C$.

Now, Theorem \ref{main} and Proposition \ref{homrep} express $R_{n,I}$ and $R_{n,I}^{\mathrm{hom}}$ as direct sums of $V^{S}_{\vec{h}^{S}_{I}}$, $V_{C}^{\vec{h}_{C}^{I}}$, $V^{S}_{\vec{h}(I,S)}$, or $V_{C}^{\vec{h}(C,I)}$ over such tableaux $S$ or $C$, and we have $R_{n+1,I}=\operatorname{Ext}_{S_{n}}^{S_{n+1}}R_{n,I}$ and $R_{n+1,I}^{\mathrm{hom}}=\operatorname{Ext}_{S_{n}}^{S_{n+1}}R_{n,I}^{\mathrm{hom}}$ by Proposition \ref{embreps}. As Lemma \ref{Extstab} clearly implies the same result for $V^{S}_{\vec{h}}$ and $V_{C}^{\vec{h}}$ for any vector $\vec{h}$ (with the higher Specht polynomials multiplied by the corresponding symmetric functions), and for the direct sum of such representations (via Definition \ref{IndExt}), the desired assertion follows. This proves the corollary.
\end{proof}

\begin{ex}
Consider the two sets in Example \ref{opersex}, in which we had $n=3$ (and $k=2$). When $I=\{1\}$, with $i_{\max}=1$, the condition from Corollary \ref{RnIstab} is satisfied, and indeed both $R_{3,I}$ and its extension $R_{4,I}$, as well as the homogeneous versions $R_{3,I}^{\mathrm{hom}}$ and $R_{4,I}^{\mathrm{hom}}$, were direct sums of two representations, with $\operatorname{Ext}_{S_{3}}^{S_{4}}$ taking each irreducible representation to an irreducible one. However, we saw that for $I=\{2\}$, where $n=3<2i_{\max}=4$, this operator took a representation to the direct sum of two irreducible ones. As another example, for $n=4$, the representation $W_{1,1,2}$ given in Example \ref{n4kI} is $R_{4,I}$ for $I=\{1,2\}$, where $i_{\max}=2$ and the inequality from Corollary \ref{RnIstab} holds as a non-strict one. Proposition \ref{embreps} implies that its extension is $R_{5,I}=W_{1,1,3}$, which is given by \[V^{12345} \oplus V^{\substack{1234 \\ 5\hphantom{678}}} \oplus V^{\substack{1235 \\ 4\hphantom{678}}} \oplus V^{\substack{123 \\ 45\hphantom{6}}} \oplus V^{\substack{123 \\ 4\hphantom{67} \\ 5\hphantom{89}}}=V_{00000} \oplus V_{\substack{0000 \\ 1\hphantom{111}}} \oplus V_{\substack{0001 \\ 1\hphantom{112}}} \oplus V_{\substack{000 \\ 11\hphantom{1}}} \oplus V_{\substack{000 \\ 1\hphantom{11} \\ 2\hphantom{22}}},\] with the same number of irreducible components, each of which is $V^{\tilde{\iota}S}_{\vec{h}}$ or $V_{\hat{\iota}C}^{\vec{h}}$ for some $V^{S}_{\vec{h}}$ or $V_{C}^{\vec{h}}$ showing up in $W_{1,1,2}$ and satisfying the condition from Lemma \ref{Extstab}, or its weakened version from Remark \ref{strong} (a similar result holds, of course, for $R_{4,I}^{\mathrm{hom}}$ and $R_{5,I}^{\mathrm{hom}}$). \label{stabex}
\end{ex}
As Theorem \ref{main} shows, the representations $W_{1,3,1}$ and $W_{3,1,1}$ are isomorphic to $W_{1,1,3}$ showing up in Example \ref{stabex}. But as they are associated with the subsets $\{1,4\}$ and $\{3,4\}$ of $\mathbb{N}_{4}$, containing $n=4$, they are not obtained by Proposition \ref{embreps}, but rather through induction, via Proposition \ref{embInd}, from $R_{4,I}$ for a singleton $I$ (and the same for their homogeneous counterparts).

\medskip

\begin{rmk}
The map taking $F_{w}\in\mathbb{Q}[\mathbf{x}_{n}]$ for $w \in S_{n}$ to $F_{w_{+}}\in\mathbb{Q}[\mathbf{x}_{n+1}]$ as in Proposition \ref{incnSpecht}, or equivalently $F_{T}^{S}$ to $F_{\iota T}^{\tilde{\iota}S}$ when $T$ and $S$ are standard of the same shape, or $F_{C,T}$ to $F_{\hat{\iota}C,\iota T}$, embeds $V^{S}$ (or $V_{C}$) into $V^{\tilde{\iota}S}$ (or $V_{\hat{\iota}C}$) in a natural $S_{n}$-equivariant manner, with the image generating the larger representation over $\mathbb{Q}[S_{n+1}]$. The same applies for $V^{S}_{\vec{h}}$ and $V_{C}^{\vec{h}}$, and therefore produces natural embeddings of $R_{n,I}$ and $R_{n,I}^{\mathrm{hom}}$ into $R_{n+1,I}$ and $R_{n+1,I}^{\mathrm{hom}}$ respectively, with the image generate the parts of the latter that is based on indices that are images of $\tilde{\iota}$ or $\hat{\iota}$. Corollary \ref{RnIstab} in fact means that $\{R_{n,I}\}_{n>i_{\max}}$ and $\{R_{n,I}^{\mathrm{hom}}\}_{n>i_{\max}}$ form stable families of representations in the sense of \cite{[CF]}, \cite{[CEF]}, and \cite{[Fa]} (and also \cite{[SS]}), which are also centrally stable in the sense of \cite{[Pu]}, with the stability showing up for $n>2i_{\max}$ (or even $n\geq2i_{\max}$, via Remark \ref{strong}). \label{repstab}
\end{rmk}
Note that $R_{n,I}$ and $R_{n,I}^{\mathrm{hom}}$ are only defined for $n>i_{\max}$, since we assume $I\subseteq\mathbb{N}_{n-1}$. One can define them for smaller $n$ either by intersecting $I$ with $\mathbb{N}_{n-1}$, or by observing that the condition $\operatorname{Dsi}^{c}(S) \subseteq I$ or $\operatorname{Dsp}^{c}(C) \subseteq I$ still makes sense, and using the same vectors $\vec{h}^{S}_{I}$, $\vec{h}_{C}^{I}$, $\vec{h}(I,S)$, and $\vec{h}(C,I)$ for the definition, with the usual convention that $e_{i}$ vanishes wherever $i>n$. Both extensions will then produce families $\{R_{n,I}\}_{n=1}^{\infty}$ and $\{R_{n,I}^{\mathrm{hom}}\}_{n=1}^{\infty}$ that exhibit the representation stability from Remark \ref{repstab}.

Remark \ref{repstab} is, in some sense, a complement to Theorem 7.1 of \cite{[CF]} (under the isomorphism of the cohomology ring there, which is the one from \cite{[B]}, with $R_{n}=R_{n,n}$ as a graded ring).

We conclude by commenting on the action of the other insertions.
\begin{rmk}
Lemma \ref{operslast} shows that the star and bar operations at the last location behave the best in terms of the relation with Lemma \ref{Snparts}. A more general bar insertion, say in location $0 \leq p<k$, will add $i_{p}+1$ to $I$ (which is 1 in case $p=0$) via this lemma, and replace every $i_{h}$ with $h>p$ by $i_{h}+1$, and for the star insertion there again each entry $i_{h}$ with $h>p$ will be replaced by $i_{h}+1$. One can formulate an induction rule for the former, and some expression for the latter, both more complicated than Definition \ref{IndtVSVC} (they will also be related via part $(ii)$ of Proposition \ref{incI} with $\ell=i_{p}+1$, which is no longer as neat as part $(iii)$ there). Since with our description using Lemma \ref{operslast}, our conventions are much more natural, we content ourselves with the results that we proved. \label{otherlocs}
\end{rmk}

\noindent\textsc{Einstein Institute of Mathematics, the Hebrew University of Jerusalem, Edmund Safra Campus, Jerusalem 91904, Israel}

\noindent E-mail address: zemels@math.huji.ac.il

\end{document}